\documentclass[11pt]{amsart}
\usepackage{mathtools, amssymb,amsfonts,amsthm,amscd,amsmath,amsgen,upref}

\usepackage[margin=1in]{geometry}

\theoremstyle{plain}
\newtheorem{cor}{Corollary}
\newtheorem{lem}[cor]{Lemma}
\newtheorem{prop}[cor]{Proposition}

\newtheorem{thm}[cor]{Theorem}
\theoremstyle{definition}
\newtheorem{definition}[cor]{Definition}
\newtheorem{remark}[cor]{Remark}

\numberwithin{cor}{section}
\numberwithin{equation}{section}

\DeclareMathOperator{\C}{C}

\DeclareMathOperator{\BUC}{BUC}
\DeclareMathOperator{\BV}{BV}
\DeclareMathOperator{\Lip}{Lip}

\DeclareMathOperator{\sgn}{sgn}

\newcommand{\abs}[1]{\left|#1\right|}
\newcommand{\norm}[1]{\left\|#1\right\|}
\providecommand{\ud}[1]{\, \mathrm{d} #1}
\providecommand{\dx}{\ud{x}}
\providecommand{\dy}{\ud{y}}
\providecommand{\dxi}{\ud \xi}
\providecommand{\deta}{\ud{\eta}}
\providecommand{\dr}{\ud{r}}

\providecommand{\dxp}{\ud{x'}}
\providecommand{\dxip}{\ud{\xi'}}
\providecommand{\dyp}{\ud{y'}}
\providecommand{\ds}{\ud{s}}
\providecommand{\dt}{\ud{t}}
\providecommand{\dz}{\ud{z}}
\providecommand{\dd}{\ud}

\def\XXint#1#2#3{{\setbox0=\hbox{$#1{#2#3}{\int}$ }
\vcenter{\hbox{$#2#3$ }}\kern-.6\wd0}}

\usepackage{graphicx}
\usepackage{color}

\title{Well-posedness of nonlinear diffusion equations with nonlinear, conservative noise}

\author{Benjamin Fehrman, Benjamin Gess}

\address{University of Oxford, Mathematical Institute, Oxford OX2 6GG, United Kingdom  \newline Max Planck Institute for Mathematics in the Sciences \\ Inselstr. 22\\ 04103 Leipzig, Germany \newline Fakult\"at f\"ur Mathematik \\ Universit\"at Bielefeld \\ Universit\"atstr. 25 \\ 33615 Bielefeld \\ Germany }
\email{Benjamin.Fehrman@maths.ox.ac.uk \\ Benjamin.Gess@mis.mpg.de}

\date{\today}

\begin{document}

\begin{abstract}
We prove the pathwise well-posedness of stochastic porous media and fast diffusion equations driven by nonlinear, conservative noise. As a consequence, the generation of a random dynamical system is obtained. This extends results of the second author and Souganidis, who considered analogous spatially homogeneous and first-order equations, and earlier works of Lions, Perthame, and Souganidis. 
\end{abstract}

\maketitle

\section{Introduction}\label{intro}

In this paper, we consider stochastic porous media and fast diffusion equations with nonlinear, conservative noise of the form
\begin{equation}\label{intro_eq} \left\{\begin{array}{ll} \partial_t u=\Delta (|u|^{m-1}u)+\nabla\cdot (A(x,u)\circ \dz_t)& \textrm{on}\;\;\mathbb{T}^d\times(0,\infty), \\ u=u_0 & \textrm{on}\;\;\mathbb{T}^d\times\{0\},\end{array}\right.\end{equation}
for a diffusion exponent $m\in(0,\infty)$, nonnegative initial data $u_0\in L^2(\mathbb{T}^d)$, and an $n$-dimensional, $\alpha$-H\"older continuous, geometric rough path $z$, which in particular applies to the case when $z$ is an $n$-dimensional Brownian motion.  The domain $\mathbb{T}^d$ is the $d$-dimensional unit torus.  The matrix-valued nonlinearity
$$A(x,\xi)=\left(a_{ij}(x,\xi)\right):\mathbb{T}^d\times\mathbb{R}\rightarrow\mathcal{M}^{d\times n},$$
is assumed to be regular, with required regularity dictated by regularity of the rough path $z$.  

This type of stochastic porous media equation arises, for example, as an approximative model for the fluctuating hydrodynamics of the zero range particle process about its hydrodynamic limit, as a continuum limit of mean field stochastic differential equations with common noise, with notable relation to the theory of mean field games, as an approximation to the Dean-Kawasaki equation arising in fluctuating fluid dynamics, and as a model for thin films of Newtonian fluids with negligible surface tension.  More details on these applications are given in Section~\ref{sec:appl} below. 

The methods of this paper prove that equation \eqref{intro_eq} is pathwise well-posed using primarily analytic techniques and rough path analysis. It should be noted that even in the case where $z$ is given by a Brownian motion and even in the probabilistic (i.e.\ non-pathwise) sense, the well-posedness of \eqref{intro_eq} could not be shown thus far.  In addition, the results of this paper establish the existence of a random dynamical system for \eqref{intro_eq}, which is known to be a notoriously difficult problem for stochastic partial differential equations with nonlinear noise and which is, in general, largely open.  These are the first results proving the existence of a random dynamical system for a nonlinear SPDE with $x$-dependent, nonlinear noise.  Even in the linear case $m=1$, and despite much effort \cite{Flandoli,Gess1,MohammedZhangZhao}, this could not be shown previously.

The nonlinearity of the stochastic term prevents the application of transformation methods that are often used for equations driven by affine-linear noise. Instead, our method is based on passing to the equation's kinetic formulation, introduced by Chen and Perthame \cite{ChenPerthame}. 
Motivated by the theory of stochastic viscosity solutions for fully-nonlinear second-order stochastic partial differential equations of Lions and Souganidis \cite{LSstoch5,LSstoch4,LSstoch3,LSstoch2,LSstoch1}, and the work of Lions, Perthame and Souganidis \cite{LPS1,LPS} and the second author and Souganidis \cite{GessSouganidis,GessSouganidis1,GessSouganidis2} on stochastic scalar conservation laws, this gives rise to the notion of a \emph{pathwise kinetic solution} (cf.\ 
Definition~\ref{def_solution} below).

The methods developed in \cite{GessSouganidis} for scalar conservation laws with $x$-dependent flux rely on weak convergence arguments and so-called generalized kinetic solutions. These kind of arguments do not apply to the parabolic-hyperbolic case \eqref{intro_eq}, since the class of pathwise entropy solutions to \eqref{intro_eq} is not closed under weak convergence. For this reason, in \cite{GessSouganidis1} a strong convergence method, based on a uniform $BV$-estimate and continuous dependence on the driving signal $z$ with respect to the uniform topology was introduced. These arguments are strictly restricted to $x$-independent noise. Indeed, neither a uniform $BV$-estimate for solutions to \eqref{intro_eq} seems to be available, nor, as the theory of rough paths tells us, should the continuity of solutions with respect to $z$ in uniform topology be expected.

As a consequence, new arguments have to be introduced in order to handle \eqref{intro_eq}. In this spirit, the proof of uniqueness of solutions to \eqref{intro_eq} heavily relies on the observation of new cancellations and error estimates.  The proof furthermore uses sharp regularity estimates which, in the fast diffusion case $m\in(0,1)$, are new even in the deterministic setting.  As a first main result, in Section~\ref{uniqueness}, we obtain the uniqueness of pathwise kinetic solutions with nonnegative initial data.

\begin{thm}\label{intro_unique}  Let $u^1_0,u^2_0\in L^2_+(\mathbb{T}^d)$.  Pathwise kinetic solutions $u^1$ and $u^2$ of \eqref{intro_eq} with initial data $u^1_0$ and $u^2_0$ satisfy
$$\norm{u^1-u^2}_{L^\infty\left([0,\infty);L^1(\mathbb{T}^d)\right)}\leq \norm{u^1_0-u^2_0}_{L^1(\mathbb{T}^d)}.$$
In particular, pathwise kinetic solutions are unique.  \end{thm}

As pointed out above, compactness arguments used in the spatially homogeneous setting are not available for \eqref{intro_eq}. Instead, the proof of existence introduced in this work relies on new a priori estimates both in space and time. In Section~\ref{sol_exists}, we prove existence for general initial data.

\begin{thm}\label{intro_existence}  Let $u_0\in L^2(\mathbb{T}^d)$.  There exists a pathwise kinetic solution $u$ of \eqref{intro_eq} with initial data $u_0$.  Furthermore, if $u_0\in L^2_+(\mathbb{T}^d)$, then, for each $T>0$,
$$u\in L^\infty([0,T];L^2_+(\mathbb{T}^d)).$$
\end{thm}

{\color{black}It is well known, see for instance Lyons \cite{Lyons91}, that solutions to stochastic differential equations do not depend continuously on the driving noise.  However, in \cite{Lyons98} Lyons observed that continuity of the solution map can be recovered by means of a finer \emph{rough path} topology.  These ideas are recalled in Section~\ref{sec_rough}.

We prove an analogous result for pathwise kinetic solutions.  Namely, as a consequence of the analysis leading to Theorem~\ref{intro_unique} and Theorem~\ref{intro_existence}, we prove that solutions of \eqref{intro_eq} depend continuously on the driving noise.  In the following statement, the metric $d_\alpha$ denotes the $\alpha$-H\"older metric on the space of geometric rough paths introduced in Section~\ref{sec_rough}.  Since the solution map is a map between metric spaces, continuity is phrased in terms of sequential continuity.

\begin{thm}\label{intro_noise_cts}  Let $u_0\in L^2_+(\mathbb{T}^d)$ and $T>0$.  Let $\{z^n\}_{n=1}^\infty$ and $z$ be a sequence of $n$-dimensional, $\alpha$-H\"older continuous geometric rough paths on $[0,T]$ satisfying
$$\lim_{n\rightarrow\infty}d_\alpha(z^n,z)=0.$$
Let $\{u^n\}_{n=1}^\infty$ and $u$ denote the pathwise kinetic solutions to \eqref{intro_eq} on $[0,T]$ with initial data $u_0$ and driving signals $\{z^n\}_{n=1}^\infty$ and $z$ respectively.   Then,
$$\lim_{n\rightarrow\infty}\norm{u^n-u}_{L^\infty([0,T];L^1(\mathbb{T}^d))}=0.$$

\end{thm}}

Furthermore, the existence of a random dynamical system for \eqref{intro_eq} is immediate from Theorem \ref{intro_unique} and Theorem \ref{intro_existence}.   A more complete discussion concerning random dynamical systems in general can be found in the work of Flandoli \cite{Flandoli}, the second author \cite{Gess1}, and Mohammed, Zhang, and Zhao \cite{MohammedZhangZhao}.  In the context of this paper, the existence of a random dynamical system amounts to proving an almost-sure inhomogeneous semigroup property for the equation.

Precisely, suppose that $t\in[0,\infty)\mapsto z_t=z_t(\omega)$ arises from the sample paths of a stochastic process defined on a probability space $\omega\in(\Omega,\mathcal{F},\mathbb{P})$.  Let $u(u_0,s,t;z_\cdot(\omega))$ denote the solution of \eqref{intro_eq} at time $t\geq s$, beginning from time $s\geq0$ with noise $z_\cdot(\omega)$ and initial data $u_0$.  To prove the existence of a random dynamical system, it is necessary to show that, for every $u_0\in L^2_+(\mathbb{T}^d)$, for almost every $\omega\in\Omega$,
\begin{equation}\label{intro_RDS} u(u_0,s,t;z_\cdot(\omega))=u(u_0,0,t-s;z_{\cdot+s}(\omega))\;\;\textrm{for every}\;\;0\leq s\leq t<\infty.\end{equation}
The pathwise results of Theorem~\ref{intro_unique} and Theorem~\ref{intro_existence} immediately imply \eqref{intro_RDS}, since there is precisely one zero set for all times.  For simplicity, the statement is specialized to the case of fractional Brownian motion.

\begin{thm}\label{intro_rds}  Suppose that the noise $t\in[0,\infty)\mapsto z_t(\omega)$ arises from the sample paths of a fractional Brownian motion with Hurst parameter $H\in(\frac{1}{4},1)$ defined on a probability space $\omega\in(\Omega,\mathcal{F},\mathbb{P})$.  Equation~\eqref{intro_eq} interpreted in the sense of Definition~\ref{def_solution} defines a random dynamical system on $L^2_+(\mathbb{T}^d)$. \end{thm}

We remark that the methods of this paper apply to general initial data in $L^2(\mathbb{T}^d)$ provided the diffusion exponent satisfies $m=1$ or $m>2$.

\begin{thm}\label{intro_sign}  Suppose that $m=1$ or $m>2$.  For every $u_0\in L^2(\mathbb{T}^d)$, there exists a unique pathwise kinetic solution of \eqref{intro_eq} and the analogous conclusions of Theorems~\ref{intro_unique}, \ref{intro_existence}, \ref{intro_noise_cts}, and \ref{intro_rds} are satisfied.  \end{thm}

Finally, the methods of this paper also apply to equations set on the whole space, provided the diffusion coefficient satisfies $m=1$ or $m\geq 3$, and the details can be found in the first version of this paper \cite{FehrmanGess1}.

\begin{thm}\label{intro_space}  Suppose that $m=1$ or $m\geq 3$.  For every $u_0\in \left(L^1\cap L^2\right)(\mathbb{R}^d)$, there exists a unique pathwise kinetic solution of \eqref{intro_eq} and the analogous conclusions of Theorems~\ref{intro_unique}, \ref{intro_existence},  \ref{intro_noise_cts}, and \ref{intro_rds} are satisfied.  \end{thm}

\begin{remark} The $L^2$ integrability of the initial data is assumed for simplicity only. At the cost of additional technicalities, the results of this paper can be extended to nonnegative initial data in $L^1_+(\mathbb{T}^d)$.  This requires, in particular, a modification to the definition of a pathwise kinetic solution, since the \emph{entropy} and \emph{parabolic defect measures} will no longer be globally integrable (cf.\ Definition~\ref{def_solution} below).  The proof of uniqueness and the stable estimates would also need to be localized in order to account for the lack of integrability. \end{remark}

\subsection{Applications}\label{sec:appl}

Equations of the form \eqref{intro_eq} arise in several applications.  It was shown by Ferrari, Presutti, and Vares \cite{FerrariPresuttiVares1} that the hydrodynamic limit of a zero range particle process satisfies a nonlinear diffusion equation of the type
\begin{equation}\label{eq_Dirr}\partial_tu=\Delta \Phi(u) \;\;\textrm{in}\;\;\mathbb{T}^d\times(0,\infty),\end{equation}
where $\Phi$ is the mean local jump rate.  For instance, in the porous media case $\Phi(\rho)=\rho\abs{\rho}^{m-1}$, this means that the process exhibits a high rate of diffusion in regions of high concentration.

The fluctuating hydrodynamics of the zero range process about its hydrodynamic limit were subsequently studied by Ferrari, Presutti, and Vares \cite{FerrariPresuttiVares}, and were informally shown by Dirr, Stamatakis, and Zimmer \cite{DirrStamatakisZimmer} to satisfy a stochastic nonlinear diffusion equation of the type
\begin{equation}\label{eq_Dirr_2}\partial_t u=\Delta\Phi(u)+\nabla\cdot\left(\sqrt{\Phi(u)}\mathcal{N}\right)\;\;\textrm{in}\;\;\mathbb{T}^d\times(0,\infty),\end{equation}
where $\mathcal{N}$ is a space-time white noise.   Equation \eqref{intro_eq} represents a regularization of \eqref{eq_Dirr_2} for $\Phi(\rho)=\rho\abs{\rho}^{m-1}$ given by a smoothing of the square root function and a regularization of the noise in space.

For a second example, consider an {\color{black} $(L+1)$-dimensional} system of mean field stochastic differential equations, for $i\in\{0,\ldots,L\}$,
\begin{equation}\label{mean_field} \dd X^i_t=A^L(X^i_t,\frac{1}{L}\sum_{j\neq i}\delta_{X^j_t})\circ \dd B_t+\Sigma^L(\frac{1}{L}\sum_{j\neq i}\delta_{X^j_t})\dd W^i_t\;\;\textrm{for}\;\;t\in(0,\infty),\end{equation}
where $L\geq 1$ and {\color{black}$\{B_t^i\}_{i=1}^d$ and $\{W_t^i\}_{i=1}^n$} are independent Brownian motions.  The first term is interpreted in the Stratonovich sense and the second term is interpreted in the It\^o sense.  For each $L\geq 1$, the nonlinearities $A^L:\mathbb{T}^d\times\mathcal{P}(\mathbb{T}^d)\rightarrow\mathcal{M}^{d\times n}$ and $\Sigma^L:\mathcal{P}(\mathbb{T}^d)\rightarrow\mathbb{R}$ are assumed to be continuous with respect to the topology of weak convergence on the space of probability measures.

It follows informally from the theory of mean field games, as introduced by Lasry and Lions \cite{LasryLions1,LasryLions2,LasryLions}, that the {\color{black} conditional density $m$ of the empirical law of the solution {\color{black}$X_t=(X^0_t,\ldots,X^L_t)$} with respect to $B_t$}, in the mean field limit $L\rightarrow\infty$, evolves according to an equation of the form
\begin{equation}\label{mean_field_1} \partial_tm =\frac{1}{2}\Delta\left( \sigma^2(m)m\right) +\nabla\cdot(A(x,m)m\circ \dd B_t)\;\; \textrm{in}\;\;\mathbb{T}^d\times(0,\infty),\end{equation}
provided the nonlocal nonlinearities $\{A^L\}_{\{L\geq 1\}}$ and $\{\Sigma^L\}_{\{L\geq 1\}}$ satisfy appropriate assumptions which guarantee that, as $L\rightarrow\infty$, they converge to local functions $A:\mathbb{T}^d\times\mathbb{R}\rightarrow\mathcal{M}^{n\times d}$ and $\sigma:\mathbb{R}\rightarrow\mathbb{R}$ of the density.

A third application of equations of the type \eqref{intro_eq}, for $m=1$, is given as an approximation to the Dean-Kawasaki model for the diffusion of particles subject to thermal advection in a fluctuating fluid.  In this model, proposed by Dean \cite{Dean}, Kawasaki \cite{Kawasaki}, and Marconi and Tarazona \cite{MarconiTarazona}, and recently studied by Donev, Fai and Vanden-Eijnden \cite{DonevFaiVanden}, the density of the particles $c$ evolves according to the stochastic equation
\begin{equation}\label{particle_eq}\partial_t c= \sigma\Delta c+\nabla\cdot\left(cv+\sqrt{2\sigma c}\mathcal{N}\right)\;\;\textrm{in}\;\;\mathbb{T}^d\times(0,\infty),\end{equation}
where $\sigma>0$ is a diffusion coefficient, $v$ is a smooth and divergence free velocity field, and $\mathcal{N}$ is a space-time white noise.  Equation \eqref{intro_eq}, for $m=1$, therefore represents a regularized version of \eqref{particle_eq}, which is obtained by smoothing the square root function and considering noise that is regular in space and driven by a rough path in time.

An additional application arises as a stochastic model for the evolution of a thin film consisting of an incompressible Newtonian liquid on a flat $d$-dimensional substrate proposed by Gr\"un, Mecke, and Rauscher \cite{GrunMeckeRauscher}.  Their model describes the evolution of the thickness $h$ of the substrate, which is the solution of the stochastic partial differential equation
\begin{equation}\label{thin_film} \partial_t h=\nabla\cdot\left(h^n\nabla\left(\frac{1}{3}\Phi'(h)-\gamma\Delta h\right)\right)+\nabla\cdot\left(\frac{h^{\frac{3}{2}}}{3}\mathcal{N}\right)\;\; \textrm{in}\;\;\mathbb{T}^d\times(0,\infty),\end{equation}
where $\Phi$ is the effective interface potential describing the interaction of the liquid and the substrate, $\gamma>0$ is the surface tension coefficient, $\mathcal{N}$ is a space-time white noise, and $n>0$ describes the mobility function depending on the flow condition at the liquid-solid interface.  In \cite{GrunMeckeRauscher}, a no-slip boundary condition is assumed, which corresponds to $n=3$.  Equation \eqref{intro_eq} can be viewed as a simplified model of equation \eqref{thin_film} in the case that the effective interface potential $\Phi(\xi)\simeq\abs{\xi}^{s}$ for small values $\xi\in\mathbb{R}$ and for some $s\geq 1-n$, and in the case that the surface tension $\gamma\simeq 0$ is negligible.

\subsection{Relation to previous work and methodology}  The methods of this paper build upon the theory of stochastic viscosity solutions for fully-nonlinear second-order stochastic partial differential equations introduced by \cite{LSstoch5,LSstoch4,LSstoch3,LSstoch2,LSstoch1}, and the work \cite{LPS1,LPS} and \cite{GessSouganidis,GessSouganidis1,GessSouganidis2} on scalar conservation laws driven by multiple rough fluxes. As laid out above, the application of these ideas is, however, complicated by the nonlinear structure of the noise.

Motivated by the methods of \cite{GessSouganidis,GessSouganidis1}, we first pass to the kinetic formulation of \eqref{intro_eq} introduced by \cite{ChenPerthame} and Perthame \cite{Perthame}.  The precise details can be found in Section~\ref{kinetic}.  This yields an equation in $(d+1)$-variables for which the noise enters as a linear transport.  The transport is well-defined for rough driving signals, as shown in Lyons and Qian \cite{LyonsQian}, when interpreting the underlying system as a rough differential equation.  The details are presented in Section~\ref{characteristics}, where Definition~\ref{def_solution} presents the notion of a \emph{pathwise kinetic solution}.

The definition is formally obtained by flowing the corresponding kinetic solution along the system of rough characteristics, which are defined globally in time.  This is effectively achieved by considering a class of test functions which are  transported by the corresponding system of inverse characteristics.  In this regard, our setting resembles more closely \cite{LPS1,LPS} and \cite{GessSouganidis,GessSouganidis1} and is simpler than the general stochastic viscosity theory  \cite{LSstoch5,LSstoch4,LSstoch3,LSstoch2,LSstoch1}.  There, the noise is removed by flowing test functions along a system of stochastic characteristics arising from a stochastic Hamilton-Jacobi equation, which are defined only locally in time and are therefore less easily inverted.

With regard to the stochastic term, in comparison to \cite{LPS1,LPS}, the noise is multi-dimensional, if $n>1$, and spatially inhomogeneous---that is, $x$-dependent.  Therefore, the characteristic equations cannot be solved explicitly and it is therefore necessary to use rough path estimates from Section~\ref{sec_rough} in order to understand the cancellations.  Furthermore, these cancellations depend crucially on the conservative structure of the equation, which implies, in particular, that the stochastic characteristics preserve the underlying Lebesgue measure.

The interaction between the $x$-dependent characteristics and nonlinear diffusion term significantly complicates the proof of uniqueness.  This is evidenced by our need to use Proposition~\ref{aux_log} to handle the case of small diffusion exponents, an argument which has no analogue in the deterministic or stochastic settings.  The estimate of Proposition~\ref{aux_log} is simply false, in general, for signed initial data and is, in some sense, an optimal regularity statement encoded by a finite singular moment of the solution's \emph{parabolic defect measure} (cf. Definition~\ref{def_solution}).

The proof of existence for second-order equations is also significantly more involved than in the first-order case.  This is due to the aforementioned fact that the space of pathwise kinetic solutions is not closed with respect to weak convergence.   We therefore prove the existence of solutions by proving the strong convergence of the kinetic solutions corresponding to a sequence of regularized equations in Section~\ref{sol_exists}.  In particular, we prove a stable estimate for the kinetic functions in the fractional Sobolev space $W^{s,1}$, for any $s\in(0,\frac{2}{m+1}\wedge 1)$ (cf. Proposition~\ref{fractional_sobolev}).  This regularity is based upon Proposition~\ref{stable_L1} and Proposition~\ref{stable_estimates}, which prove that, locally in time, pathwise kinetic solutions preserve the basic regularity of solutions to the deterministic porous medium equation.

In combination, Theorem \ref{intro_unique} and Theorem \ref{intro_existence} prove the pathwise well-posedness of equation \eqref{intro_eq} for every initial data $u_0\in L^2_+(\mathbb{T}^d)$, and for every diffusion exponent $m\in(0,\infty)$.  We remark that these results also incorporate the notion of renormalized solutions, as originally introduced by DiPerna and Lions \cite{DiPernaLions} in the context of the Boltzmann equation and subsequently used in the context of nonlinear parabolic problems by Blanchard and Murat \cite{BlanchardMurat} and Blanchard and Redwane \cite{BlanchardRedwane1,BlanchardRedwane}.  This is due to the fact that we do not, in general, require the integrability of the signed power of the initial data $\abs{u_0}^{m-1}u_0$.

Finally, we remark that while probabilistic and pathwise techniques have not been successful in treating \eqref{intro_eq}, they have previously been used to prove the well-posedness of stochastic porous medium equations in the simpler cases of additive or multiplicative noise.  This includes, for instance, the work of Barbu, Bogachev, Da Prato, and R\"ockner \cite{BarbuBogachevDaPratoRockner}, Barbu, Da Prato, and R\"ockner \cite{BarbuDaPratoRoeckner,BarbuDaPratoRoeckner1,BarbuDaPratoRoeckner2,BarbuDaPratoRoeckner3}, Barbu and R\"ockner \cite{BarbuRoeckner}, Barbu, R\"ockner, and Russo \cite{BarbuRoecknerRusso}, Da Prato and R\"ockner \cite{DaPratoRoeckner}, Da Prato, R\"ockner, Rozovski\u\i, and Wang \cite{DaPratoRoecknerRozovskiiWang}, the second author \cite{Gess}, Kim \cite{Kim}, Krylov and Rozovski\u\i\ \cite{KrylovRozovskii1,KrylovRozovskii}, Pardoux \cite{Pardoux}, Pr\'ev\^ot and R\"ockner \cite{PrevotRoeckner}, Ren, R\"ockner, and Wang \cite{RenRoecknerWang}, R\"ockner and Wang \cite{RoecknerWang}, and Rozovski\u\i\ \cite{Rozovskii}.

\subsection{Structure of the paper}  The paper is organized as follows.  In Section~\ref{preliminaries}, we present our assumptions.  In Section~\ref{characteristics}, we analyze the associated system of stochastic characteristics and present the definition of a pathwise kinetic solution.  The proof of uniqueness appears in Section~\ref{uniqueness} and the proof of existence appears in Section~\ref{sol_exists}.  The remainder of the paper consists of an appendix.  In Section~\ref{kinetic}, we prove the existence of kinetic solutions to a regularization of equation \eqref{intro_eq}.  In Section~\ref{sec_rough}, we present some stability results from the theory of rough paths.  Finally, in Section~\ref{sec_frac}, we prove some basic properties of fractional Sobolev spaces and establish the regularity of pathwise kinetic solutions on the level of their kinetic functions.

\section{Preliminaries}\label{preliminaries}

\subsection{Assumptions}

The spatial dimension is one or greater:
\begin{equation}\label{prelim_dimension} d\geq 1.\end{equation}
The diffusion exponent is $m\in(0,\infty)$, and the signed power
$$u^{[m]}:=\abs{u}^{m-1}u.$$
{\color{black}The noise is a geometric rough path:  for $n\geq 1$ and a H\"older exponent $\alpha\in(0,1)$, for each $T>0$,
\begin{equation}\label{prelim_Holder} z_t=(z^1_t,\ldots,z^n_t)\in \C^{0,\alpha}\left([0,T];G^{\left\lfloor \frac{1}{\alpha}\right\rfloor}(\mathbb{R}^n)\right),\end{equation}
where  $\C^{0,\alpha}([0,T];G^{\left\lfloor \frac{1}{\alpha}\right\rfloor}(\mathbb{R}^n))$ is the space of $n$-dimensional, $\alpha$-H\"older continuous geometric rough paths on $[0,T]$.  See Section~\ref{sec_rough} for a brief introduction to and references on rough path theory.}

The coefficients have derivatives which are smooth and bounded:  for $\gamma>\frac{1}{\alpha}$, for each $i\in\{1,\ldots,d\}$ and $j\in\{1,\ldots,n\}$,
\begin{equation}\label{prelim_regular} \nabla_xa_{ij}(x,\xi)\in \C^{\gamma+2}(\mathbb{T}^d\times\mathbb{R};\mathbb{R}^d)\;\;\textrm{and}\;\;\partial_\xi a_{ij}(x,\xi)\in \C^{\gamma+2}(\mathbb{T}^d\times\mathbb{R}).\end{equation}
This regularity is necessary in order to obtain the rough path estimates of Proposition~\ref{rough_est}.  In particular, as the regularity of the noise decreases, more regularity is required from the coefficients.

Finally, the nonlinearity $A(x,\xi)$ satisfies:
\begin{equation}\label{prelim_vanish}\sum_{i=1}^d\partial_{x_i}a_{ij}(x,0)=0\;\;\textrm{for each}\;\;x\in\mathbb{T}^{d}\;\;\textrm{and}\;\;j\in\{1,\ldots,n\}.\end{equation}
This assumption guarantees that the underlying stochastic characteristics preserve the sign of the velocity variable.  Even in the case of smooth driving signals, this condition is necessary to ensure that the evolution of \eqref{intro_eq} does not increase the mass of the initial condition.

Finally, for every $p\in[0,\infty]$, the space $L^p_+(\mathbb{T}^d)$ denotes the the space of nonnegative $L^p$-functions on the torus.  That is, $L^p_+(\mathbb{T}^d)$ is the closure of the space of nonnegative, smooth functions on $\mathbb{T}^d$ with respect to the $L^p(\mathbb{T}^d)$-norm.

\section{Definition of pathwise kinetic solutions}\label{characteristics}

In order to understand equation \eqref{intro_eq}, we will introduce a uniformly elliptic regularization driven by smooth noise.  The assumption \eqref{prelim_Holder} that $z$ is a geometric rough path ensures that there exists a sequence of smooth paths
\begin{equation}\label{e_noise} \left\{z^\epsilon:[0,\infty)\rightarrow\mathbb{R}^n\right\}_{\epsilon\in(0,1)},\end{equation}
such that, as $\epsilon\rightarrow 0$, {\color{black}for each $T>0$,} the paths $z^\epsilon$ converge to $z$ with respect to the $\alpha$-H\"older norm on the space of geometric rough paths $\C^{0,\alpha}([0,T];G^{\left\lfloor \frac{1}{\alpha}\right\rfloor}(\mathbb{R}^n))$ in the sense of \eqref{geometric_path}.  The precise meaning of this convergence is presented in Section~\ref{sec_rough}.  In what follows, for $\epsilon\in(0,1)$, we will use $\dot{z}^\epsilon$ to denote the time derivative of the smooth path.

It is furthermore necessary to introduce an $\eta$-perturbation by the Laplacian, for $\eta\in(0,1)$, in order to remove the degeneracy of the porous medium operator.  We therefore consider the equation, for $\eta\in(0,1)$ and $\epsilon\in(0,1)$,
\begin{equation}\label{chareq_eq}\left\{\begin{array}{ll} \partial_tu=\Delta u^{[m]}+\eta\Delta u+\nabla\cdot \left(A(x,u)\dot{z}^\epsilon_t\right) & \textrm{in}\;\;\mathbb{T}^d\times(0,\infty), \\ u=u_0 & \textrm{on}\;\;\mathbb{T}^d\times\{0\}.\end{array}\right.\end{equation}
The following proposition establishes the well-posedness of \eqref{chareq_eq}.  The proof and additional estimates can be found in Proposition~\ref{smooth_equation}.

\begin{prop}\label{chareq_prop}  For each $\eta\in(0,1)$, $\epsilon\in(0,1)$, and $u_0\in L^2(\mathbb{T}^d)$, there exists a classical solution of the equation
$$\left\{\begin{array}{ll} \partial_t u=\Delta u^{[m]}+\eta\Delta u+\nabla\cdot \left(A(x,u)\dot{z}^\epsilon_t\right) & \textrm{in}\;\;\mathbb{T}^d\times(0,\infty), \\ u=u_0 & \textrm{on}\;\;\mathbb{T}^d\times\{0\}.\end{array}\right.$$
\end{prop}

The kinetic formulation of \eqref{chareq_eq}, which is derived in more detail in Section~\ref{kinetic}, is obtained by introducing the kinetic function $\overline{\chi}:\mathbb{R}^2\rightarrow\{-1,0,1\}$ defined by
\begin{equation}\label{chareq_kinetic_foundation}\overline{\chi}(s,\xi):=\left\{\begin{array}{ll} 1 & \textrm{if}\;\;0<\xi<s, \\ -1 & \textrm{if}\;\;s<\xi<0, \\ 0 & \textrm{else.}\end{array}\right.\end{equation}
We then define, for each $\eta\in(0,1)$ and $\epsilon\in(0,1)$, for $u^{\eta,\epsilon}$ the solution of \eqref{chareq_eq}, the composition
\begin{equation}\label{chareq_kinetic_function}\chi^{\eta,\epsilon}(x,\xi,t):=\overline{\chi}(u^{\eta,\epsilon}(x,t),\xi).\end{equation}
After expanding the divergence appearing in \eqref{chareq_eq} by defining the matrix-valued function
\begin{equation}\label{chareq_b}b(x,\xi)=(b_{ij}(x,\xi)):=\partial_\xi A(x,\xi)\in\mathcal{M}^{d\times n},\end{equation}
and the vector-valued
\begin{equation}\label{chareq_c}c(x,\xi)=(c_j(x,\xi)):=\left(\sum_{i=1}^d\partial_{x_i}a_{ij}(x,\xi)\right)\in\mathbb{R}^n,\end{equation}
we prove in Proposition~\ref{smooth_kinetic_solution} that, for each $\eta\in(0,1)$ and $\epsilon\in(0,1)$, the kinetic function $\chi^{\eta,\epsilon}$ is a distributional solution of the equation
\begin{equation}\label{chareq_kin_eq}\begin{aligned} \partial_t\chi^{\eta,\epsilon}= & m\abs{\xi}^{m-1}\Delta_x\chi^{\eta,\epsilon}+\eta\Delta_x\chi^{\eta,\epsilon} +b(x,\xi)\dot{z}^\epsilon_t\cdot \nabla_x\chi^{\eta,\epsilon}- \left(c(x,\xi)\cdot\dot{z}^\epsilon_t\right)\partial_\xi\chi^{\eta,\epsilon} \\ & + \partial_\xi p^{\eta,\epsilon}(x,\xi,t)+\partial_\xi q^{\eta,\epsilon}(x,\xi,t), \end{aligned}\end{equation}
on $\mathbb{T}^d\times\mathbb{R}\times(0,\infty)$, with initial data $\overline{\chi}(u_0(x),\xi)$.  Here, the measure $p^{\eta,\epsilon}$ is the entropy defect measure
$$p^{\eta,\epsilon}(x,\xi,t):=\delta_0\left(\xi-u^{\eta,\epsilon}(x,t)\right)\eta\abs{\nabla u^{\eta,\epsilon}(x,t)}^2,$$
and the measure $q^{\eta,\epsilon}$ is the parabolic defect measure
$$q^{\eta,\epsilon}(x,\xi,t):=\delta_0\left(\xi-u^{\eta,\epsilon}(x,t)\right)\frac{4m}{(m+1)^2}\abs{\nabla\left(u^{\eta,\epsilon}\right)^{\left[\frac{m+1}{2}\right]}(x,t)}^2,$$
where $\delta_0$ denotes the one-dimensional Dirac mass centered at the origin.  The sense in which the kinetic function satisfies \eqref{chareq_kin_eq} is made precise by the following proposition.  The proof can be found in Proposition~\ref{smooth_kinetic_solution}.

\begin{prop}\label{chareq_smooth_kinetic_solution}  For each $\eta\in(0,1)$, $\epsilon\in(0,1)$, and $u_0\in L^2(\mathbb{T}^d)$, let $u^{\eta,\epsilon}$ denote the solution of \eqref{chareq_eq} from Proposition \ref{chareq_prop}.  Then, the kinetic function $\chi^{\eta,\epsilon}$ defined in \eqref{chareq_kinetic_function} is a distributional solution of \eqref{chareq_kin_eq} in the sense that, for every $t_1,t_2\in[0,\infty)$, for every $\psi\in\C^\infty_c(\mathbb{T}^d\times\mathbb{R}\times[t_1,t_2]))$,
\begin{equation}\label{chareq_smooth_kinetic_eq}\begin{aligned}& \left.\int_\mathbb{R}\int_{\mathbb{T}^d}\chi^{\eta,\epsilon}(x,\xi,t)\psi(x,\xi,t)\dx\dxi\right|_{t=t_1}^{t_2} =  \int_{t_1}^{t_2}\int_\mathbb{R}\int_{\mathbb{T}^d}\chi^{\eta,\epsilon} \partial_t\psi\dx\dxi\dt  \\ &\quad +\int_{t_1}^{t_2}\int_\mathbb{R}\int_{\mathbb{T}^d}\left(m\abs{\xi}^{m-1}+\eta\right)\chi^{\eta,\epsilon} \Delta_x\psi\dx\dxi\dt \\ &\quad- \int_{t_1}^{t_2}\int_\mathbb{R}\int_{\mathbb{T}^d}\chi^{\eta,\epsilon}\nabla_x\cdot\left(\left(b(x,\xi)\dot{z}^\epsilon_t\right)\psi\right)-\chi^{\eta,\epsilon}\partial_\xi\left(\left(c(x,\xi)\cdot \dot{z}^\epsilon_t\right)\psi\right)\dx\dxi\dt \\ &\quad-\int_{t_1}^{t_2}\int_\mathbb{R}\int_{\mathbb{T}^d}\left(p^{\eta,\epsilon}+q^{\eta,\epsilon}\right)\partial_\xi\psi \dx\dxi\dt. \end{aligned}\end{equation}
\end{prop}

The purpose of this section is to understand the system of stochastic characteristics associated to equation \eqref{chareq_smooth_kinetic_eq}, where the goal is to remove the dependency of equation on the derivative of the noise.  To achieve this, test functions are transported by a system of inverse stochastic characteristics, where the transport of a test function $\rho_0\in \C^\infty_c(\mathbb{T}^d\times\mathbb{R})$ is the solution
\begin{equation}\label{char_eq_initial} \left\{\begin{array}{ll} \partial_t\rho^\epsilon=\nabla_x\cdot\left(\left(b(x,\xi)\dot{z}^\epsilon_t\right)\rho^\epsilon\right)-\partial_\xi\left(\left(c(x,\xi)\cdot \dot{z}^\epsilon_t\right) \rho^\epsilon\right) & \textrm{in}\;\;\mathbb{T}^d\times\mathbb{R}\times(0,\infty), \\ \rho^\epsilon=\rho_0 & \textrm{on}\;\;\mathbb{T}^d\times\mathbb{R}\times\{0\}.\end{array}\right.\end{equation}
Indeed, it is not immediately clear that (\ref{char_eq_initial}) is a transport equation.  However, thanks to the equation's conservative structure, and in particular using definitions (\ref{chareq_b}) and (\ref{chareq_c}), it follows from a direct computation that
\begin{equation}\label{char_conservative} \nabla_x\cdot\left(b(x,\xi)\dot{z}^\epsilon_t\right)-\partial_\xi\left(c(x,\xi)\cdot \dot{z}^\epsilon_t\right)=0.\end{equation}
Therefore, equation (\ref{char_eq_initial}) simplifies to yield the pure transport equation
\begin{equation}\label{char_eq} \left\{\begin{array}{ll} \partial_t\rho^\epsilon=b(x,\xi)\dot{z}^\epsilon_t\cdot\nabla_x\rho^\epsilon-\left(c(x,\xi)\cdot \dot{z}^\epsilon_t\right) \partial_\xi\rho^\epsilon & \textrm{in}\;\;\mathbb{T}^d\times\mathbb{R}\times(0,\infty), \\ \rho^\epsilon=\rho_0 & \textrm{on}\;\;\mathbb{T}^d\times\mathbb{R}\times\{0\}.\end{array}\right.\end{equation}
We will now prove that $\rho^\epsilon$ of (\ref{char_eq}) is represented by the initial data $\rho_0$ transported by a system of underlying inverse characteristics.

The forward characteristic $(X^{x,\xi,\epsilon}_{t_0,t}, \Xi^{x,\xi,\epsilon}_{t_0,t})$ associated to (\ref{char_eq}) beginning at $t_0\geq 0$ and $(x,\xi)\in\mathbb{T}^d\times\mathbb{R}$ is defined as the solution of the system
\begin{equation}\label{kin_forward} \left\{\begin{array}{ll} \dot{X}^{x,\xi,\epsilon}_{t_0,t} = -b(X^{x,\xi,\epsilon}_{t_0,t}, \Xi^{x,\xi,\epsilon}_{t_0,t})\dot{z}^\epsilon_t & \textrm{in}\;\;(t_0,\infty), \\ \dot{\Xi}^{x,\xi,\epsilon}_{t_0,t} = c(X^{x,\xi,\epsilon}_{t_0,t}, \Xi^{x,\xi,\epsilon}_{t_0,t})\cdot \dot{z}^\epsilon_t & \textrm{in}\;\;(t_0,\infty), \\ (X^{x,\xi,\epsilon}_{t_0,t_0}, \Xi^{x,\xi,\epsilon}_{t_0,t_0})=(x,\xi). & \end{array}\right.\end{equation}
The corresponding backward characteristic is obtained by reversing the path $z$.  For each $t_0\geq 0$, define the reversed path
$$z^\epsilon_{t_0,t}:=z^\epsilon_{t-t_0}\;\;\textrm{for each}\;t\in[0,t_0].$$
The backward characteristic $(Y^{x,\xi,\epsilon}_{t_0,t},\Pi^{x,\xi,\epsilon}_{t_0,t})$ beginning from $(x,\xi)\in\mathbb{T}^d\times\mathbb{R}$ corresponding to the path reversed at time $t_0\geq 0$ is the solution of the system
\begin{equation}\label{kin_backward} \left\{\begin{array}{ll} \dot{Y}^{x,\xi,\epsilon}_{t_0,t}= -b(Y^{x,\xi,\epsilon}_{t_0,t},\Pi^{x,\xi,\epsilon}_{t_0,t})\dot{z}^\epsilon_{t_0,t} & \textrm{in}\;\;(0,t_0), \\ \dot{\Pi}^{x,\xi,\epsilon}_{t_0,t} = c(Y^{x,\xi,\epsilon}_{t_0,t},\Pi^{x,\xi,\epsilon}_{t_0,t})\cdot \dot{z}^\epsilon_{t_0,t} & \textrm{in}\;\;(0,t_0), \\ (Y^{x,\xi,\epsilon}_{t_0,0},\Pi^{x,\xi,\epsilon}_{t_0,0})=(x,\xi). &\end{array}\right.\end{equation}
The characteristics (\ref{kin_forward}) and (\ref{kin_backward}) are mutually inverse in the sense that, for each $(x,\xi)\in\mathbb{T}^d\times\mathbb{R}$, for each $t_0\geq 0$ and $t\geq t_0$, and for each $s_0\geq 0$ and $s\in[0,s_0]$,
\begin{equation}\label{kin_inverse} \left(X^{Y^{x,\xi,\epsilon}_{t,t-t_0},\Pi^{x,\xi,\epsilon}_{t,t-t_0},\epsilon}_{t_0,t}, \Xi^{Y^{x,\xi,\epsilon}_{t,t-t_0}, \Pi^{x,\xi,\epsilon}_{t,t-t_0},\epsilon}_{t_0,t}\right)=\left(Y^{X^{x,\xi,\epsilon}_{s_0-s,s},\Xi^{x,\xi,\epsilon}_{s_0-s,s},\epsilon}_{s_0,s}, \Pi^{X^{x,\xi,\epsilon}_{s_0-s,s}, \Xi^{x,\xi,\epsilon}_{s_0-s,s},\epsilon}_{s_0,s}\right)=(x,\xi). \end{equation}

The solution of (\ref{char_eq}) is the transport of the initial data by the backward characteristics (\ref{kin_backward}).  Precisely, for each $\rho_0\in\C^\infty_c(\mathbb{T}^d\times\mathbb{R})$, a direct computation proves that the solution $\rho$ of (\ref{char_eq}) admits the representation
\begin{equation}\label{kin_rep} \rho^\epsilon(x,\xi,t)=\rho_0(Y^{x,\xi,\epsilon}_{t,t}, \Pi^{x,\xi,\epsilon}_{t,t}).\end{equation}
For the arguments of this paper, it will be furthermore necessary to start the forward and backward characteristics at arbitrary points $t_0\in[0,\infty)$.  That is, for each $t_0\in[0,\infty)$, consider the equation
\begin{equation}\label{kin_translate} \left\{\begin{array}{ll} \partial_t\rho^\epsilon_{t_0,t}=\left(b(x,\xi)\dot{z}^\epsilon_t\right)\cdot\nabla_x\rho^\epsilon_{t_0,t}-\left(c(x,\xi)\cdot \dot{z}^\epsilon_t\right) \partial_\xi\rho^\epsilon_{t_0,t} & \textrm{in}\;\;\mathbb{T}^d\times\mathbb{R}\times(t_0,\infty), \\ \rho^\epsilon_{t_0,t_0}=\rho_0 & \textrm{on}\;\;\mathbb{T}^d\times\mathbb{R}\times\{t_0\}.\end{array}\right.\end{equation}
The identical computations leading to (\ref{kin_rep}) prove that, for each $\rho_0\in \C^\infty_c(\mathbb{T}^d\times\mathbb{R})$, the solution of (\ref{kin_translate}) is given by
\begin{equation}\label{kinetic_rep_1} \rho^\epsilon_{t_0,t}(x,\xi,t)=\rho_0(Y^{x,\xi,\epsilon}_{t,t-t_0}, \Pi^{x,\xi,\epsilon}_{t,t-t_0}).\end{equation}

Furthermore, as a consequence of (\ref{char_eq_initial}) and (\ref{char_conservative}), the characteristics preserve the Lebesgue measure on $\mathbb{T}^d\times\mathbb{R}$.  That is, for every $0\leq t_0<t_1$ and $0<s_1<s_0$, for every $\psi\in L^1(\mathbb{T}^d\times\mathbb{R})$,
\begin{equation}\label{kin_measure}\int_\mathbb{R}\int_{\mathbb{T}^d}\psi(x,\xi)\dx \dxi = \int_\mathbb{R}\int_{\mathbb{T}^d} \psi(X^{x,\xi,\epsilon}_{t_0,t_1}, \Xi^{x,\xi,\epsilon}_{t_0,t_1})\;\dx\dxi = \int_\mathbb{R}\int_{\mathbb{T}^d}\psi(Y^{x,\xi,\epsilon}_{s_0,s_1}, \Pi^{x,\xi,\epsilon}_{s_0,s_1})\dx\dxi.\end{equation}
This observation is implicit in the definition of a pathwise kinetic solution to (\ref{intro_eq}), and it is essential to the proof of uniqueness in the next section.  It is also a consequence of (\ref{prelim_vanish}) that the characteristics preserve the sign of the velocity.  That is, for each $(x,\xi)\in\mathbb{T}^d\times\mathbb{R}$, for each $t_0\geq 0$ and $t\geq t_0$, and for each $s_0\geq 0$ and $s\in[0,s_0]$,
\begin{equation}\label{kin_sign} \Xi^{x,\xi,\epsilon}_{t_0,t}=\Pi^{x,\xi,\epsilon}_{s_0,s}=0\;\;\textrm{if and only if}\;\;\xi=0,\;\;\textrm{and}\;\;\sgn(\xi)=\sgn(\Xi^{x,\xi,\epsilon}_{t_0,t})=\sgn(\Pi^{x,\xi,\epsilon}_{s_0,s})\;\;\textrm{if}\;\;\xi\neq 0.\end{equation}

The following proposition, which is an immediate consequence of the smoothness (\ref{prelim_regular}), Proposition \ref{chareq_prop}, and equation (\ref{kin_translate}), makes precise the notion of testing equation \eqref{chareq_smooth_kinetic_eq} with functions transported along the inverse characteristics.  The transport is expressed by the representation (\ref{kinetic_rep_1}).  Finally, we remark that the integration by parts formula is an immediate consequence of the distributional equality
$$\nabla_x\chi^{\eta,\epsilon}(x,\xi,t)=\delta_0(\xi-u^{\eta,\epsilon}(x,t))\nabla u^{\eta,\epsilon},$$
which can be proven, for instance, by considering the composition of a convolution of \eqref{chareq_kinetic_foundation} with $u^{\eta,\epsilon}$, and then using the fact that $u^{\eta,\epsilon}$ has a distributional derivative.

\begin{prop}\label{kinetic_smooth_equation}  Let $\eta\in(0,1)$, $\epsilon\in(0,1)$, and $u_0\in L^2(\mathbb{T}^d)$.  The kinetic function $\chi^{\eta,\epsilon}$ from Proposition \ref{chareq_smooth_kinetic_solution} satisfies, for each $t_0,t_1\in[0,\infty)$ and $\rho_0\in \C^\infty_c(\mathbb{T}^d\times\mathbb{R})$, for the solution $\rho^\epsilon_{t_0,\cdot}(\cdot,\cdot)$ of \eqref{kin_translate},
\begin{equation}\label{transport_equation_smooth}\begin{aligned} &\left.\int_\mathbb{R}\int_{\mathbb{T}^d}\chi^{\eta,\epsilon}(x,\xi,s)\rho^\epsilon_{t_0,s}(x,\xi)\dx\dxi\right|_{s=t_0}^{t_1}   \\ & = \int_{t_0}^{t_1}\int_\mathbb{R}\int_{\mathbb{T}^d}\left(m\abs{\xi}^{m-1}+\eta\right)\chi^{\eta,\epsilon}(x,\xi,s) \Delta_x\rho^\epsilon_{t_0,s}(x,\xi) \dx\dxi\ds \\ & \quad -\int_{t_0}^{t_1}\int_\mathbb{R}\int_{\mathbb{T}^d}\left(p^{\eta,\epsilon}(x,\xi,s)+q^{\eta,\epsilon}(x,\xi,s)\right)\partial_\xi\rho^\epsilon_{t_0,s}(x,\xi) \dx\dxi\ds. \end{aligned}\end{equation}
\end{prop}

The essential observation in the passage to the singular limit $\epsilon\rightarrow 0$ is that the system of characteristics \eqref{kin_backward} is well-posed for rough noise when interpreted as a rough differential equation.  In view of the representation \eqref{kinetic_rep_1}, this implies the well-posedness of the transport equation \eqref{char_eq} for rough signals as well.  The details are presented in Section~\ref{sec_rough}.

For each $(x,\xi)\in\mathbb{T}^d\times\mathbb{R}$ and $t_0\geq 0$, let $\left(X^{x,\xi}_{t_0,t},\Xi^{x,\xi}_{t_0,t}\right)$ denote the solution of the rough differential equation
\begin{equation}\label{kin_rough_for} \left\{\begin{array}{ll}  \dd X^{x,\xi}_{t_0,t} = -b(X^{x,\xi}_{t_0,t}, \Xi^{x,\xi}_{t_0,t})\circ \dz_t & \textrm{in}\;\;(t_0,\infty), \\ \dd \Xi^{x,\xi}_{t_0,t}= c(X^{x,\xi}_{t_0,t}, \Xi^{x,\xi}_{t_0,t})\circ \dz_t & \textrm{in}\;\;(t_0,\infty), \\ (X^{x,\xi}_{t_0,t_0}, \Xi^{x,\xi}_{t_0,t_0})=(x,\xi). & \end{array}\right.\end{equation}
Similarly, for each $t_0\geq 0$ and $(x,\xi)\in\mathbb{T}^d\times\mathbb{R}$, for the reversed path
$$z_{t_0,t}:=z_{t_0-t}\;\;\textrm{for}\;\;t\in[0,t_0],$$
let $\left(Y^{x,\xi}_{t_0,t}, \Pi^{x,\xi}_{t_0,t}\right)$ denote the solution of the inverse rough differential equation
\begin{equation}\label{kin_rough_back} \left\{\begin{array}{ll}  \dd Y^{x,\xi}_{t_0,t} = -b(Y^{x,\xi}_{t_0,t}, \Pi^{x,\xi}_{t_0,t})\circ \dz_{t_0,t} & \textrm{in}\;\;(0,t_0), \\ \dd \Pi^{x,\xi}_{t_0,t}= c(Y^{x,\xi}_{t_0,t}, \Pi^{x,\xi}_{t_0,t})\circ \dz_{t_0,t} & \textrm{in}\;\;(0,t_0), \\ (Y^{x,\xi}_{t_0,0}, \Pi^{x,\xi}_{t_0,0})=(x,\xi). & \end{array}\right.\end{equation}
The systems \eqref{kin_rough_for} and \eqref{kin_rough_back} are inverse in the sense that, for every $(x,\xi)\in\mathbb{T}^d\times\mathbb{R}$, $0\leq t_0\leq t$, and $0\leq s\leq s_0$,
$$\left(X^{Y^{x,\xi}_{t,t-t_0},\Pi^{x,\xi}_{t,t-t_0}}_{t_0,t},\Xi^{Y^{x,\xi}_{t,t-t_0},\Pi^{x,\xi}_{t,t-t_0}}_{t_0,t}\right)=(x,\xi)\;\;\textrm{and}\;\;\left(Y^{X^{x,\xi}_{s-s_0,s},\Xi^{x,\xi}_{s-s_0,s}}_{s_0,s},\Pi^{X^{x,\xi}_{s-s_0,s},\Xi^{x,\xi}_{s-s_0,s}}_{s_0,s}\right)=(x,\xi).$$

The conservative structure of the equation is preserved in the limit, since it is immediate from \eqref{kin_measure} that the rough characteristics preserve the Lebesgue measure.  That is, for each $0\leq t_0\leq t_1$ and for each $0\leq s_1\leq s_0$, for every $\psi\in L^1(\mathbb{T}^d\times\mathbb{R})$,
\begin{equation}\label{kin_rough_measure}\int_\mathbb{R}\int_{\mathbb{T}^d}\psi(x,\xi)\dx \dxi = \int_\mathbb{R}\int_{\mathbb{T}^d} \psi(X^{x,\xi}_{t_0,t_1}, \Xi^{x,\xi}_{t_0,t_1})\;\dx\dxi = \int_\mathbb{R}\int_{\mathbb{T}^d}\psi(Y^{x,\xi,\epsilon}_{s_0,s_1}, \Pi^{x,\xi}_{s_0,s_1})\dx\dxi.\end{equation}
It is also a consequence of \eqref{prelim_vanish} and \eqref{kin_sign} that the rough characteristics preserve the sign of the velocity.  That is, for each $(x,\xi)\in\mathbb{T}^d\times\mathbb{R}$, $0\leq t_0\leq t_1$, and $0\leq s_1\leq s_0$,
\begin{equation}\label{kin_rough_sign} \Xi^{x,\xi}_{t_0,t_1}=\Pi^{x,\xi}_{s_0,s_1}=0\;\;\textrm{if and only if}\;\;\xi=0,\;\;\textrm{and}\;\;\sgn(\xi)=\sgn(\Xi^{x,\xi}_{t_0,t_1})=\sgn(\Pi^{x,\xi}_{s_0,s_1})\;\;\textrm{if}\;\;\xi\neq 0.\end{equation}

It follows from well-posedness of the characteristics systems \eqref{kin_rough_for} and \eqref{kin_rough_back} that the rough transport equation, for each $t_0\geq 0$,
\begin{equation}\label{char_rough_eq} \left\{\begin{array}{ll} \partial_t\rho_{t_0,t}=\left(b(x,\xi)\circ\dz_t\right)\cdot\nabla_x\rho_{t_0,t}-\left(c(x,\xi)\circ\dz_t\right) \partial_\xi\rho_{t_0,t} & \textrm{in}\;\;\mathbb{T}^d\times\mathbb{R}\times(t_0,\infty), \\ \rho_{t_0,t_0}=\rho_0 & \textrm{on}\;\;\mathbb{T}^d\times\mathbb{R}\times\{t_0\},\end{array}\right.\end{equation}
is well-posed for initial data $\rho_0\in\C^\infty_c(\mathbb{T}^d\times\mathbb{R})$.  Indeed, in analogy with \eqref{kinetic_rep_1}, the solution is represented by the transport of the initial data by the inverse characteristics \eqref{kin_rough_back}.  That is, for each $t_0\geq 0$ and $\rho_0\in\C^\infty_c(\mathbb{T}^d\times\mathbb{R})$, the solution of \eqref{char_rough_eq} admits the representation
\begin{equation}\label{kin_rough_rep} \rho_{t_0,t}(x,\xi)=\rho_0\left(Y^{x,\xi}_{t,t-t_0}, \Pi^{x,\xi}_{t,t-t_0}\right).\end{equation}

We are now prepared to present the definition of a pathwise kinetic solution.  Proposition~\ref{stable_L1} and Proposition~\ref{stable_estimates} prove that, uniformly for the solutions $\{u^{\eta,\epsilon}\}_{\eta,\epsilon\in(0,1)}$,
\begin{equation}\label{kin_stable_estimates_2}\norm{u^{\eta,\epsilon}}_{L^\infty\left([0,\infty);L^1(\mathbb{T}^d)\right)}\leq \norm{u_0}_{L^1(\mathbb{T}^d)},\end{equation}
and, for each $T>0$, for $C=C(m,d,T)>0$,
\begin{equation}\begin{aligned}\label{kin_stable_estimates_1} & \norm{u^{\eta,\epsilon}}_{L^\infty([0,T];L^2(\mathbb{T}^d))}^2+\norm{\nabla \left(u^{\eta,\epsilon}\right)^{\left[\frac{m+1}{2}\right]}}_{L^2([0,T];L^2(\mathbb{T}^d;\mathbb{R}^d))}+\norm{\eta^\frac{1}{2}\nabla u^{\eta,\epsilon}}^2_{L^2\left([0,T];L^2(\mathbb{T}^d;\mathbb{R}^d)\right)} \\ & \leq  C\left(\norm{u_0}^2_{L^2(\mathbb{T}^d)}+\norm{u_0}^2_{L^1(\mathbb{T}^d)}+\norm{u_0}_{L^1(\mathbb{T}^d)}^{m+1}\right).\end{aligned}\end{equation}

It is not difficult to prove that, as $\eta\rightarrow 0$, the entropy defect measures $\left\{p^{\eta,\epsilon}\right\}_{\eta,\epsilon\in(0,1)}$ converge weakly to zero, owing to the regularity implied by the parabolic defect measures $\left\{q^{\eta,\epsilon}\right\}_{\eta,\epsilon\in(0,1)}$.  However, due to the weak lower semicontinuity of the $L^2$-norm, along a subsequence, the weak limit of the parabolic defect measures $\left\{q^{\eta,\epsilon}\right\}_{\eta,\epsilon\in(0,1)}$ may lose mass in the limit, since the gradients
$$\left\{\nabla\left(u^{\eta,\epsilon}\right)^{\left[\frac{m+1}{2}\right]}\right\}_{\eta,\epsilon\in(0,1)},$$
will, in general, converge only weakly.  The entropy defect measure appearing in Definiton~\ref{def_solution} is therefore necessary to account for this potential loss of mass.

\begin{definition}\label{def_solution} For $u_0\in L^2(\mathbb{T}^d)$, a \emph{pathwise kinetic solution} of \eqref{intro_eq} is a function satisfying, for each $T>0$,
{\color{black}$$u\in L^\infty([0,T];L^2(\mathbb{T}^d)),$$}
and the following two properties.

(i)  For each $T>0$, 
$$u^{\left[\frac{m+1}{2}\right]}\in L^2([0,T];H^1(\mathbb{T}^d)).$$
In particular, for each $T>0$, the parabolic defect measure
$$q(x,\xi,t):=\frac{4m}{(m+1)^2}\delta_0(\xi-u(x,t))\abs{\nabla u^{\left[\frac{m+1}{2}\right]}}^2\;\;\textrm{for}\;\;(x,\xi,t)\in \mathbb{T}^d\times\mathbb{R}\times(0,\infty),$$
is finite on $\mathbb{T}^d\times\mathbb{R}\times (0,T)$.

(ii)  For the kinetic function
$$\chi(x,\xi,t):=\overline{\chi}(u(x,t),\xi)\;\;\textrm{for}\;\;(x,\xi,t)\in \mathbb{T}^d\times\mathbb{R}\times[0,\infty),$$
there exists a finite, nonnegative entropy defect measure $p$ on $\mathbb{T}^d\times\mathbb{R}\times(0,\infty)$ satisfying, for each $T>0$,
$$\int_0^T\int_\mathbb{R}\int_{\mathbb{T}^d}p\dx\dxi\dr<\infty,$$
and a subset $\mathcal{N}\subset(0,\infty)$ of Lebesgue measure zero such that, for every $\rho_0\in\C^\infty_c(\mathbb{T}^d\times\mathbb{R})$, for $\rho_{s,\cdot}(\cdot,\cdot)$ satisfying \eqref{char_rough_eq}, for every $s<t\in[0,\infty)\setminus\mathcal{N}$,
\begin{equation}\begin{aligned}\label{transport_equation}\left.\int_\mathbb{R}\int_{\mathbb{T}^d}\chi(x,\xi,r)\rho_{s,r}(x,\xi)\dx\dxi\right|_{r=s}^{r=t} = & \int_s^t\int_\mathbb{R}\int_{\mathbb{T}^d} m\abs{\xi}^{m-1}\chi\Delta \rho_{s,r}\dx\dxi\dr \\ & - \int_s^t\int_\mathbb{R}\int_{\mathbb{T}^d} \left(p+q\right) \partial_\xi \rho_{s,r}\dx\dxi\dr,\end{aligned}\end{equation}
where the initial condition is enforced in the sense that, when $s=0$,
$$\int_\mathbb{R}\int_{\mathbb{T}^d}\chi(x,\xi,0)\rho_{0,0}(x,\xi)\dx\dxi=\int_\mathbb{R}\int_{\mathbb{T}^d}\overline{\chi}(u_0(x),\xi)\rho_0(x,\xi)\dx\dxi.$$
\end{definition}

\begin{remark} Observe that \eqref{transport_equation} is equivalent to requiring that the kinetic function $\chi$ satisfies, for each $\phi\in\C^\infty_c([0,\infty))$, $t_0\geq 0$, and $\rho_0\in \C^\infty_c(\mathbb{T}^d\times\mathbb{R})$, for the solution $\rho_{t_0,\cdot}(\cdot,\cdot)$ of (\ref{char_rough_eq}),
\begin{equation}\label{transport_equation_last}\begin{aligned} \int_{t_0}^\infty\int_\mathbb{R}\int_{\mathbb{T}^d}\chi(x,\xi,r)\rho_{t_0,r}(x,\xi)\phi'(r) = & -\int_\mathbb{R}\int_{\mathbb{T}^d}\chi(x,\xi,0)\rho_{t_0,t_0}(x,\xi)\phi(t_0)\\ & -\int_{t_0}^\infty\int_\mathbb{R}\int_{\mathbb{T}^d}m\abs{\xi}^{m-1}\chi(x,\xi,r)\Delta_x\rho_{t_0,r}(x,\xi)\phi(r) \\ & +\int_{t_0}^\infty\int_\mathbb{R}\int_{\mathbb{T}^d}\left(p(x,\xi,r)+q(x,\xi,r)\right)\partial_\xi\rho_{t_0,r}(x,\xi)\phi(r). \end{aligned}\end{equation}
The proof is a consequence of the Lebesgue differentiation theorem applied in time to a sequence of smooth approximations of the indicator functions of intervals $[t_0,t_1]$, for each $t_1\geq t_0$. \end{remark}

{\color{black}We observe that the regularity of Definition~\ref{def_solution} (i) implies that every pathwise kinetic solutions satisfies the following integration by parts formula:  for each $\psi\in\C^\infty_c(\mathbb{T}^d\times\mathbb{R}\times[0,\infty))$, for each $t\geq 0$,
\begin{equation}\label{ibp_formula}\int_0^t\int_\mathbb{R}\int_{\mathbb{T}^d} \frac{m+1}{2}\abs{\xi}^\frac{m-1}{2}\chi(x,\xi,r) \nabla\psi(x,\xi,r)\dx\dxi\dr = -\int_0^t\int_{\mathbb{T}^d} \nabla u^{\left[\frac{m+1}{2}\right]}\psi(x,u(x,r),r)\dx\dr.\end{equation}
We emphasize that in anisotropic settings, see for instance \cite[Definition~2.2]{ChenPerthame}, it would be furthermore necessary to postulate either a chain rule or integration by parts formula like \eqref{ibp_formula} in the definition of a pathwise kinetic solution.  The proof of the \eqref{ibp_formula} is consequence of the following lemma, which is motivated by \cite[Appendix~A]{ChenPerthame} and which relies upon the fact that the nonlinear diffusive term is isotropic.

\begin{lem}\label{lem_ibp}  Let $z:\mathbb{T}^d\rightarrow\mathbb{R}$ be measurable and suppose that
$$z^{\left[\frac{m+1}{2}\right]}\in H^1(\mathbb{T}^d).$$
Then, for each $\psi\in\C^\infty_c(\mathbb{T}^d\times\mathbb{R})$, for the kinetic function $\chi$ of $z$,
$$\int_\mathbb{R}\int_{\mathbb{T}^d} \frac{m+1}{2}\abs{\xi}^\frac{m-1}{2}\chi(x,\xi) \nabla\psi(x,\xi)\dx\dxi = -\int_{\mathbb{T}^d} \nabla z^{\left[\frac{m+1}{2}\right]}\psi(x,z(x))\dx.$$
\end{lem}

\begin{proof}  Let $\psi\in\C^\infty_c(\mathbb{T}^d\times\mathbb{R})$ be arbitrary.  For a measurable function $z$ on $\mathbb{T}^d$ satisfying $z^{\left[\frac{m+1}{2}\right]}\in H^1(\mathbb{T}^d)$, we will write $\chi$ for the kinetic function of $z$ and $\tilde{\chi}$ for the kinetic function of the signed power $z^{\left[\frac{m+1}{2}\right]}$.  Define the signed power, for $\xi\in\mathbb{R}$,
$$\beta(\xi):=\xi^{\left[\frac{m+1}{2}\right]}.$$
The monotonicity of $\beta$ and the change of variables formula prove that
\begin{equation}\label{libp_1}\int_\mathbb{R}\int_{\mathbb{T}^d} \frac{m+1}{2}\abs{\xi}^\frac{m-1}{2}\chi(x,\xi) \nabla\psi(x,\xi)\dx\dxi=\int_\mathbb{R}\int_{\mathbb{T}^d}\chi(x,\beta^{-1}(\xi))\nabla\psi(x,\beta^{-1}(\xi)).\end{equation}
It follows from the definitions of $\beta$ and the kinetic functions $\chi$ and $\tilde{\chi}$ that, for each $(x,\xi)\in\mathbb{T}^d\times\mathbb{R}$,
\begin{equation}\label{libp_2}\chi(x,\beta^{-1}(\xi))=\tilde{\chi}(x,\xi).\end{equation}
Since $z^{\left[\frac{m+1}{2}\right]}\in H^1(\mathbb{T}^d)$, an approximation argument proves the distributional equality
\begin{equation}\label{libp_3}\nabla\tilde{\chi}(x,\xi)=\delta_0(\xi-z^{\left[\frac{m+1}{2}\right]})\nabla z^{\left[\frac{m+1}{2}\right]}.\end{equation}
where $\delta_0$ is the one-dimensional Dirac mass at zero.  Therefore, returning to \eqref{libp_1}, it follows from \eqref{libp_2}, \eqref{libp_3}, and the definition of $\beta$ that
$$\begin{aligned}\int_\mathbb{R}\int_{\mathbb{T}^d} \frac{m+1}{2}\abs{\xi}^\frac{m-1}{2}\chi(x,\xi) \nabla\psi(x,\xi)\dx\dxi= & \int_\mathbb{R}\int_{\mathbb{T}^d}\tilde{\chi}(x,\xi)\nabla\psi(x,\beta^{-1}(\xi))\dx\dxi \\ = & -\int_{\mathbb{T}^d}\nabla z^{\left[\frac{m+1}{2}\right]}\psi\left(x,\beta^{-1}\left(z^{\left[\frac{m+1}{2}\right]}(x)\right)\right)\dx \\ = & -\int_{\mathbb{T}^d}\nabla z^{\left[\frac{m+1}{2}\right]}\psi(x,z(x))\dx,\end{aligned}$$
which completes the proof. \end{proof}}

\section{Uniqueness}\label{uniqueness}

In this section, we prove that pathwise kinetic solutions are unique.  In order to motivate and give an overview of the proof, we begin by briefly sketching the uniqueness argument for the deterministic porous medium equation
\begin{equation}\label{sketch_0}\left\{\begin{array}{ll} \partial_t u=\Delta u^{[m]} & \textrm{in}\;\;\mathbb{T}^d\times(0,\infty), \\ u=u_0 & \textrm{on}\;\;\mathbb{\mathbb{T}}^d\times\{0\}.\end{array}\right.\end{equation}
The corresponding kinetic formulation is
\begin{equation}\label{sketch}\left\{\begin{array}{ll} \partial_t \chi=m\abs{\xi}^{m-1}\Delta_x\chi+\partial_\xi(p+q) & \textrm{in}\;\;\mathbb{T}^d\times\mathbb{R}\times(0,\infty), \\ \chi=\overline{\chi}(u_0;\xi) & \textrm{on}\;\;\mathbb{T}^d\times\mathbb{R}\times\{0\},\end{array}\right.\end{equation}
where $p\geq 0$ is the nonnegative entropy defect measure and the parabolic defect measure $q$ is defined by
$$q(x,\xi,t):=\delta_0(\xi-u(x,t))\frac{4m}{(m+1)^2}\abs{\nabla u^{\left[\frac{m+1}{2}\right]}(x,t)}^2.$$

In this setting, the following proof of uniqueness is due to \cite{ChenPerthame}.  Suppose that $u^1$ and $u^2$ are two kinetic solutions of \eqref{sketch_0} in the sense that the associated kinetic functions $\chi^1$ and $\chi^2$ solve \eqref{sketch}.  Properties of the kinetic function yield the identity
\begin{equation}\label{cp_arg}\begin{aligned}  \int_{\mathbb{T}^d}\abs{u^1-u^2}\dx = &  \int_{\mathbb{R}}\int_{\mathbb{T}^d}\abs{\chi^1-\chi^2}^2\dx\dxi = \int_{\mathbb{R}}\int_{\mathbb{T}^d}\abs{\chi^1}+\abs{\chi^2}-2\chi^1\chi^2\dx\dxi \\ = & \int_{\mathbb{R}}\int_{\mathbb{T}^d}\chi^1\sgn(\xi)+\chi^2\sgn(\xi)-2\chi^1\chi^2\dx\dxi.\end{aligned}\end{equation}
The distributional equalities, for $i\in\{1,2\}$,
$$\partial_\xi\chi^i(x,\xi,t)=\delta_0(\xi)-\delta_0(\xi-u^i(x,t))\;\;\textrm{and}\;\;\;\nabla_x\chi^i(x,\xi,t)=\delta_0(\xi-u^i(x,t))\nabla u^i(x,t),$$
yield formally, after taking the derivative in time of \eqref{cp_arg}, applying equation \eqref{sketch}, and integrating by parts in space,
{\color{black}\begin{equation}\label{sketch_1}\begin{aligned}\partial_t\int_{\mathbb{T}^d}\abs{u^1-u^2}= & \frac{16m}{(m+1)^2}\int_{\mathbb{R}}\int_{\mathbb{T}^d}\delta_0(\xi-u^1(x,t))\delta_0(\xi-u^2(x,t))\nabla(u^1)^{\left[\frac{m+1}{2}\right]}\cdot \nabla(u^2)^{\left[\frac{m+1}{2}\right]}\\ & -2\int_{\mathbb{R}}\int_{\mathbb{T}^d}\delta_0(\xi-u^1(x,t))\left(p^2(x,\xi,t)+q^2(x,\xi,t)\right) \\ & -2\int_{\mathbb{R}}\int_{\mathbb{T}^d}\delta_0(\xi-u^2(x,t))\left(p^1(x,\xi,t)+q^1(x,\xi,t)\right).  \end{aligned}\end{equation}}
Applications of H\"older's inequality and Young's inequality, together with the definition of the parabolic defect measure and the nonegativity of the entropy defect measure, prove that the righthand side of \eqref{sketch_1} is nonpositive.  Integrating in time then completes the proof of uniqueness.

The formal argument leading to \eqref{sketch_1} provides the outline for the proof of Theorem~\ref{theorem_uniqueness} below.  However, even to justify the formal computation, care must be taken to avoid the product of $\delta$-distributions.  This is achieved by regularizing the $\sgn$ and kinetic functions in the spatial and velocity variables.  Additional error terms arise due to the transport of test functions by the inverse characteristics, which are handled using a time-splitting argument that relies crucially on the conservative structure of the equation.

The proof of uniqueness is broken down into six steps.  The first introduces the regularization, the second handles the terms involving the $\sgn$ function, and the third handles the mixed term.  The fourth makes rigorous the cancellation coming from the parabolic defect measures, the fifth analyzes the error terms, and the sixth concludes the proof by passing to the limit first with respect to the regularization and second the time-splitting.

\begin{remark}  In the proof of Theorem~\ref{theorem_uniqueness} and for the remainder of the paper, after applying the integration by parts formula of Lemma~\ref{lem_ibp}, we will frequently encounter derivatives of functions $f(x,\xi,r):\mathbb{T}^d\times\mathbb{R}\times[0,\infty)\rightarrow\mathbb{R}$ evaluated at $\xi=u(x,r)$.  In order to simplify the notation, we make the convention that
$$\nabla_xf(x,u(x,r),r):=\left.\nabla_xf(x,\xi,r)\right|_{\xi=u(x,r)},$$
and analogous conventions for all possible derivatives.  That is, in every case, the notation indicates the derivative of $f$ evaluated at $(x,u(x,r),r)$ as opposed to the derivative of the full composition.  \end{remark}

\begin{thm}\label{theorem_uniqueness}  Let $u_0^1,u_0^2\in L^2_+(\mathbb{T}^d)$.  Suppose that $u^1$ and $u^2$ are pathwise kinetic solutions of \eqref{intro_eq} in the sense of Definition~\ref{def_solution} with initial data $u^1_0$ and $u^2_0$.  Then,
$$\norm{u^1-u^2}_{L^\infty\left([0,\infty);L^1(\mathbb{T}^d)\right)}\leq \norm{u^1_0-u^2_0}_{L^1(\mathbb{T}^d)}.$$
\end{thm}

\begin{proof}  The proof will proceed in six steps.  The first introduces an approximation scheme which is necessary in order to apply the equation.

\textbf{Step 1:  The approximation scheme.}  Let $u^1$ and $u^2$ be two pathwise kinetic solutions corresponding to initial data $u_0^1,u_0^2\in L^1(\mathbb{T}^d)$.  We will write $\chi^1$ and $\chi^2$ for the corresponding kinetic functions, and $p^1, p^2$ and $q^1, q^2$ respectively for the entropy and parabolic defect measures.  In order to simplify the notation in what follows, for each $j\in\{1,2\}$ and for each $(x,\xi,t)\in\mathbb{T}^d\times\mathbb{R}\times[0,T]$, we will write
$$\chi^j_r(x,\xi):=\chi^j(x,\xi,r),\;\;p^j_r(x,\xi):=p^j(x,\xi,r),\;\;\textrm{and}\;\;q^j_r(x,\xi)=q^j(x,\xi,r).$$

The argument will proceed via a time-splitting argument that is made possible by the conservative structure of the equation and, in particular, equation (\ref{kin_measure}), which asserts that characteristics preserve the Lebesgue measure.  Let $\mathcal{N}^1$ and $\mathcal{N}^2$ denote the zero sets corresponding to $u^1$ and $u^2$ respectively, and define $\mathcal{N}=\mathcal{N}^1\cup\mathcal{N}^2$.  Let $T\in\left([0,\infty)\setminus\mathcal{N}\right)$ be arbitrary and fix a partition $\mathcal{P}\subset\left([0,T]\setminus \mathcal{N}\right)$,
$$\mathcal{P}:=\left\{0=t_0<t_1<\ldots<t_{N-1}<t_N=T\right\}.$$
For each $i\in\{0,\ldots,N-1\}$, we will write
$$\tilde{\chi}_{t_i,t}(x,\xi):=\chi_t(X^{x,\xi}_{t_i,t}, \Xi^{x,\xi}_{t_i,t})\;\;\textrm{for each}\;\;(x,\xi,t)\in\mathbb{T}^d\times\mathbb{R}\times[t_i,\infty),$$
where $\left(X^{x,\xi}_{t_i,t},\Xi^{x,\xi}_{t_i,t}\right)$ denote the solution of the translated characteristic equation beginning from $t_i\geq 0$ and $(x,\xi)\in\mathbb{T}^d\times\mathbb{R}$.

It is then immediate from (\ref{kin_measure}) and properties of the kinetic function that
\begin{equation}\label{u_0}\begin{aligned} & \left.\int_{\mathbb{R}}\int_{\mathbb{T}^d}\abs{\chi^1_r-\chi^2_r}^2\dy\deta\right|_{r=0}^T \\  & = \sum_{i=0}^{N-1}\left.\int_\mathbb{R}\int_{\mathbb{T}^d}\abs{\chi^1_r-\chi^2_r}^2\dy\deta\right|_{r=t_i}^{t_{i+1}} \\ & = \left.\sum_{i=0}^{N-1}\int_\mathbb{R}\int_{\mathbb{T}^d}\left(\abs{\chi^1_r}+\abs{\chi^2_r}-2\chi^1_r\chi^2_r\right)\dy\deta \right|_{r=t_i}^{t_{i+1}} \\ & = \left.\sum_{i=0}^{N-1}\int_\mathbb{R}\int_{\mathbb{T}^d}\left(\abs{\tilde{\chi}^1_{t_i,r}}+\abs{\tilde{\chi}^2_{t_i,r}}-2\tilde{\chi}^1_{t_i,r}\tilde{\chi}^2_{t_i,r}\right)\dy\deta \right|_{r=t_i}^{t_{i+1}}  \\ & = \left.\sum_{i=0}^{N-1}\lim_{\epsilon\rightarrow 0} \int_\mathbb{R}\int_{\mathbb{T}^d}\left(\tilde{\chi}^{1,\epsilon}_{t_i,r}\tilde{\sgn}^\epsilon_{t_i,r}+\tilde{\chi}^{2,\epsilon}_{t_i,r}\tilde{\sgn}^\epsilon_{t_i,r}-2\tilde{\chi}^{1,\epsilon}_{t_i,r}\tilde{\chi}^{2,\epsilon}_{t_i,r}\right)\dy\deta \right|_{r=t_i}^{t_{i+1}},\end{aligned}\end{equation}
where, for each $\epsilon\in(0,1)$, $i\in\{0,\ldots,N\}$, and $r\in\{t_i, t_{i+1}\}$, for standard convolution kernels $\rho^{d,\epsilon}$ of scale $\epsilon$ on $\mathbb{T}^d$ and $\rho^{1,\epsilon}$ of scale $\epsilon$ on $\mathbb{R}$,
$$\tilde{\chi}^{j,\epsilon}_{t_i,r}(y,\eta):=(\tilde{\chi}^j_{t_i,r}*\rho^{d,\epsilon}\rho^{1,\epsilon})(y,\eta)=\int_\mathbb{R}\int_{\mathbb{T}^d}\chi^j_r(X^{x,\xi}_{t_i,r},\Xi^{x,\xi}_{t_i,r})\rho^{d,\epsilon}(x-y)\rho^{1,\epsilon}(\xi-\eta)\dx\dxi,$$
and
$$\tilde{\sgn}^\epsilon_{t_i,r}(y,\eta):=(\tilde{\sgn}_{t_i,r}*\rho^{d,\epsilon}\rho^{1,\epsilon})(y,\eta)=\int_\mathbb{R}\int_{\mathbb{T}^d}\sgn(\Xi^{x,\xi}_{t_i,r})\rho^{d,\epsilon}(x-y)\rho^{1,\epsilon}(\xi-\eta)\dx\dxi.$$

In particular, in view of the inverse relationship (\ref{kin_inverse}) and the conservative property of the characteristics (\ref{kin_measure}), it follows that, for each $j\in\{1,2\}$,
\begin{equation}\label{u_1}\tilde{\chi}^{j,\epsilon}_{t_i,r}(y,\eta)=\int_\mathbb{R}\int_{\mathbb{T}^d}\chi^j_r(x,\xi)\rho^{d,\epsilon}(Y^{x,\xi}_{r,r-t_i}-y)\rho^{1,\epsilon}(\Pi^{x,\xi}_{r,r-t_i}-\eta)\dx\dxi,\end{equation}
where, returning to (\ref{kinetic_rep_1}), the function
\begin{equation}\label{u_2}\rho_{t_i,r}^\epsilon(x,y,\xi,\eta):=\rho^{d,\epsilon}(Y^{x,\xi}_{r,r-t_i}-y)\rho^{1,\epsilon}(\Pi^{x,\xi}_{r,r-t_i}-\eta)\;\;\textrm{for}\;\;(x,y,\xi,\eta,r)\in\mathbb{T}^{2d}\times\mathbb{R}^2\times[t_i,\infty),\end{equation}
is the solution of (\ref{kin_translate}) beginning from time $t_i\geq 0$ with initial data $\rho^{d,\epsilon}(\cdot-y)\rho^{1,\epsilon}(\cdot-\eta)$.  Also, since \eqref{kin_sign} proved that the velocity characteristics preserve the sign of $\xi$, the same computation proves that
\begin{equation}\label{u_002} \tilde{\sgn}^\epsilon_{t_i,r}(y,\eta)=\int_\mathbb{R}\int_{\mathbb{T}^d}\sgn(\xi)\rho^\epsilon_{t_i,r}(x,y,\xi,\eta)\dx\dxi=\int_\mathbb{R}\int_{\mathbb{T}^d}\sgn(\xi)\rho^{d,\epsilon}(x-y)\rho^{1,\epsilon}(\xi-\eta)\dx\dxi.\end{equation}
Observe that, while it is immediate from (\ref{u_002}) that the regularization of the $\sgn$ function is constant in time, independent of $y\in\mathbb{R}^d$, and independent of $i\in\{1,\ldots,N-1\}$, it will nevertheless be useful to consider the regularized and transported expression, since it will clarify an important cancellation property of the equation in the arguments to follow.

In what follows, let $i\in\{1,\ldots,N-1\}$ and $\epsilon\in(0,1)$ be arbitrary.  The following steps will estimate the difference
\begin{equation}\label{u_03}\left.\int_\mathbb{R}\int_{\mathbb{T}^d}\left(\tilde{\chi}^{1,\epsilon}_{t_i,r}\tilde{\sgn}^\epsilon_{t_i,r}+\tilde{\chi}^{2,\epsilon}_{t_i,r}\tilde{\sgn}^\epsilon_{t_i,r}-2\tilde{\chi}^{1,\epsilon}_{t_i,r}\tilde{\chi}^{2,\epsilon}_{t_i,r}\right)\dy\deta\right|_{r=t_i}^{t_{i+1}},\end{equation}
by considering first the terms involving the $\sgn$ function, and second the mixed term.

\textbf{Step 2:  The $\sgn$ terms.}  We will first analyze the terms involving the $\sgn$ function in (\ref{u_03}).  For the convolution kernel \eqref{u_2}, we will write $(x,\xi)\in\mathbb{T}^d\times\mathbb{R}$ for the integration variables defining $\tilde{\chi}^{1,\epsilon}_{t_i,r}$ and we will write $\rho^{1,\epsilon}_{t_i,r}$ for the corresponding convolution kernel.  We will write $(x',\xi')\in\mathbb{T}^d\times\mathbb{R}$ for the integration variables defining $\tilde{\sgn}^{\epsilon}_{t_i,r}$ and $\rho^{2,\epsilon}_{t_i,r}$ for the corresponding convolution kernel.

The equation and (\ref{u_002}) imply that, with the notation from (\ref{u_1}) and (\ref{u_2}),
\begin{equation}\begin{aligned}\label{u_3}& \left.\int_\mathbb{R}\int_{\mathbb{T}^d}\tilde{\chi}^{1,\epsilon}_{t_i,r}(y,\eta)\tilde{\sgn}^\epsilon_{t_i,r}\dy\deta\right|_{r=t_i}^{t_{i+1}} \\ & =\int_{t_i}^{t_{i+1}}\int_\mathbb{R}\int_{\mathbb{T}^d}\left(\int_\mathbb{R}\int_{\mathbb{T}^d}m\abs{\xi}^{m-1}\chi^1_r\Delta_x\rho_{t_i,r}^{1,\epsilon}\dx\dxi\right)\tilde{\sgn}^\epsilon_{t_i,r}\dy\deta\dr \\ & \quad -\int_{t_i}^{t_{i+1}}\int_\mathbb{R}\int_{\mathbb{T}^d}\left(\int_\mathbb{R}\int_{\mathbb{T}^d}(p^1_r+q^1_r)\partial_\xi\rho_{t_i,r}^{1,\epsilon}\dx\dxi\right)\tilde{\sgn}^\epsilon_{t_i,r}\dy\deta\dr.\end{aligned}\end{equation}
The first and second terms of (\ref{u_3}) will be handled separately.  Observe that, from (\ref{u_2}), for each $(x,y,\xi,\eta,r)\in\mathbb{R}^{2d+2}\times[t_i,\infty)$,
\begin{equation}\label{u_4}\nabla_x\rho_{t_i,r}^{1,\epsilon}(x,y,\xi,\eta)=-\nabla_y\rho_{t_i,r}^{1,\epsilon}(x,y,\xi,\eta)\cdot\nabla_xY^{x,\xi}_{r,r-t_i}-\partial_\eta\rho_{t_i,r}^{1,\epsilon}(x,y,\xi,\eta)\nabla_x\Pi^{x,\xi}_{r,r-t_i},\end{equation}
and
\begin{equation}\label{u_5}\partial_\xi\rho_{t_i,r}^{1,\epsilon}(x,y,\xi,\eta)=-\nabla_y\rho^{1,\epsilon}_{t_i,r}(x,y,\xi,\eta)\partial_\xi Y^{x,\xi}_{r,r-t_i}-\partial_\eta\rho_{t_i,r}^{1,\epsilon}(x,y,\xi,\eta)\partial_\xi\Pi^{x,\xi}_{r,r-t_i}.\end{equation}
For the first term of (\ref{u_3}), it is then immediate from (\ref{u_4}) that
{\color{black}\begin{equation}\begin{aligned} \label{u_6} & \int_{t_i}^{t_{i+1}}\int_\mathbb{R}\int_{\mathbb{T}^d}\left(\int_\mathbb{R}\int_{\mathbb{T}^d}m\abs{\xi}^{m-1}\chi^1_r\Delta_x\rho_{t_i,r}^{1,\epsilon}\dx\dxi\right)\tilde{\sgn}^\epsilon_{t_i,r}\dy\deta\dr \\ & =\int_{t_i}^{t_{i+1}}\int_\mathbb{R}\int_{\mathbb{T}^d}\left(\int_\mathbb{R}\int_{\mathbb{T}^d}m\abs{\xi}^{m-1}\chi^1\nabla_x\cdot\left(\rho^{1,\epsilon}_{t_i,r}\nabla_y\tilde{\sgn}^\epsilon_{t_i,r}\nabla_x Y^{x,\xi}_{r,r-t_i}\right)\dx\dxi\right)\dy\deta\dr \\ & \quad + \int_{t_i}^{t_{i+1}}\int_\mathbb{R}\int_{\mathbb{T}^d}\left(\int_\mathbb{R}\int_{\mathbb{T}^d}m\abs{\xi}^{m-1}\chi^1\nabla_x\cdot\left(\rho^{1,\epsilon}_{t_i,r}\partial_\eta\tilde{\sgn}^\epsilon_{t_i,r}\nabla_x\Pi^{x,\xi}_{r,r-t_i}\right)\dx\dxi\right)\dy\deta\dr,\end{aligned} \end{equation}
where this equality uses the fact that the regularization $\tilde{\sgn}^\epsilon_{t_i,r}$ is independent of $x\in\mathbb{T}^d$.}

In the case of (\ref{u_6}), it follows from the definition (\ref{u_002}) and the computation (\ref{u_4}) that, after adding and subtracting the terms $\nabla_{x'}Y^{x',\xi'}_{r,r-t_i}$ and $\nabla_{x'}\Pi^{x',\xi'}_{r,r-t_i}$,
\begin{equation}\begin{aligned}\label{new_1} & \int_{t_i}^{t_{i+1}}\int_\mathbb{R}\int_{\mathbb{T}^d}\left(\int_\mathbb{R}\int_{\mathbb{T}^d}m\abs{\xi}^{m-1}\chi^1_r\Delta_x\rho_{t_i,r}^{1,\epsilon}\dx\dxi\right)\tilde{\sgn}^\epsilon_{t_i,r}\dy\deta\dr \\ & =\textrm{Err}^{0,1}_i  - \int_{t_i}^{t_{i+1}}\int_{\mathbb{R}^3}\int_{\mathbb{T}^{3d}}m\abs{\xi}^{m-1}\chi^1_r\nabla_x\rho^{1,\epsilon}_{t_i,r}\sgn(\xi')\nabla_{x'}\rho^{2,\epsilon}_{t_i,r}\dx\dxi\dxp\dxip\dy\deta\dr,\end{aligned}\end{equation}
for the error term
{\color{black}\begin{equation}\begin{aligned}\label{new_2}  \textrm{Err}^{0,1}_i:= & \int_{t_i}^{t_{i+1}}\int_{\mathbb{R}^3}\int_{\mathbb{T}^{3d}}m\abs{\xi}^{m-1}\chi^1\nabla_x\cdot\left(\rho^{1,\epsilon}_{t_i,r}\sgn(\xi')\nabla_y\rho^{2,\epsilon}_{t_i,r}\left(\nabla_x Y^{x,\xi}_{r,r-t_i}-\nabla_{x'} Y^{x',\xi'}_{r,r-t_i}\right)\right) \\ & + \int_{t_i}^{t_{i+1}}\int_{\mathbb{R}^3}\int_{\mathbb{T}^{3d}}m\abs{\xi}^{m-1}\chi^1\nabla_x\cdot\left(\rho^{1,\epsilon}_{t_i,r}\sgn(\xi')\partial_\eta\rho^{2,\epsilon}_{t_i,r}\left(\nabla_x\Pi^{x,\xi}_{r,r-t_i}-\nabla_{x'}\Pi^{x',\xi'}_{r,r-t_i}\right)\right),\end{aligned}\end{equation}}
and where the last term of (\ref{new_1}) vanishes after integrating by parts in the $x'$-variable.  That is,
\begin{equation}\label{new_200}\int_{t_i}^{t_{i+1}}\int_{\mathbb{R}^3}\int_{\mathbb{T}^{3d}}m\abs{\xi}^{m-1}\chi^1_r\nabla_x\rho^{1,\epsilon}_{t_i,r}\sgn(\xi')\nabla_{x'}\rho^{2,\epsilon}_{t_i,r}\dx\dxi\dxp\dxip\dy\deta\dr=0.\end{equation}

For the second term of (\ref{u_3}), it follows from (\ref{u_5}) that
\begin{equation}\begin{aligned}\label{new_4} & \int_{t_i}^{t_{i+1}}\int_\mathbb{R}\int_{\mathbb{T}^d}\left(\int_\mathbb{R}\int_{\mathbb{T}^d}(p^1_r+q^1_r)\partial_\xi\rho_{t_i,r}^{1,\epsilon}\dx\dxi\right)\tilde{\sgn}^\epsilon_{t_i,r}\dy\deta\dr \\ & =\int_{t_i}^{t_{i+1}}\int_\mathbb{R}\int_{\mathbb{T}^d}\left(\int_\mathbb{R}\int_{\mathbb{T}^d}(p^1_r+q^1_r)\rho_{t_i,r}^{1,\epsilon}\partial_\xi Y^{x,\xi}_{r,r-t_i}\dx\dxi\right)\cdot\nabla_y\tilde{\sgn}^\epsilon_{t_i,r}\dy\deta\dr \\ & \quad +\int_{t_i}^{t_{i+1}}\int_\mathbb{R}\int_{\mathbb{T}^d}\left(\int_\mathbb{R}\int_{\mathbb{T}^d}(p^1_r+q^1_r)\rho_{t_i,r}^{1,\epsilon}\partial_\xi\Pi^{x,\xi}_{r,r-t_i}\dx\dxi\right)\partial_\eta\tilde{\sgn}^\epsilon_{t_i,r}\dy\deta\dr.\end{aligned}\end{equation}

In the case of (\ref{new_4}), it follows from the representation (\ref{u_002}) and the computation (\ref{u_5}) that, after adding and subtracting the derivatives $\partial_{\xi'}Y^{x',\xi'}_{r,r-t_i}$ and $\partial_{\xi'}\Pi^{x',\xi'}_{r,r-t_i}$,
\begin{equation}\begin{aligned}\label{new_5} & \int_{t_i}^{t_{i+1}}\int_\mathbb{R}\int_{\mathbb{T}^d}\left(\int_\mathbb{R}\int_{\mathbb{T}^d}(p^1_r+q^1_r)\partial_\xi\rho_{t_i,r}^{1,\epsilon}\dx\dxi\right)\tilde{\sgn}^\epsilon_{t_i,r}\dy\deta\dr \\ & = \textrm{Err}^{1,1}_i  - \int_{t_i}^{t_{i+1}}\int_{\mathbb{R}^3}\int_{\mathbb{T}^{3d}}(p^1_r+q^1_r)\rho_{t_i,r}^{1,\epsilon}\sgn(\xi')\partial_{\xi'}\rho^{2,\epsilon}_{t_i,r}\dx\dxi\dxp\dxip\dy\deta\dr,\end{aligned}\end{equation}
for the error term
{\color{black}\begin{equation}\begin{aligned}\label{new_6}  \textrm{Err}^{1,1}_i:= &  \int_{t_i}^{t_{i+1}}\int_{\mathbb{R}^3}\int_{\mathbb{T}^{3d}}(p^1_r+q^1_r)\rho_{t_i,r}^{1,\epsilon}\sgn(\xi')\nabla_y\rho^{2,\epsilon}_{t_i,r}\cdot\left(\partial_\xi Y^{x,\xi}_{r,r-t_i}-\partial_{\xi'}Y^{x',\xi'}_{r,r-t_i}\right) \\ & + \int_{t_i}^{t_{i+1}}\int_{\mathbb{R}^3}\int_{\mathbb{T}^{3d}}(p^1_r+q^1_r)\rho_{t_i,r}^{1,\epsilon}\sgn(\xi')\partial_\eta\rho^{2,\epsilon}_{t_i,r}\left(\partial_\xi \Pi^{x,\xi}_{r,r-t_i}-\partial_{\xi'}\Pi^{x',\xi'}_{r,r-t_i}\right). \end{aligned}\end{equation}}
Additionally, after integrating by parts in the $\xi'$-variable and using the distributional equality $\partial_{\xi'}\sgn(\xi')=2\delta_0(\xi')$, the second term of (\ref{new_5}) becomes
\begin{equation}\begin{aligned}\label{new_7} & -\int_{t_i}^{t_{i+1}}\int_{\mathbb{R}^3}\int_{\mathbb{T}^{3d}}(p^1_r+q^1_r)\rho_{t_i,r}^{1,\epsilon}\sgn(\xi')\partial_{\xi'}\rho^{2,\epsilon}_{t_i,r}\dx\dxi\dxp\dxip\dy\deta\dr \\  & =2\int_{t_i}^{t_{i+1}}\int_{\mathbb{R}^2}\int_{\mathbb{T}^{3d}}(p^1_r+q^1_r)\rho_{t_i,r}^{1,\epsilon}\rho^{2,\epsilon}_{t_i,r}(x',y,0,\eta)\dx\dxi\dxp\dy\deta\dr.\end{aligned} \end{equation}

Returning to (\ref{u_3}), it follows from (\ref{u_6}), (\ref{new_1}), (\ref{new_200}), (\ref{new_5}) and (\ref{new_7}) that
\begin{equation}\begin{aligned}\label{new_9}& \left.\int_\mathbb{R}\int_{\mathbb{T}^d}\tilde{\chi}^{1,\epsilon}_{t_i,r}(y,\eta)\tilde{\sgn}^\epsilon_{t_i,r}\dy\deta\right|_{r=t_i}^{t_{i+1}}  & \\ & =\textrm{Err}^{0,1}_i-\textrm{Err}^{1,1}_i - 2\int_{t_i}^{t_{i+1}}\int_{\mathbb{R}^2}\int_{\mathbb{T}^{3d}}(p^1_r+q^1_r)\rho_{t_i,r}^{1,\epsilon}\rho^{2,\epsilon}_{t_i,r}(x',y,0,\eta)\dx\dxi\dxp\dy\deta\dr.\end{aligned}\end{equation}
Furthermore, the identical considerations with $\chi^1$ replaced by $\chi^2$ prove that, after swapping the roles of $(x,\xi)$ and $(x',\xi')$,
\begin{equation}\begin{aligned}\label{new_10}& \left.\int_\mathbb{R}\int_{\mathbb{T}^d}\tilde{\chi}^{2,\epsilon}_{t_i,r}(y,\eta)\tilde{\sgn}^\epsilon_{t_i,r}\dy\deta\right|_{r=t_i}^{t_{i+1}} \\ & =\textrm{Err}^{0,2}_i-\textrm{Err}^{1,2}_i - 2\int_{t_i}^{t_{i+1}}\int_{\mathbb{R}^2}\int_{\mathbb{T}^{3d}}(p^2_r+q^2_r)\rho_{t_i,r}^{2,\epsilon}\rho^{1,\epsilon}_{t_i,r}(x,y,0,\eta)\dxp\dxip\dx\dy\deta\dr,\end{aligned}\end{equation}
for error terms $\textrm{Err}^{0,2}_i$ and $\textrm{Err}^{1,2}_i$ defined in exact analogy with \eqref{new_2} and \eqref{new_6} with $\chi^1$ replaced by $\chi^2$.  This completes the initial analysis of the $\sgn$ terms.

\textbf{Step 3:  The mixed term.}  We will now analyze the mixed term appearing in (\ref{u_03}).  For the convolution kernel \eqref{u_2}, we will write $(x,\xi)\in\mathbb{T}^d\times\mathbb{R}$ for the integration variables defining $\tilde{\chi}^{1,\epsilon}_{t_i,r}$ and we will write $\rho^{1,\epsilon}_{t_i,r}$ for the corresponding convolution kernel.  We will write $(x',\xi')\in\mathbb{T}^d\times\mathbb{R}$ for the integration variables defining $\tilde{\chi}^{2,\epsilon}_{t_i,r}$ and $\rho^{2,\epsilon}_{t_i,r}$ for the corresponding convolution kernel.

The equation implies that
\begin{equation}\begin{aligned}\label{u_8}& \left.\int_\mathbb{R}\int_{\mathbb{T}^d}\tilde{\chi}^{1,\epsilon}_{t_i,r}\tilde{\chi}^{2,\epsilon}_{t_i,r}\dy\deta\right|_{r=t_i}^{t_{i+1}} \\ & =  \int_{t_i}^{t_{i+1}}\int_\mathbb{R}\int_{\mathbb{T}^d}\left(\int_\mathbb{R}\int_{\mathbb{T}^d}m\abs{\xi}^{m-1}\chi^1_r\Delta_x\rho_{t_i,r}^{1,\epsilon}\dx\dxi\right)\tilde{\chi}^{2,\epsilon}_{t_i,r}\dy\deta\dr \\ & \quad - \int_{t_i}^{t_{i+1}}\int_\mathbb{R}\int_{\mathbb{T}^d}\left(\int_\mathbb{R}\int_{\mathbb{T}^d}(p^1_r+q^1_r)\partial_\xi \rho_{t_i,r}^{1,\epsilon}\dx\dxi\right)\tilde{\chi}^{2,\epsilon}_{t_i,r}\dy\deta\dr \\ & \quad + \int_{t_i}^{t_{i+1}}\int_\mathbb{R}\int_{\mathbb{T}^d}\left(\int_\mathbb{R}\int_{\mathbb{T}^d}m\abs{\xi'}^{m-1}\chi^2_r\Delta_{x'}\rho_{t_i,r}^{2,\epsilon}\dxp\dxip\right)\tilde{\chi}^{1,\epsilon}_{t_i,r}\dy\deta\dr \\ & \quad - \int_{t_i}^{t_{i+1}}\int_\mathbb{R}\int_{\mathbb{T}^d}\left(\int_\mathbb{R}\int_{\mathbb{T}^d}(p^2_r+q^2_r)\partial_{\xi'} \rho_{t_i,r}^{2,\epsilon}\dxp\dxip\right)\tilde{\chi}^{1,\epsilon}_{t_i,r}\dy\deta\dr. \end{aligned}\end{equation}
We will begin by analyzing the first term of (\ref{u_8}).  It is an immediate consequence of the computation (\ref{u_4}) that
\begin{equation*}\begin{aligned} & \int_{t_i}^{t_{i+1}}\int_\mathbb{R}\int_{\mathbb{T}^d}\left(\int_\mathbb{R}\int_{\mathbb{T}^d}m\abs{\xi}^{m-1}\chi^1_r\Delta_x\rho_{t_i,r}^{1,\epsilon}\dx\dxi\right)\tilde{\chi}^{2,\epsilon}_{t_i,r}\dy\deta\dr  \\  & =\int_{t_i}^{t_{i+1}}\int_\mathbb{R}\int_{\mathbb{T}^d}\left(\int_\mathbb{R}\int_{\mathbb{T}^d}m\abs{\xi}^{m-1}\chi^1_r\nabla_x\cdot\left(\rho_{t_i,r}^{1,\epsilon}\nabla_y\tilde{\chi}^{2,\epsilon}_{t_i,r}\nabla_x Y^{x,\xi}_{r,r-t_i}\right)\dx\dxi\right)\dy\deta\dr \\ & \quad + \int_{t_i}^{t_{i+1}}\int_\mathbb{R}\int_{\mathbb{T}^d}\left(\int_\mathbb{R}\int_{\mathbb{T}^d}m\abs{\xi}^{m-1}\chi^1_r\nabla_x\cdot\left(\rho_{t_i,r}^{1,\epsilon}\partial_\eta\tilde{\chi}^{2,\epsilon}_{t_i,r}\nabla_x \Pi^{x,\xi}_{r,r-t_i}\right)\dx\dxi\right)\dy\deta\dr.\end{aligned}\end{equation*}
These terms will be treated by adding and subtracting the gradients $\nabla_{x'}Y^{x',\xi'}_{r,r-t_i}$ and $\nabla_{x'}\Pi^{x',\xi'}_{r,r-t_i}$.  Indeed, it follows from (\ref{u_4}) that
\begin{equation}\begin{aligned}\label{new_11} & \int_{t_i}^{t_{i+1}}\int_\mathbb{R}\int_{\mathbb{T}^d}\left(\int_\mathbb{R}\int_{\mathbb{T}^d}m\abs{\xi}^{m-1}\chi^1_r\Delta_x\rho_{t_i,r}^{1,\epsilon}\dx\dxi\right)\tilde{\chi}^{2,\epsilon}_{t_i,r}\dy\deta\dr  \\ & =\textrm{Err}^{2,1}_i  - \int_{t_i}^{t_{i+1}}\int_{\mathbb{R}^3}\int_{\mathbb{T}^{3d}}m\abs{\xi}^{m-1}\chi^1_r\chi^2_r\nabla_x\rho_{t_i,r}^{1,\epsilon}\nabla_{x'}\rho_{t_i,r}^{2,\epsilon}\dx\dxi\dxp\dxip\dy\deta\dr,\end{aligned}\end{equation}
where
{\color{black}\begin{equation}\begin{aligned}\label{u_09}  \textrm{Err}^{2,1}_i:= & \int_{t_i}^{t_{i+1}}\int_{\mathbb{R}^3}\int_{\mathbb{T}^{3d}}m\abs{\xi}^{m-1}\chi^1_r\nabla_x\cdot\left(\rho_{t_i,r}^{1,\epsilon}\chi^2_r\nabla_y\rho^{2,\epsilon}_{t_i,r}\left(\nabla_xY^{x,\xi}_{r,r-t_i}-\nabla_{x'}Y^{x',\xi'}_{r,r-t_i}\right)\right) \\ & + \int_{t_i}^{t_{i+1}}\int_{\mathbb{R}^3}\int_{\mathbb{T}^{3d}}m\abs{\xi}^{m-1}\chi^1_r\nabla_x\cdot\left(\rho_{t_i,r}^{1,\epsilon}\chi^2_r\partial_\eta\rho^{2,\epsilon}_{t_i,r}\left(\nabla_x \Pi^{x,\xi}_{r,r-t_i}-\nabla_{x'}\Pi^{x',\xi'}_{r,r-t_i}\right)\right). \end{aligned}\end{equation}}
After defining $\textrm{Err}^{2,2}_i$ analogously, by swapping the roles of $\chi^1$ and $\chi^2$, the third term of (\ref{u_8}) can be treated similarly.  That is,
\begin{equation}\begin{aligned}\label{new_12} & \int_{t_i}^{t_{i+1}}\int_\mathbb{R}\int_{\mathbb{T}^d}\left(\int_\mathbb{R}\int_{\mathbb{T}^d}m\abs{\xi'}^{m-1}\chi^2_r\Delta_x\rho_{t_i,r}^{2,\epsilon}\dxp\dxip\right)\tilde{\chi}^{1,\epsilon}_{t_i,r}\dy\deta\dr   \\ & =\textrm{Err}^{2,2}_i - \int_{t_i}^{t_{i+1}}\int_{\mathbb{R}^3}\int_{\mathbb{T}^{3d}}m\abs{\xi'}^{m-1}\chi^2_r\chi^1_r\nabla_{x'}\rho_{t_i,r}^{2,\epsilon}\nabla_x\rho_{t_i,r}^{1,\epsilon}\dxp\dxip\dx\dxi\dy\deta\dr.\end{aligned}\end{equation}

We will now treat the second and fourth terms of (\ref{u_8}).  It follows from computation (\ref{u_5}) that
\begin{equation*}\begin{aligned} & \int_{t_i}^{t_{i+1}}\int_\mathbb{R}\int_{\mathbb{T}^d}\left(\int_\mathbb{R}\int_{\mathbb{T}^d}(p^1_r+q^1_r)\partial_\xi \rho_{t_i,r}^{1,\epsilon}\dx\dxi\right)\tilde{\chi}^{2,\epsilon}_{t_i,r}\dy\deta\dr  \\ & =\int_{t_i}^{t_{i+1}}\int_\mathbb{R}\int_{\mathbb{T}^d}\left(\int_\mathbb{R}\int_{\mathbb{T}^d}(p^1_r+q^1_r)\rho^{1,\epsilon}_{t_i,r}\partial_\xi Y^{x,\xi}_{r,r-t_i}\dx\dxi\right)\cdot \nabla_y\tilde{\chi}^{2,\epsilon}_{t_i,r}\dy\deta\dr \\ & \quad + \int_{t_i}^{t_{i+1}}\int_\mathbb{R}\int_{\mathbb{T}^d}\left(\int_\mathbb{R}\int_{\mathbb{T}^d}(p^1_r+q^1_r)\rho^{1,\epsilon}_{t_i,r}\partial_\xi \Pi^{x,\xi}_{r,r-t_i}\dx\dxi\right)\partial_\eta\tilde{\chi}^{2,\epsilon}_{t_i,r}\dy\deta\dr.\end{aligned}\end{equation*}
Proceeding as before, after adding and subtracting the gradients $\partial_{\xi'}Y^{x',\xi'}_{r,r-t_i}$ and $\partial_{\xi'}\Pi^{x',\xi'}_{r,r-t_i}$, it follows from (\ref{u_5}) that
\begin{equation}\begin{aligned}\label{u_15} & \int_{t_i}^{t_{i+1}}\int_\mathbb{R}\int_{\mathbb{T}^d}\left(\int_\mathbb{R}\int_{\mathbb{T}^d}(p^1_r+q^1_r)\partial_\xi \rho_{t_i,r}^{1,\epsilon}\dx\dxi\right)\tilde{\chi}^{2,\epsilon}_{t_i,r}\dy\deta\dr  \\ & = \textrm{Err}^{3,1}_i -\int_{t_i}^{t_{i+1}}\int_{\mathbb{R}^3}\int_{\mathbb{T}^{3d}}(p^1_r+q^1_r)\rho^{1,\epsilon}_{t_i,r}\chi^2\partial_{\xi'}\rho^{2,\epsilon}_{t_i,r}\dx\dxi\dxp\dxip\dy\deta\dr,\end{aligned}\end{equation}
where
{\color{black}\begin{equation}\begin{aligned}\label{u_16}  \textrm{Err}^{3,1}_i:=& \int_{t_i}^{t_{i+1}}\int_{\mathbb{R}^3}\int_{\mathbb{T}^{3d}}(p^1_r+q^1_r)\rho^{1,\epsilon}_{t_i,r}\chi^2_r\nabla_y\rho^{2,\epsilon}_{t_i,r}\cdot\left(\partial_\xi Y^{x,\xi}_{r,r-t_i}-\partial_{\xi'}Y^{x',\xi'}_{r,r-t_i}\right) \\ & + \int_{t_i}^{t_{i+1}}\int_{\mathbb{R}^3}\int_{\mathbb{T}^{3d}}(p^1_r+q^1_r)\rho^{1,\epsilon}_{t_i,r}\chi^2_r\partial_\eta\rho^{2,\epsilon}_{t_i,r}\left(\partial_\xi \Pi^{x,\xi}_{r,r-t_i}-\partial_{\xi'}\Pi^{x',\xi'}_{r,r-t_i}\right).\end{aligned}\end{equation}}
Then, define $\textrm{Err}^{3,2}_i$ in analogy with (\ref{u_16}) by swapping the roles of $\chi^1$ and $\chi^2$, to obtain
\begin{equation}\begin{aligned}\label{u_160} & \int_{t_i}^{t_{i+1}}\int_\mathbb{R}\int_{\mathbb{T}^d}\left(\int_\mathbb{R}\int_{\mathbb{T}^d}(p^2_r+q^2_r)\partial_{\xi'} \rho_{t_i,r}^{2,\epsilon}\dxp\dxip\right)\tilde{\chi}^{1,\epsilon}_{t_i,r}\dy\deta\dr  \\ & =\textrm{Err}^{3,2}_i -\int_{t_i}^{t_{i+1}}\int_{\mathbb{R}^3}\int_{\mathbb{T}^{3d}}(p^2_r+q^2_r)\rho^{2,\epsilon}_{t_i,r}\chi^1\partial_\xi\rho^{1,\epsilon}_{t_i,r}\dxp\dxip\dx\dxi\dy\deta\dr.\end{aligned}\end{equation}

For the second term of (\ref{u_15}), the distributional equality
$$\partial_{\xi'}\chi^2(x,\xi',r)=\delta_0(\xi')-\delta_0(u^2(x',r)-\xi')\;\;\textrm{for}\;\;(x',\xi',r)\in\mathbb{T}^d\times\mathbb{R}\times[0,\infty),$$
implies that
\begin{equation}\begin{aligned}\label{u_21} & -\int_{t_i}^{t_{i+1}}\int_{\mathbb{R}^3}\int_{\mathbb{T}^{3d}}(p^1_r+q^1_r)\rho^{1,\epsilon}_{t_i,r}(x,y,\xi,\eta)\chi^2\partial_{\xi'}\rho^{2,\epsilon}_{t_i,r}(x',y,\xi',\eta)\dx\dxi\dxp\dxip\dy\deta\dr  \\ & =\int_{t_i}^{t_{i+1}}\int_{\mathbb{R}^2}\int_{\mathbb{T}^{3d}}(p^1_r+q^1_r)\rho^{1,\epsilon}_{t_i,r}\rho^{2,\epsilon}_{t_i,r}(x',y,0,\eta)\dx\dxi\dxp\dy\deta\dr \\ &  \quad -\int_{t_i}^{t_{i+1}}\int_{\mathbb{R}^2}\int_{\mathbb{T}^{3d}}(p^1_r+q^1_r)\rho^{1,\epsilon}_{t_i,r}\rho^{2,\epsilon}_{t_i,r}(x',y,u^2(x',r),\eta)\dx\dxi\dxp\dy\deta\dr.\end{aligned}\end{equation}
Hence, returning to (\ref{u_15}), it follows from (\ref{u_21}) that
\begin{equation}\begin{aligned}\label{u_22} & \int_{t_i}^{t_{i+1}}\int_\mathbb{R}\int_{\mathbb{T}^d}\left(\int_\mathbb{R}\int_{\mathbb{T}^d}(p^1_r+q^1_r)\partial_\xi \rho_{t_i,r}^{1,\epsilon}\dx\dxi\right)\tilde{\chi}^{2,\epsilon}_{t_i,r}\dy\deta\dr \\ & = \textrm{Err}^{3,1}_i +\int_{t_i}^{t_{i+1}}\int_{\mathbb{R}^2}\int_{\mathbb{T}^{3d}}(p^1_r+q^1_r)\rho^{1,\epsilon}_{t_i,r}\rho^{2,\epsilon}_{t_i,r}(x',y,0,\eta)\dx\dxi\dxp\dy\deta\dr \\ & \quad -\int_{t_i}^{t_{i+1}}\int_{\mathbb{R}^2}\int_{\mathbb{T}^{3d}}(p^1_r+q^1_r)\rho^{1,\epsilon}_{t_i,r}\rho^{2,\epsilon}_{t_i,r}(x',y,u^2(x',r),\eta)\dx\dxi\dxp\dy\deta\dr.\end{aligned}\end{equation}
Similarly, by swapping the roles of $\chi^1$ and $\chi^2$,
\begin{equation}\begin{aligned}\label{u_23} & \int_{t_i}^{t_{i+1}}\int_\mathbb{R}\int_{\mathbb{T}^d}\left(\int_\mathbb{R}\int_{\mathbb{T}^d}(p^2_r+q^2_r)\partial_{\xi'} \rho_{t_i,r}^{2,\epsilon}\dxp\dxip\right)\tilde{\chi}^{1,\epsilon}_{t_i,r}\dy\deta\dr \\ & = \textrm{Err}^{3,2}_i  +\int_{t_i}^{t_{i+1}}\int_{\mathbb{R}^2}\int_{\mathbb{T}^{3d}}(p^2_r+q^2_r)\rho^{2,\epsilon}_{t_i,r}\rho^{1,\epsilon}_{t_i,r}(x,y,0,\eta)\dxp\dxip\dx\dy\deta\dr \\ &  \quad -\int_{t_i}^{t_{i+1}}\int_{\mathbb{R}^2}\int_{\mathbb{T}^{3d}}(p^2_r+q^2_r)\rho^{2,\epsilon}_{t_i,r}\rho^{1,\epsilon}_{t_i,r}(x,y,u^1(x,r),\eta)\dxp\dxip\dx\dy\deta\dr.\end{aligned}\end{equation}

Returning to (\ref{u_8}), it follows from (\ref{new_11}), (\ref{new_12}), (\ref{u_22}), and (\ref{u_23}) that
\begin{equation}\begin{aligned} \label{u_24} & \left.\int_\mathbb{R}\int_{\mathbb{T}^d}\tilde{\chi}^{1,\epsilon}_{t_i,r}\tilde{\chi}^{2,\epsilon}_{t_i,r}\dy\deta\right|_{r=t_i}^{t_{i+1}} =  \sum_{j=1}^2\left(\textrm{Err}^{2,j}_i-\textrm{Err}^{3,j}_i\right) \\ & \quad - \int_{t_i}^{t_{i+1}}\int_{\mathbb{R}^3}\int_{\mathbb{T}^{3d}}\left(m\abs{\xi}^{m-1}+m\abs{\xi'}^{m-1}\right)\chi^1_r\chi^2_r\nabla_x\rho_{t_i,r}^{1,\epsilon}\nabla_{x'}\rho_{t_i,r}^{2,\epsilon}\dx\dxi\dxp\dxip\dy\deta\dr \\ & \quad -\int_{t_i}^{t_{i+1}}\int_{\mathbb{R}^2}\int_{\mathbb{T}^{3d}}(p^1_r+q^1_r)\rho^{1,\epsilon}_{t_i,r}\rho^{2,\epsilon}_{t_i,r}(x',y,0,\eta)\dx\dxi\dxp\dy\deta\dr \\ & \quad +\int_{t_i}^{t_{i+1}}\int_{\mathbb{R}^2}\int_{\mathbb{T}^{3d}}(p^1_r+q^1_r)\rho^{1,\epsilon}_{t_i,r}\rho^{2,\epsilon}_{t_i,r}(x',y,u^2(x',r),\eta)\dx\dxi\dxp\dy\deta\dr \\ & \quad -\int_{t_i}^{t_{i+1}}\int_{\mathbb{R}^2}\int_{\mathbb{T}^{3d}}(p^2_r+q^2_r)\rho^{2,\epsilon}_{t_i,r}\rho^{1,\epsilon}_{t_i,r}(x,y,0,\eta)\dxp\dxip\dx\dy\deta\dr \\ & \quad +\int_{t_i}^{t_{i+1}}\int_{\mathbb{R}^2}\int_{\mathbb{T}^{3d}}(p^2_r+q^2_r)\rho^{2,\epsilon}_{t_i,r}\rho^{1,\epsilon}_{t_i,r}(x,y,u^1(x,r),\eta)\dxp\dxip\dx\dy\deta\dr.\end{aligned}\end{equation}
This completes the initial analysis of the mixed term.

\textbf{Step 4:  Cancellation from the parabolic defect measures.}  In view of (\ref{new_9}), (\ref{new_10}), and (\ref{u_24}), it is now possible to return to (\ref{u_03}).  Precisely, thanks to the cancellation between the terms involving the parabolic and kinetic defect measures evaluated at zero,
\begin{equation}\begin{aligned}\label{u_25}& \left.\int_\mathbb{R}\int_{\mathbb{T}^d}\left(\tilde{\chi}^{1,\epsilon}_{t_i,r}\tilde{\sgn}^\epsilon_{t_i,r}+\tilde{\chi}^{2,\epsilon}_{t_i,r}\tilde{\sgn}^\epsilon_{t_i,r}-2\tilde{\chi}^{1,\epsilon}_{t_i,r}\tilde{\chi}^{2,\epsilon}_{t_i,r}\right)\dy\deta\right|_{r=t_i}^{t_{i+1}} \\ & =\sum_{j=1}^2\left(\textrm{Err}^{0,j}_i-\textrm{Err}^{1,j}_i+\textrm{Err}^{2,j}_i-\textrm{Err}^{3,j}_i\right) \\ & \quad +2\int_{t_i}^{t_{i+1}}\int_{\mathbb{R}^3}\int_{\mathbb{T}^{3d}}\left(m\abs{\xi}^{m-1}+m\abs{\xi'}^{m-1}\right)\chi^1_r\chi^2_r\nabla_x\rho_{t_i,r}^{1,\epsilon}\nabla_{x'}\rho_{t_i,r}^{2,\epsilon}\dx\dxi\dxp\dxip\dy\deta\dr \\ & \quad -2\int_{t_i}^{t_{i+1}}\int_{\mathbb{R}^2}\int_{\mathbb{T}^{3d}}(p^1_r+q^1_r)\rho^{1,\epsilon}_{t_i,r}\rho^{2,\epsilon}_{t_i,r}(x',y,u^2(x',r),\eta)\dx\dxi\dxp\dy\deta\dr \\ & \quad -2\int_{t_i}^{t_{i+1}}\int_{\mathbb{R}^2}\int_{\mathbb{T}^{3d}}(p^2_r+q^2_r)\rho^{2,\epsilon}_{t_i,r}\rho^{1,\epsilon}_{t_i,r}(x,y,u^1(x,r),\eta)\dxp\dxip\dx\dy\deta\dr.\end{aligned}\end{equation}
In order to see the additional cancellation coming from the parabolic defect measures, which will require an application of the integration by parts formula of Lemma~\ref{lem_ibp}, we will use the equality
$$\left(\abs{\xi}^{\frac{m-1}{2}}-\abs{\xi'}^{\frac{m-1}{2}}\right)^2+2\abs{\xi}^{\frac{m-1}{2}}\abs{\xi'}^{\frac{m-1}{2}}=\abs{\xi}^{m-1}+\abs{\xi'}^{m-1}\;\;\textrm{for}\;\;\xi,\xi'\in\mathbb{R}.$$
This implies that
\begin{equation}\begin{aligned} \label{u_26} & 2\int_{t_i}^{t_{i+1}}\int_{\mathbb{R}^3}\int_{\mathbb{T}^{3d}}\left(m\abs{\xi}^{m-1}+m\abs{\xi'}^{m-1}\right)\chi^1_r\chi^2_r\nabla_x\rho_{t_i,r}^{1,\epsilon}\nabla_{x'}\rho_{t_i,r}^{2,\epsilon}\dx\dxi\dxp\dxip\dy\deta\dr \\ & = 4m\int_{t_i}^{t_{i+1}}\int_{\mathbb{R}^3}\int_{\mathbb{T}^{3d}}\abs{\xi}^{\frac{m-1}{2}}\abs{\xi'}^{\frac{m-1}{2}}\chi^1_r\chi^2_r\nabla_x\rho_{t_i,r}^{1,\epsilon}\nabla_{x'}\rho_{t_i,r}^{2,\epsilon}\dx\dxi\dxp\dxip\dy\deta\dr \\ & \quad + 2m\int_{t_i}^{t_{i+1}}\int_{\mathbb{R}^3}\int_{\mathbb{T}^{3d}}\left(\abs{\xi}^{\frac{m-1}{2}}-\abs{\xi'}^{\frac{m-1}{2}}\right)^2\chi^1_r\chi^2_r\nabla_x\rho_{t_i,r}^{1,\epsilon}\nabla_{x'}\rho_{t_i,r}^{2,\epsilon}\dx\dxi\dxp\dxip\dy\deta\dr.\end{aligned}\end{equation}

For the first term on the righthand side of \eqref{u_26}, after applying the integration by parts formula in the $x$-variable and $x'$-variable,
$$\begin{aligned} & 4m\int_{t_i}^{t_{i+1}}\int_{\mathbb{R}^3}\int_{\mathbb{T}^{3d}}\abs{\xi}^\frac{m-1}{2}\abs{\xi'}^\frac{m-1}{2}\chi^1_r\chi^2_r\nabla_x\rho_{t_i,r}^{1,\epsilon}\nabla_{x'}\rho_{t_i,r}^{2,\epsilon}  \\ & = \frac{16m}{(m+1)^2}\int_{t_i}^{t_{i+1}}\int_\mathbb{R}\int_{\mathbb{T}^{3d}}\nabla(u^1)^{[\frac{m+1}{2}]}\cdot\nabla(u^2)^{[\frac{m+1}{2}]}\rho_{t_i,r}^{1,\epsilon}(x,y,u^1(x,r),\eta)\rho_{t_i,r}^{2,\epsilon}(x',y,u^2(x',r),\eta).\end{aligned}$$
It therefore follows from an application of H\"older's inequality and Young's inequality, the definition of the parabolic defect measure, and the nonnegativity of the entropy defect measure that
\begin{equation}\begin{aligned}\label{u_30} & 4m\int_{t_i}^{t_{i+1}}\int_{\mathbb{R}^3}\int_{\mathbb{T}^{3d}}\abs{\xi}^\frac{m-1}{2}\abs{\xi'}^\frac{m-1}{2}\chi^1_r\chi^2_r\nabla_x\rho_{t_i,r}^\epsilon\nabla_{x'}\rho_{t_i,r}^\epsilon\dx\dxi\dxp\dxip\dy\deta\dr \\ & \leq 2\int_{t_i}^{t_{i+1}}\int_{\mathbb{R}^2}\int_{\mathbb{T}^{3d}}(p^1_r+q^1_r)\rho^{1,\epsilon}_{t_i,r}\rho^{2,\epsilon}_{t_i,r}(x',y,u^2(x',r),\eta)\dx\dxi\dxp\dy\deta\dr \\ & \quad + 2\int_{t_i}^{t_{i+1}}\int_{\mathbb{R}^2}\int_{\mathbb{T}^{3d}}(p^2_r+q^2_r)\rho^{2,\epsilon}_{t_i,r}\rho^{1,\epsilon}_{t_i,r}(x,y,u^1(x,r),\eta)\dxp\dxip\dx\dy\deta\dr.\end{aligned}\end{equation}

Therefore, returning to \eqref{u_25}, it follows from \eqref{u_26} and \eqref{u_30} that
\begin{equation}\begin{aligned} \label{u_31} & \left.\int_\mathbb{R}\int_{\mathbb{T}^d}\left(\tilde{\chi}^{1,\epsilon}_{t_i,r}\tilde{\sgn}^\epsilon_{t_i,r}+\tilde{\chi}^{2,\epsilon}_{t_i,r}\tilde{\sgn}^\epsilon_{t_i,r}-2\tilde{\chi}^{1,\epsilon}_{t_i,r}\tilde{\chi}^{2,\epsilon}_{t_i,r}\right)\dy\deta\right|_{r=t_i}^{t_{i+1}} \\ & \leq \sum_{j=1}^2\left(\textrm{Err}^{0,j}_i-\textrm{Err}^{1,j}_i+\textrm{Err}^{2,j}_i-\textrm{Err}^{3,j}_i\right)+\textrm{Err}^4_i,\end{aligned}\end{equation}
where
\begin{equation}\label{u_0310} \textrm{Err}^4_i:=2\int_{t_i}^{t_{i+1}}\int_{\mathbb{R}^3}\int_{\mathbb{T}^{3d}}m\left(\abs{\xi}^\frac{m-1}{2}-\abs{\xi'}^\frac{m-1}{2}\right)^2\chi^1_r\chi^2_r\nabla_x\rho_{t_i,r}^{1,\epsilon}\nabla_{x'}\rho_{t_i,r}^{2,\epsilon}\dx\dxi\dxp\dxip\dy\deta\dr.\end{equation}
It remains to analyze the error terms.

\textbf{Step 5:  The error terms.}   We will first use Proposition~\ref{rough_est} to obtain estimates for the characteristics.  Observe that, for each $(x,\xi),(x',\xi')\in\mathbb{T}^d\times\mathbb{R}$ and $r\in[t_i,t_{i+1}]$,
\begin{equation*}\begin{aligned} & \abs{x-x'}=\abs{X^{Y^{x,\xi}_{r,r-t_i},\Pi^{x,\xi}_{r,r-t_i}}_{t_i,r}-X^{Y^{x',\xi'}_{r,r-t_i},\Pi^{x',\xi'}_{r,r-t_i}}_{t_i,r}} \\ & \leq\sup_{(y,\eta)\in\mathbb{T}^d\times\mathbb{R}}\abs{\nabla_x X^{y,\eta}_{t_i,r}}\abs{Y^{x,\xi}_{r,r-t_i}-Y^{x',\xi'}_{r,r-t_i}}+\sup_{(y,\eta)\in\mathbb{T}^d\times\mathbb{R}}\abs{\partial_\eta X^{y,\eta}_{t_i,r}}\abs{\Pi^{x,\xi}_{r,r-t_i}-\Pi^{x',\xi'}_{r,r-t_i}},\end{aligned}\end{equation*}
and
\begin{equation*}\begin{aligned} & \abs{\xi-\xi'}=\abs{\Xi^{Y^{x,\xi}_{r,r-t_i},\Pi^{x,\xi}_{r,r-t_i}}_{t_i,r}-\Xi^{Y^{x',\xi'}_{r,r-t_i},\Pi^{x',\xi'}_{r,r-t_i}}_{t_i,r}} \\ & \leq\sup_{(y,\eta)\in\mathbb{T}^d\times\mathbb{R}}\abs{\nabla_x \Xi^{y,\eta}_{t_i,r}}\abs{Y^{x,\xi}_{r,r-t_i}-Y^{x',\xi'}_{r,r-t_i}}+\sup_{(y,\eta)\in\mathbb{T}^d\times\mathbb{R}}\abs{\partial_\eta \Xi^{y,\eta}_{t_i,r}}\abs{\Pi^{x,\xi}_{r,r-t_i}-\Pi^{x',\xi'}_{r,r-t_i}}.\end{aligned}\end{equation*}
Therefore, assumption \eqref{prelim_regular} and Proposition~\ref{rough_est} imply that, for $C=C(T)>0$, for each $(x,\xi),(x',\xi')\in\mathbb{T}^d\times\mathbb{R}$,
\begin{equation}\label{u_9}\abs{x-x'}+\abs{\xi-\xi'}\leq C\left(\abs{Y^{x,\xi}_{r,r-t_i}-Y^{x',\xi'}_{r,r-t_i}}+\abs{\Pi^{x,\xi}_{r,r-t_i}-\Pi^{x',\xi'}_{r,r-t_i}}\right).\end{equation}

Second, it follows from properties of the convolution kernel that there exists $C=C(T)>0$ such that, for every $r\in[t_i,\infty)$ and $(x,\xi),(x',\xi'),(y,\eta)\in\mathbb{T}^d\times\mathbb{R}$,
\begin{equation*}\rho^{1,\epsilon}_{t_i,r}(x,y,\xi,\eta)\rho^{2,\epsilon}_{t_i,r}(x',y,\xi',\eta)\neq 0,\end{equation*}
implies that
\begin{equation}\label{u_10}\left(\abs{Y^{x,\xi}_{r,r-t_i}-Y^{x',\xi'}_{r,r-t_i}}+\abs{\Pi^{x,\xi}_{r,r-t_i}-\Pi^{x',\xi'}_{r,r-t_i}}\right)\leq C\epsilon.\end{equation}
Furthermore, in view of (\ref{u_9}) and Proposition~\ref{rough_est} with $k=n=2$, for $C=C(T)>0$, for each $r\in[t_i,t_{i+1}]$ and $(x,\xi),(x',\xi')\in\mathbb{T}^d\times\mathbb{R}$,
\begin{equation}\label{u_11}\begin{aligned} \abs{\nabla_xY^{x,\xi}_{r,r-t_i}-\nabla_{x'}Y^{x',\xi'}_{r,r-t_i}}\leq &  \sup_{(y,\eta)\in\mathbb{T}^d\times\mathbb{R}}\left(\abs{\nabla^2_yY^{y,\eta}_{r,r-t_i}}+\abs{\partial_\eta\nabla_yY^{y,\eta}_{r,r-t_i}}\right)\left(\abs{x-x'}+\abs{\xi-\xi'}\right) \\ \leq & C(t_{i+1}-t_i)^\alpha\left(\abs{x-x'}+\abs{\xi-\xi'}\right) \\ \leq &  C(t_{i+1}-t_i)^\alpha\left(\abs{Y^{x,\xi}_{r,r-t_i}-Y^{x',\xi'}_{r,r-t_i}}+\abs{\Pi^{x,\xi}_{r,r-t_i}-\Pi^{x',\xi'}_{r,r-t_i}}\right).\end{aligned}\end{equation}
Similarly, for $C=C(T)>0$, for each $r\in[t_i,t_{i+1}]$ and $(x,\xi),(x',\xi')\in\mathbb{T}^d\times\mathbb{R}$,
\begin{equation}\label{u_12}\begin{aligned} \abs{\nabla_x\Pi^{x,\xi}_{r,r-t_i}-\nabla_{x'}\Pi^{x',\xi'}_{r,r-t_i}}\leq &  \sup_{(y,\eta)\in\mathbb{T}^d\times\mathbb{R}}\left(\abs{\nabla^2_y\Pi^{y,\eta}_{r,r-t_i}}+\abs{\partial_\eta\nabla_y\Pi^{y,\eta}_{r,r-t_i}}\right)\left(\abs{x-x'}+\abs{\xi-\xi'}\right) \\ \leq & C(t_{i+1}-t_i)^\alpha\left(\abs{x-x'}+\abs{\xi-\xi'}\right) \\ \leq &  C(t_{i+1}-t_i)^\alpha\left(\abs{Y^{x,\xi}_{r,r-t_i}-Y^{x',\xi'}_{r,r-t_i}}+\abs{\Pi^{x,\xi}_{r,r-t_i}-\Pi^{x',\xi'}_{r,r-t_i}}\right).\end{aligned}\end{equation}
Estimates \eqref{u_10}, \eqref{u_11}, and \eqref{u_12} will be now be used to estimate the first and third error terms.

{\color{black}We observe from \eqref{new_2} that, after applying the integration by parts formula of Lemma~\ref{lem_ibp},
$$\begin{aligned}  \abs{\textrm{Err}^{0,1}_i}\leq & \frac{2m}{m+1}\int_{t_i}^{t_{i+1}}\int_{\mathbb{T}^d}\abs{u^1}^{\frac{m-1}{2}}\abs{\nabla \left(u^1\right)^{\left[\frac{m+1}{2}\right]}}\dx\dr \\ & \times \sup_{(y,\eta,r)\in\mathbb{T}^d\times\mathbb{R}\times[t_i,t_{i+1}]}\left(\abs{\nabla_y\tilde{\sgn}^\epsilon_{t_i,r}(y,\eta)}+\abs{\partial_\eta\tilde{\sgn}^\epsilon_{t_i,r}(y,\eta)}\right) \\ & \times \sup_{(x,x',\xi,\xi',r)\in\mathbb{T}^{2d}\times\mathbb{R}^2\times[t_i,t_{i+1}]}\left(\abs{\nabla_xY^{x,\xi}_{r,r-t_i}-\nabla_{x'}Y^{x',\xi'}_{r,r-t_i}}+\abs{\nabla_x\Pi^{x,\xi}_{r,r-t_i}-\nabla_{x'}\Pi^{x',\xi'}_{r,r-t_i}}\right).\end{aligned}$$
The error terms $\{\textrm{Err}^{0,j}_i\}_{j\in\{1,2\}}$ defined in \eqref{new_2} and the error terms $\{\textrm{Err}^{2,j}_i\}_{j\in\{1,2\}}$ defined in \eqref{u_09} are treated similarly.  Since there exists $C=C(T)>0$ such that, for each $(y,\eta)\in\mathbb{T}^d\times\mathbb{R}$,
\begin{equation}\label{u_010}\abs{\nabla_y\tilde{\sgn}^\epsilon_{t_i,r}(y,\eta)}+\abs{\partial_\eta\tilde{\sgn}^\epsilon_{t_i,r}(y,\eta)}+\abs{\nabla_y\tilde{\chi}^{j,\epsilon}_{t_i,r}(y,\eta)}+\abs{\partial_\eta\tilde{\chi}^{j,\epsilon}_{t_i,r}(y,\eta)}\leq \frac{C}{\epsilon},\end{equation}
it follows from the definition of the parabolic defect measures, H\"older's inequality, and Young's inequality that, with the estimates \eqref{u_10}, \eqref{u_11}, and \eqref{u_12},  for $C=C(m,T)>0$, for each $j\in\{1,2\}$,
\begin{equation}\label{new_3} \abs{\textrm{Err}^{0,j}_i}+\abs{\textrm{Err}^{2,j}_i} \leq C\abs{\mathcal{P}}^\alpha \left(\int_{t_i}^{t_{i+1}}\int_{\mathbb{T}^d}\abs{u^j}^{(m-1)\vee 0}\dx\dr+\int_{t_i}^{t_{i+1}}\int_\mathbb{R}\int_{\mathbb{T}^d}\abs{\xi}^{(m-1)\wedge 0}q^j_r\dx\dxi\dr\right).\end{equation}
The righthand side of \eqref{new_3} will be estimated in the final step of the proof using Lemma~\ref{lem_interpolate} and Proposition~\ref{aux_p} below.}

The remaining two error terms are controlled using rough path estimates virtually identical to (\ref{u_11}) and (\ref{u_12}).  Namely, for $C=C(T)>0$, for each $(x,\xi),(x',\xi')\in\mathbb{T}^d\times\mathbb{R}$ and $r\in[t_i,t_{i+1}]$, it follows from (\ref{u_9}) that
\begin{equation}\label{u_17} \begin{aligned} \abs{\partial_\xi Y^{x,\xi}_{r,r-t_i}-\partial_{\xi'}Y^{x',\xi'}_{r,r-t_i}} \leq & \sup_{(y,\eta)\in\mathbb{T}^d\times\mathbb{R}}\left(\abs{\nabla_y\partial_\eta Y^{y,\eta}_{r,r-t_i}}+\abs{\partial^2_\eta Y^{y,\eta}_{r,r-t_i}}\right)\left(\abs{x-x'}+\abs{\xi-\xi'}\right) \\ \leq & C(t_{i+1}-t_i)^\alpha \left(\abs{x-x'}+\abs{\xi-\xi'}\right) \\ \leq & C (t_{i+1}-t_i)^\alpha\left(\abs{Y^{x,\xi}_{r,r-t_i}-Y^{x',\xi'}_{r,r-t_i}}+\abs{\Pi^{x,\xi}_{r,r-t_i}-\Pi^{x',\xi'}_{r,r-t_i}}\right).\end{aligned}\end{equation}
Similarly, for $C=C(T)>0$, for each $(x,\xi),(x',\xi')\in\mathbb{T}^d\times\mathbb{R}$ and $r\in[t_i,t_{i+1}]$,
\begin{equation}\label{u_18} \begin{aligned} \abs{\partial_\xi \Pi^{x,\xi}_{r,r-t_i}-\partial_{\xi'}\Pi^{x',\xi'}_{r,r-t_i}} \leq & \sup_{(y,\eta)\in\mathbb{T}^d\times\mathbb{R}}\left(\abs{\nabla_y\partial_\eta \Pi^{y,\eta}_{r,r-t_i}}+\abs{\partial^2_\eta \Pi^{y,\eta}_{r,r-t_i}}\right)\left(\abs{x-x'}+\abs{\xi-\xi'}\right) \\ \leq & C(t_{i+1}-t_i)^\alpha \left(\abs{x-x'}+\abs{\xi-\xi'}\right) \\ \leq & C (t_{i+1}-t_i)^\alpha\left(\abs{Y^{x,\xi}_{r,r-t_i}-Y^{x',\xi'}_{r,r-t_i}}+\abs{\Pi^{x,\xi}_{r,r-t_i}-\Pi^{x',\xi'}_{r,r-t_i}}\right).\end{aligned}\end{equation}

The error terms $\{\textrm{Err}^{1,j}_i\}_{j\in\{1,2\}}$ defined in \eqref{new_6} and the error terms $\{\textrm{Err}^{3,j}_i\}_{j\in\{1,2\}}$ defined in \eqref{u_16} are treated in analogy with \eqref{new_3}.  The estimates \eqref{u_10}, \eqref{u_010}, \eqref{u_17}, and \eqref{u_18} imply that, for $C=C(T)>0$,
\begin{equation}\label{new_8} \abs{\textrm{Err}^{1,j}_i}+\abs{\textrm{Err}^{3,j}_i} \leq C(t_{i+1}-t_i)^\alpha\int_{t_i}^{t_{i+1}}\int_\mathbb{R}\int_{\mathbb{T}^d}\left(p^j_r+q^j_r\right)\dx\dxi\dr.\end{equation}
Estimates \eqref{new_3} and \eqref{new_8} complete the analysis of the first four error terms.

The analysis of the final error term $\textrm{Err}^4_i$, defined in \eqref{u_0310}, will be broken down into three cases:  $m=1$, $m\in(2,\infty)$, or $m\in(0,1)\cup(1,2]$.  The simplest of these is the case $m=1$.  Indeed, if $m=1$, then it is immediate from \eqref{u_0310} that $\textrm{Err}^4_i=0$.

\textit{Case $m\in(2,\infty)$}: We form a velocity decomposition of the integral.  For each $M>1$, let $K_M:\mathbb{R}\rightarrow[0,1]$ be a smooth function satisfying
$$K_M(\xi):=\left\{\begin{array}{ll} 1 & \textrm{if}\;\;\abs{\xi}\leq M, \\ 0 & \textrm{if}\;\;\abs{\xi}\geq M+1.\end{array}\right.$$
Then, for each $M>1$ and $\epsilon\in(0,1)$,
\begin{equation}\label{neww_1}\begin{aligned} \textrm{Err}^4_i= & 2\int_{t_i}^{t_{i+1}}\int_{\mathbb{R}^3}\int_{\mathbb{T}^{3d}}K_M(\xi)m\left(\abs{\xi}^\frac{m-1}{2}-\abs{\xi'}^\frac{m-1}{2}\right)^2\chi^1_r\chi^2_r\nabla_x\rho_{t_i,r}^{1,\epsilon}\nabla_{x'}\rho_{t_i,r}^{2,\epsilon} \\ & + 2\int_{t_i}^{t_{i+1}}\int_{\mathbb{R}^3}\int_{\mathbb{T}^{3d}}(1-K_M(\xi))m\left(\abs{\xi}^\frac{m-1}{2}-\abs{\xi'}^\frac{m-1}{2}\right)^2\chi^1_r\chi^2_r\nabla_x\rho_{t_i,r}^{1,\epsilon}\nabla_{x'}\rho_{t_i,r}^{2,\epsilon}.\end{aligned}\end{equation}
For the first term on the righthand side of \eqref{neww_1}, the local Lipschitz continuity, if $m\geq 3$, or the H\"older continuity, if $m\in(2,3)$, of the map $\xi\in\mathbb{R}\mapsto \abs{\xi}^{\frac{m-1}{2}}$, Lemma~\ref{auxlem}, observation \eqref{u_10}, and the definition of the convolution kernel imply that, for $C=C(m,T,M)>0$ and $c=c(T)>0$,
\begin{equation}\begin{aligned}\label{tu_33} & \abs{2\int_{t_i}^{t_{i+1}}\int_{\mathbb{R}^3}\int_{\mathbb{T}^{3d}}K_M(\xi)m\left(\abs{\xi}^\frac{m-1}{2}-\abs{\xi'}^\frac{m-1}{2}\right)^2\chi^1_r\chi^2_r\nabla_x\rho_{t_i,r}^{1,\epsilon}\nabla_{x'}\rho_{t_i,r}^{2,\epsilon}} \\ & \leq C\int_{t_i}^{t_{i+1}}\int_{\mathbb{R}^3}\int_{\mathbb{T}^{3d}}m\left(\abs{\xi}^{\frac{m-1}{2}}-\abs{\xi'}^{\frac{m-1}{2}}\right)^2\abs{\nabla_x\rho^{1,\epsilon}_{t_i,r}}\abs{\nabla_{x'}\rho^{2,\epsilon}_{t_i,r}} \\ &  \leq \frac{C}{\epsilon^2}\int_{t_i}^{t_{i+1}}\int_{-c\epsilon}^{c\epsilon}\abs{\xi}^{(m-1)\wedge 2}\dxi\leq C\abs{t_{i+1}-t_i}\epsilon^{(3\wedge m)-2}.\end{aligned}\end{equation}

For the second term on the righthand side of \eqref{neww_1}, we use the following inequality, which is a consequence of the mean value theorem, for each $\xi,\xi'\in\mathbb{R}$,
$$\left(\abs{\xi}^{\frac{m-1}{2}}-\abs{\xi'}^{\frac{m-1}{2}}\right)^2\leq \abs{\frac{m-1}{2}}^2\left(\abs{\xi}^{m-3}+\abs{\xi'}^{m-3}\right)\abs{\xi-\xi'}^2.$$
This implies using \eqref{u_10} and the definition of the convolution kernel that, for $C=C(m,T)>0$ and $c=c(T)>0$,
\begin{equation}\label{neww_2}\begin{aligned} & \abs{2\int_{t_i}^{t_{i+1}}\int_{\mathbb{R}^3}\int_{\mathbb{T}^{3d}}(1-K_M(\xi))m\left(\abs{\xi}^\frac{m-1}{2}-\abs{\xi'}^\frac{m-1}{2}\right)^2\chi^1_r\chi^2_r\nabla_x\rho_{t_i,r}^{1,\epsilon}\nabla_{x'}\rho_{t_i,r}^{2,\epsilon}} \\ & \leq C\int_{t_i}^{t_{i+1}}\int_{\mathbb{R}^3}\int_{\mathbb{T}^{3d}}(1-K_M(\xi))\left(\abs{\xi}^{m-3}+\abs{\xi'}^{m-3}\right)\abs{\chi^1_r}\abs{\chi^2_r}\abs{\epsilon\nabla_x\rho_{t_i,r}^{1,\epsilon}}\abs{\epsilon\nabla_{x'}\rho_{t_i,r}^{2,\epsilon}} \\ & \leq C\left( \int_{t_i}^{t_{i+1}}\int_{\left\{\abs{u^1}\geq M\right\}}\left(\abs{u^1}-M\right)_+^{m-2}+\int_{t_i}^{t_{i+1}}\int_{\left\{\abs{u^2}\geq M-c\epsilon\right\}}\left(\abs{u^2}-M+c\epsilon\right)^{m-2}_+\right).\end{aligned}\end{equation}
The interpolation estimate Lemma~\ref{lem_interpolate} below, H\"older's inequality, Proposition~\ref{aux_p} below, and the dominated convergence theorem prove that the righthand side of \eqref{neww_2} vanishes in the limit $M\rightarrow\infty$, uniformly in $\epsilon\in(0,1)$.  Therefore, \eqref{neww_1}, \eqref{tu_33}, and \eqref{neww_2} imply that, after summing over $i\in\{0,\ldots,N-1\}$ and passing first to the limit $\epsilon\rightarrow 0$ and second to the limit $M\rightarrow\infty$,
\begin{equation}\label{tu_3300}\limsup_{\epsilon\rightarrow 0}\sum_{i=0}^{N-1}\abs{\textrm{Err}^4_i}=0.\end{equation}

\textit{Case $m\in(0,1)\cup(1,2]$}:  For this case, the idea is to remove the singularity at the origin and to use the full regularity of the solution implied by Proposition~\ref{aux_log} below.  The integration by parts formula of Lemma~\ref{lem_ibp}, which is justified using an approximation argument and Proposition~\ref{aux_p} below, implies that, for each $(y,\eta)\in \mathbb{T}^d\times\mathbb{R}$,
\begin{equation}\begin{aligned}\label{fast_1}&\int_{t_i}^{t_{i+1}}\int_{\mathbb{R}^3}\int_{\mathbb{T}^{3d}} m\left(\abs{\xi}^{\frac{m-1}{2}}-\abs{\xi'}^{\frac{m-1}{2}}\right)^2\chi^1_r\chi^2_r\nabla_x\rho^{1,\epsilon}_{t_i,r}\cdot  \nabla_{x'}\rho^{2,\epsilon}_{t_i,r} \\ &= \frac{4m}{(m+1)^2}\int_{t_i}^{t_{i+1}}\int_{\mathbb{R}}\int_{\mathbb{T}^{3d}} \psi_m(u_1,u_2)\abs{u^1}^{-\frac{1}{2}}\nabla \left(u^1\right)^{\left[\frac{m+1}{2}\right]}\cdot\abs{u^2}^{-\frac{1}{2}}\nabla\left(u^2\right)^{\left[\frac{m+1}{2}\right]}\overline{\rho}^{1,\epsilon}_{t_i,r} \overline{\rho}^{2,\epsilon}_{t_i,r},\end{aligned}\end{equation}
where
\begin{equation}\label{ffast_0}\psi_m(\xi,\xi'):=\abs{\xi}^{\frac{2-m}{2}}\abs{\xi'}^{\frac{2-m}{2}}\left(\abs{\xi}^\frac{m-1}{2}-\abs{\xi'}^\frac{m-1}{2}\right)^2\;\;\textrm{for}\;\;\xi,\xi'\in\mathbb{R},\end{equation}
and, for each $j\in\{1,2\}$,
$$\overline{\rho}^{j,\epsilon}_{t_i,r}(x,y,\eta):=\rho^{j,\epsilon}_{t_i,r}(x,y,u^j(x,t),\eta)\;\;\textrm{for}\;\;(x,y,\eta,t)\in \mathbb{T}^{2d}\times\mathbb{R}\times[t_i,\infty).$$
It follows as in \eqref{u_10} that, for $C_1=C_1(T)>0$,
\begin{equation}\label{fast_2}\overline{\rho}^{1,\epsilon}_{t_i,r}\cdot \overline{\rho}^{2,\epsilon}_{t_i,r}\neq 0\;\;\textrm{implies that}\;\;\abs{u^1-u^2}\leq C_1\epsilon.\end{equation}
Observe that if $\max\{{\abs{\xi},\abs{\xi'}\}}\leq 2C_1\epsilon$, then a direct computation yields, for $C=C(T)>0$ depending on $C_1$,
\begin{equation}\label{fast_3}\psi_m(\xi,\xi')\leq\abs{\xi}^\frac{m}{2}\abs{\xi'}^\frac{2-m}{2}+2\abs{\xi}^\frac{1}{2}\abs{\xi'}^\frac{1}{2}+\abs{\xi}^\frac{2-m}{2}\abs{\xi'}^\frac{m}{2}\leq C\epsilon.\end{equation}
Conversely, without loss of generality suppose that $\abs{\xi}\geq 2C_1\epsilon$ with $\abs{\xi}\geq\abs{\xi'}$ and $\abs{\xi-\xi'}\leq C_1\epsilon$.  Then, using a Lipschitz estimate, for $C=C(m,T)>0$ depending on $C_1$,
\begin{equation}\label{fast_323}\psi_m(\xi,\xi')\leq C\abs{\xi}^{\frac{2-m}{2}}\abs{\xi'}^{\frac{2-m}{2}}\abs{\xi'}^{m-3}\epsilon^2\leq C\abs{\xi}^{\frac{2-m}{2}}\abs{\xi'}^{\frac{m-4}{2}}\epsilon^2\leq C\abs{\xi}^{-1}\epsilon^2\leq C\epsilon,\end{equation}
where the second to last inequality uses the fact that the assumptions guarantee $\abs{\xi'}\geq\frac{1}{2}\abs{\xi}$.

We will now form a velocity decomposition of the integral.  For each $\delta\in(0,1)$, let $K^\delta:\mathbb{R}\rightarrow\mathbb[0,1]$ denote a smooth cutoff function satisfying
\begin{equation}\label{ffast_1} \left\{\begin{array}{ll} K^\delta(\xi)=1 & \textrm{if}\;\;\abs{\xi}\leq \delta\;\;\textrm{or}\;\;\frac{2}{\delta}\leq\abs{\xi}, \\ K^\delta(\xi)=0 & \textrm{if}\;\;2\delta\leq\abs{\xi}\leq\frac{1}{\delta}.\end{array}\right.\end{equation}
Returning to \eqref{fast_1} consider the decomposition
\begin{equation}\begin{aligned}\label{ffast_2}  &\int_{t_i}^{t_{i+1}}\int_{\mathbb{R}^3}\int_{\mathbb{T}^{3d}} m\left(\abs{\xi}^{\frac{m-1}{2}}-\abs{\xi'}^{\frac{m-1}{2}}\right)^2\chi^1_r\chi^2_r\nabla_x\rho^{1,\epsilon}_{t_i,r}\cdot \nabla_{x'}\rho^{2,\epsilon}_{t_i,r} \\ & = \frac{4m}{(m+1)^2}\int_{t_i}^{t_{i+1}}\int_\mathbb{R}\int_{\mathbb{T}^{3d}} \psi^\delta_m(u^1,u^2)\abs{u^1}^{-\frac{1}{2}}\nabla \left(u^1\right)^{\left[\frac{m+1}{2}\right]}\cdot\abs{u^2}^{-\frac{1}{2}}\nabla\left(u^2\right)^{\left[\frac{m+1}{2}\right]}\overline{\rho}^{1,\epsilon}_{t_i,r} \overline{\rho}^{2,\epsilon}_{t_i,r} \\ & \quad +  \frac{4m}{(m+1)^2}\int_{t_i}^{t_{i+1}}\int_{\mathbb{R}}\int_{\mathbb{T}^{3d}} \tilde{\psi}^\delta_m(u^1,u^2)\abs{u^1}^{-\frac{1}{2}}\nabla \left(u^1\right)^{\left[\frac{m+1}{2}\right]}\cdot\abs{u^2}^{-\frac{1}{2}}\nabla\left(u^2\right)^{\left[\frac{m+1}{2}\right]}\overline{\rho}^{1,\epsilon}_{t_i,r}\overline{\rho}^{2,\epsilon}_{t_i,r},\end{aligned}\end{equation}
where, for each $\delta\in(0,1)$, $\psi^\delta_m,\tilde{\psi}^\delta_m:\mathbb{R}^2\rightarrow\mathbb{R}$ are defined by
\begin{equation}\label{ffast_7}\psi^\delta_m(\xi,\xi'):=\left(K^\delta(\xi)+K^\delta(\xi')-K^\delta(\xi)K^\delta(\xi')\right)\psi_m(\xi,\xi'),\end{equation}
and
\begin{equation}\label{ffast_707} \tilde{\psi}^\delta_m(\xi,\xi'):=\left(1-K^\delta(\xi)\right)\left(1-K^\delta(\xi')\right)\psi_m(\xi,\xi').\end{equation}

It follows from \eqref{ffast_0}, \eqref{ffast_1}, and the local Lipschitz continuity of the map $\xi\in\mathbb{R}\mapsto\abs{\xi}^{\frac{m-1}{2}}$ on the set $\left\{\delta\leq\abs{\xi}\leq2/\delta\right\}$ that, $C=C(m,\delta)>0$,
$$\abs{\tilde{\psi}^\delta_m(\xi,\xi')}\leq C\abs{\xi-\xi'}^2.$$
Therefore, using Proposition~\ref{aux_log} below and Young's inequality, the second term of \eqref{ffast_2} satisfies, for $C=C(m,T,\delta)>0$,
\begin{equation}\begin{aligned}\label{ffast_3} &\abs{\int_{t_i}^{t_{i+1}}\int_\mathbb{R}\int_{\mathbb{T}^{3d}} \left(\tilde{\psi}^\delta_m(u^1,u^2)\abs{u^1}^{-\frac{1}{2}}\nabla \left(u^1\right)^{\left[\frac{m+1}{2}\right]}\cdot\abs{u^2}^{-\frac{1}{2}}\nabla\left(u^2\right)^{\left[\frac{m+1}{2}\right]}\overline{\rho}^{1,\epsilon}_{t_i,r}\overline{\rho}^{2,\epsilon}_{t_i,r}\right)} \\ & \leq C\epsilon\left(\int_{t_i}^{t_{i+1}}\int_{\mathbb{R}}\int_{\mathbb{T}^d}\abs{\xi}^{-1}q^1_r(x,\xi)\right)^\frac{1}{2}\left(\int_{t_i}^{t_{i+1}}\int_\mathbb{R}\int_{\mathbb{T}^d}\abs{\xi'}^{-1}q^2_r(x,\xi')\right)^\frac{1}{2} \\ & \leq C\epsilon\sum_{j=1}^2\left(1+\norm{u_0^j}^2_{L^2(\mathbb{T}^d)}+\int_{t_i}^{t_{i+1}}\int_\mathbb{R}\int_{\mathbb{T}^d}q^j\dx\dxi\dr\right).\end{aligned}\end{equation}

For the first term of \eqref{ffast_2}, estimates \eqref{fast_2}, \eqref{fast_3}, and \eqref{fast_323} imply that, for $C=C(m,T)>0$, we have $\abs{\psi^\delta_m(u^1,u^2)}\leq C\epsilon$ whenever $\rho^{1,\epsilon}_{t_i,r}\rho^{2,\epsilon}_{t_i,r}\neq 0$.  Therefore, using definitions \eqref{ffast_7} and \eqref{ffast_707}, the fact that $\psi^\delta_m(\xi,\xi')=0$ on the set $\{\xi=\xi'\}$, and the fact that the set
$$\{u^1\neq u^2\}\subset\left(\{u^1\neq 0\}\cup\{u^2\neq 0\}\right),$$
we conclude that, for $C=C(m,T)>0$,
\begin{equation}\begin{aligned}\label{ffast_8} &\abs{\int_{t_i}^{t_{i+1}}\int_\mathbb{R}\int_{\mathbb{T}^{3d}} \left(\psi^\delta_m(u^1,u^2)\abs{u^1}^{-\frac{1}{2}}\nabla \left(u^1\right)^{\left[\frac{m+1}{2}\right]}\cdot\abs{u^2}^{-\frac{1}{2}}\nabla\left(u^2\right)^{\left[\frac{m+1}{2}\right]}\overline{\rho}^{1,\epsilon}_{t_i,r}\overline{\rho}^{2,\epsilon}_{t_i,r}\right)} \\ &\leq C\abs{\int_{t_i}^{t_{i+1}}\int_\mathbb{R}\int_{\mathbb{T}^{3d}} \epsilon^{-1}\psi^\delta_m(u^1,u^2)\left(\abs{u^1}^{-\frac{1}{2}}\nabla \left(u^1\right)^{\left[\frac{m+1}{2}\right]}\cdot\abs{u^2}^{-\frac{1}{2}}\nabla\left(u^2\right)^{\left[\frac{m+1}{2}\right]}\right)\epsilon\overline{\rho}^{1,\epsilon}_{t_i,r}\overline{\rho}^{2,\epsilon}_{t_i,r}} \\ & \leq C\left(\int_{t_i}^{t_{i+1}}\int_\mathbb{R}\int_{U^\delta}\abs{\xi}^{-1}q^1_r(x,\xi)\right)^\frac{1}{2}\left(\int_{t_i}^{t_{i+1}}\int_\mathbb{R}\int_{U^\delta}\abs{\xi'}^{-1}q^2_r(x,\xi')\right)^\frac{1}{2},\end{aligned}\end{equation}
where, for each $\delta\in(0,1)$,
\begin{equation}\label{ffast_9} U^\delta:=\bigcup_{j=1}^2\left(\left\{0<\abs{u^j}<2\delta\right\}\cup\left\{\abs{u^j}\geq 1/\delta\right\}\right).\end{equation}
Therefore, estimates \eqref{ffast_3} and \eqref{ffast_8} imply that, for each $\delta\in(0,1)$, for $C=C(m,T)>0$,
\begin{equation}\label{fast_6}  \limsup_{\epsilon\rightarrow0} \abs{\textrm{Err}^4_i}  \leq C\left(\int_{t_i}^{t_{i+1}}\int_\mathbb{R}\int_{U^\delta}\abs{\xi}^{-1}q^1_r(x,\xi)\right)^\frac{1}{2}\left(\int_{t_i}^{t_{i+1}}\int_\mathbb{R}\int_{U^\delta}\abs{\xi'}^{-1}q^2_r(x,\xi')\right)^\frac{1}{2}.\end{equation}
The dominated convergence theorem, Proposition~\ref{aux_log} below, and \eqref{ffast_9} imply that the righthand side of \eqref{fast_6} vanishes in the limit $\delta\rightarrow 0$.  Therefore, after summing over $i\in\{0,\ldots,N-1\}$, it follows that
\begin{equation}\label{tu_3400}\limsup_{\epsilon\rightarrow 0}\sum_{i=0}^{N-1}\abs{\textrm{Err}^4_i}= 0,\end{equation}
which, together with \eqref{tu_3300}, completes the analysis of the error terms.

\textbf{Step 6:  The conclusion.}  Returning to \eqref{u_31}, and recalling the approximation scheme \eqref{u_0}, estimates (\ref{new_3}), (\ref{new_8}), \eqref{tu_3300}, and \eqref{tu_3400} imply that, after summing over $i\in\{0,\ldots,N-1\}$ and passing to the limit $\epsilon\rightarrow 0$, for $C=C(m,d,T)>0$,
\begin{equation}\label{non_2}\begin{aligned} & \left.\int_\mathbb{R}\int_{\mathbb{T}^d}\abs{\chi^1_r-\chi^2_r}^2\;\dy\deta\right|_{r=0}^T \\ 
& \leq C\abs{\mathcal{P}}^\alpha \sum_{j=1}^2\left(\int_0^T\int_{\mathbb{T}^d}\abs{u^j}^{(m-1)\vee 0}\dx\dr+\int_0^T\int_\mathbb{R}\int_{\mathbb{T}^d}\abs{\xi}^{(m-1)\wedge 0}q^j_r\dx\dxi\dr\right) \\ & \quad +C\abs{\mathcal{P}}^\alpha\sum_{j=1}^2\int_0^T\int_\mathbb{R}\int_{\mathbb{T}^d}\left(p^j+q^j\right)\dx\dxi\dr.\end{aligned}\end{equation}
Lemma~\ref{lem_interpolate} and Proposition~\ref{aux_p} below imply that, for $C=C(m,d,T)>0$, for each $j\in\{1,2\}$,
{\color{black}$$\begin{aligned} & \int_0^T\int_{\mathbb{T}^d}\abs{u^j}^{(m-1)\vee 0}\dx\dr+\int_0^T\int_\mathbb{R}\int_{\mathbb{T}^d}\abs{\xi}^{(m-1)\wedge 0}q^j_r\dx\dxi\dr \\ &  \leq C\left(\norm{u_0^j}^{(m-1)\vee 0}_{L^1(\mathbb{T}^d)}+\left(\int_0^T\int_\mathbb{R}\int_{\mathbb{T}^d}q^j\dx\dxi\dr\right)^\frac{(m-1)\vee 0}{m+1}\right) \\ &\quad + C\left(\norm{u_0^j}^{(1+m)\wedge 2}_{L^{(1+m)\wedge 2}(\mathbb{T}^d)}+\norm{u_0^j}^{2m \wedge 2}_{L^1(\mathbb{T}^d)}+\left(\int_0^T\int_{\mathbb{R}}\int_{\mathbb{T}^d}q^j\dx\dxi\dr\right)^{\frac{2m}{m+1}\wedge 1}\right).\end{aligned}$$}
Therefore, after multiple applications of H\"older's inequality and Young's inequality, it follows that for $C=C(m,d,T)>0$, for each $j\in\{1,2\}$,
{\color{black}$$\begin{aligned} & \int_0^T\int_{\mathbb{T}^d}\abs{u^j}^{(m-1)\vee 0}\dx\dr+\int_0^T\int_\mathbb{R}\int_{\mathbb{T}^d}\abs{\xi}^{(m-1)\wedge 0}q^j_r\dx\dxi\dr \\ & \leq C\left(1+\norm{u^j_0}^{(m-1)\vee 0}_{L^1(\mathbb{T}^d)}+\norm{u_0^j}^2_{L^2(\mathbb{T}^d)}+\int_0^T\int_\mathbb{R}\int_{\mathbb{T}^d}q^j\dx\dxi\dr\right).\end{aligned}$$}
Therefore, applying this estimate to \eqref{non_2}, for $C=C(m,d,T)>0$,
$$\begin{aligned} & \left.\int_\mathbb{R}\int_{\mathbb{T}^d}\abs{\chi^1_r-\chi^2_r}^2\;\dy\deta\right|_{r=0}^T \\ & \leq C\abs{\mathcal{P}}^\alpha\sum_{j=1}^2\left(1+\norm{u_0^j}_{L^2(\mathbb{T}^d)}^2+\norm{u^j_0}^{(m-1)\vee 0}_{L^1(\mathbb{T}^d)}+\int_0^T\int_\mathbb{R}\int_{\mathbb{T}^d}\left(p^j+q^j\right)\dx\dxi\dr\right).\end{aligned}$$
Hence, using the definition of the kinetic function, after passing to the limit $\abs{\mathcal{P}}\rightarrow 0$, we conclude that
\begin{equation}\label{non_200}\begin{aligned} & \int_{\mathbb{T}^d}\abs{u^1(\cdot,T)-u^2(\cdot,T)}\dx=\int_\mathbb{R}\int_{\mathbb{T}^d}\abs{\chi^1(\cdot,\cdot,T)-\chi^2(\cdot,\cdot,T)}\dx\dxi \\ & \leq \int_\mathbb{R}\int_{\mathbb{T}^d}\abs{\chi^1(\cdot,\cdot,0)-\chi^2(\cdot,\cdot,0)}\dx\dxi=\int_{\mathbb{T}^d}\abs{u^1_0-u^2_0}\dx,\end{aligned}\end{equation}
which completes the proof.  \end{proof}

\begin{remark}  We observe that the argument leading from \eqref{fast_1} to \eqref{tu_3400} was the only step in the proof of Theorem~\ref{theorem_uniqueness} that relied upon the positivity of the initial data through the application of Proposition~\ref{aux_log} below.  The remaining arguments of this paper are obtained for general initial data in $L^2(\mathbb{T}^d)$.  This completes the proof of Theorem~\ref{intro_sign}.  The details for Theorem~\ref{intro_space} are similar, but require additional estimates due to the unboundedness of the domain.  The details can be found in the first version of this paper \cite{FehrmanGess1}.  \end{remark}

We conclude this section with a few auxiliary estimates.   The first, which is an immediate corollary of Theorem~\ref{theorem_uniqueness}, obtains an $L^1$-estimate for pathwise kinetic solutions.

{\color{black}\begin{cor}\label{path_L1} Let $u_0\in L^2(\mathbb{T}^d)$ and suppose that $u$ is a pathwise kinetic solution of \eqref{intro_eq} in the sense of Definition~\ref{def_solution} with initial data $u_0$.  Then,
$$\norm{u}_{L^\infty([0,\infty);L^1(\mathbb{T}^d))}\leq \norm{u_0}_{L^1(\mathbb{T}^d)}.$$
Furthermore, if $u_0\in L^2_+(\mathbb{T}^d)$, for almost every $t\in[0,\infty)$,
$$\norm{u(\cdot,t)}_{L^1(\mathbb{T}^d)}=\norm{u_0}_{L^1(\mathbb{T}^d)}.$$
\end{cor}

\begin{proof}  Let $u_0\in L^2(\mathbb{T}^d)$ be arbitrary, and let $u$ be the pathwise kinetic solution of \eqref{intro_eq} with initial data $u_0$.  Repeating the proof of Theorem~\ref{theorem_uniqueness} with $\chi^2:=0$ implies that
$$\norm{u}_{L^\infty([0,T];L^1(\mathbb{T}^d))}=\norm{u-0}_{L^\infty([0,T];L^1(\mathbb{T}^d))}\leq \norm{u_0-0}_{L^1(\mathbb{T}^d)}=\norm{u_0}_{L^1(\mathbb{T}^d)}.$$
Indeed, in the case that $\chi^2=0$, the righthand side of \eqref{u_0} is bounded, for each $i\in\{1,\ldots,N-1\}$, by the righthand side of \eqref{new_9}.  The nonnegativity of the entropy and parabolic defect measures and estimates, estimates \eqref{new_3} and \eqref{new_8}, and a repetition of the arguments leading from \eqref{non_2} to \eqref{non_200} completes the proof.

For the second claim, suppose that $u_0\in L^2_+(\mathbb{T}^d)$ and let $u$ be the pathwise kinetic solution of \eqref{intro_eq} with initial data $u_0$, kinetic function $\chi$, and exceptional set $\mathcal{N}$.  It follows by repeating the same reasoning leading from \eqref{u_3} to \eqref{new_9} with the $\sgn$ function replaced by its negative part $\sgn_-:=(\sgn\wedge 0)$  that, due to the nonnegativity of the entropy and parabolic defect measures, after passing to the limit first with respect to the regularization and second with respect to the time splitting, for each $t\in[0,\infty)\setminus\mathcal{N}$,
$$0\leq \int_\mathbb{R}\int_{\mathbb{T}^d}\chi(x,\xi,t)\sgn_-(\xi)\dx\dxi\leq \int_\mathbb{R}\int_{\mathbb{T}^d}\overline{\chi}(u_0(x),\xi)\sgn_-(\xi)\dx\dxi=0.$$
Here, the first equality follows by the definition of the kinetic function, and the final equality follows from the nonnegativity of $u_0$.  We therefore conclude that, if $u_0\in L^2_+(\mathbb{T}^d)$ then $u\geq 0$ almost everywhere on $\mathbb{T}^d\times[0,\infty).$  The final claim now follows by testing the equation with the function that is identically equal to one, and using the nonnegativity of the solution.  \end{proof}}

In the estimates to follow, we will repeatedly use the following interpolation estimate. This estimate quantifies the gain in integrability implied by the finiteness of the parabolic defect measure.

\begin{lem}\label{lem_interpolate}  For every $z\in\C^\infty(\mathbb{T}^d)$, for $C=C(m,d,T)>0$,
$$\norm{z}^{m+1}_{L^{m+1}(\mathbb{T}^d)}=\norm{z^{\left[\frac{m+1}{2}\right]}}^2_{L^2(\mathbb{T}^d)} \leq C\left(\norm{z}^{m+1}_{L^1(\mathbb{T}^d)}+\norm{\nabla z^{\left[\frac{m+1}{2}\right]}}_{L^2(\mathbb{T}^d)}^2\right).$$
\end{lem}

\begin{proof}  Let $z\in\C^\infty(\mathbb{T}^d)$ be arbitrary.  The first equality is immediate from the definitions.  The remainder of argument is written for the case $d\geq 3$, since the cases $d=1$ and $d=2$ are similar.  In this case, for $\theta=\theta(m,d)$ defined by
$$\theta=\frac{dm}{dm+2},$$
the log-convexity of the Sobolev norm yields the estimate, for the Sobolev exponent $\frac{1}{2^*}=\frac{1}{2}-\frac{1}{d}$, for each $z\in \C^\infty(\mathbb{T}^d)$,
$$\begin{aligned} \norm{z^{\left[\frac{m+1}{2}\right]}}_{L^2(\mathbb{T}^d)} =  & \norm{z}^\frac{m+1}{2}_{L^{m+1}(\mathbb{T}^d)} \\ \leq & \norm{z}^{(1-\theta)\frac{m+1}{2}}_{L^1(\mathbb{T}^d)}\norm{z}_{L^{2^*\left(\frac{m+1}{2}\right)}(\mathbb{T}^d)}^{\theta\frac{m+1}{2}} \\ \leq & \norm{z}^{(1-\theta)\frac{m+1}{2}}_{L^1(\mathbb{T}^d)}\norm{z^{\left[\frac{m+1}{2}\right]}}_{L^{2^*}(\mathbb{T}^d)}^{\theta} \\ \leq & \norm{z}^{(1-\theta)\frac{m+1}{2}}_{L^1(\mathbb{T}^d)}\left(\norm{z^{\left[\frac{m+1}{2}\right]}-\int_{\mathbb{T}^d}z^{\left[\frac{m+1}{2}\right]}\dx}_{L^{2^*}(\mathbb{T}^d)}^{\theta}+\norm{z^{\left[\frac{m+1}{2}\right]}}_{L^1(\mathbb{T}^d)}^{\theta}\right),\end{aligned}$$
where the final inequality follows from the triangle inequality and the estimate
$$\norm{\int_{\mathbb{T}^d}z^{\left[\frac{m+1}{2}\right]}\dx}_{L^{2^*}(\mathbb{T}^d)}^\theta=\abs{\int_{\mathbb{T}^d}z^{\left[\frac{m+1}{2}\right]}\dx}^\theta\leq \norm{z^{\left[\frac{m+1}{2}\right]}}^\theta_{L^1(\mathbb{T}^d)},$$
where a constant would appear if the measure $\abs{\mathbb{T}^d}$ is not normalized to be one.  The Gagliardo-Nirenberg-Sobolev inequality and H\"older's inequality then imply that, for $C=C(d)>0$, for each $z\in \C^\infty(\mathbb{T}^d)$,
$$\begin{aligned} \norm{z^{\left[\frac{m+1}{2}\right]}}_{L^2(\mathbb{T}^d)} \leq & \norm{z}^{(1-\theta)\frac{m+1}{2}}_{L^1(\mathbb{T}^d)}\left(C\norm{\nabla z^{\left[\frac{m+1}{2}\right]}}_{L^{2}(\mathbb{T}^d)}^{\theta}+\norm{z^{\left[\frac{m+1}{2}\right]}}_{L^1(\mathbb{T}^d)}^{\theta}\right) \\ \leq & \norm{z}^{(1-\theta)\frac{m+1}{2}}_{L^1(\mathbb{T}^d)}\left(C\norm{\nabla z^{\left[\frac{m+1}{2}\right]}}_{L^{2}(\mathbb{T}^d)}^{\theta}+\norm{z^{\left[\frac{m+1}{2}\right]}}_{L^2(\mathbb{T}^d)}^{\theta}\right).\end{aligned}$$
Finally, it follows from Young's inequality that, for $C=C(m,d)>0$,
$$\norm{z^{\left[\frac{m+1}{2}\right]}}_{L^2(\mathbb{T}^d)} \leq C\left(\norm{z}^\frac{m+1}{2}_{L^1(\mathbb{T}^d)}+\norm{\nabla z^{\left[\frac{m+1}{2}\right]}}_{L^2(\mathbb{T}^d)}\right)+\frac{1}{2}\norm{z^{\left[\frac{m+1}{2}\right]}}_{L^2(\mathbb{T}^d)},$$
and, therefore,  for $C=C(m,d)>0$,
$$\norm{z^{\left[\frac{m+1}{2}\right]}}_{L^2(\mathbb{T}^d)} \leq C\left(\norm{z}^\frac{m+1}{2}_{L^1(\mathbb{T}^d)}+\norm{\nabla z^{\left[\frac{m+1}{2}\right]}}_{L^2(\mathbb{T}^d)}\right).$$
Taking the square of this equality completes the proof.  \end{proof}

The following two propositions obtain higher integrability of the entropy and kinetic defect measures in a neighborhood of the origin.  This estimate is particularly relevant for the fast diffusion case $m\in(0,1)$, since it effectively implies the $L^2$-integrability of $\nabla u^{[m]}$.

\begin{prop}\label{aux_p}  Let $u_0\in L^2(\mathbb{T}^d)$ and $\delta\in(0,1]$ be arbitrary.  Suppose that $u$ is a pathwise kinetic solution of \eqref{intro_eq} in the sense of Definition~\ref{def_solution} with initial data $u_0$.  Then, for each $T>0$, there exists $C=C(m,d,T)>0$ such that
$$\begin{aligned}& \norm{u}^{1+\delta}_{L^\infty([0,T];L^{1+\delta}(\mathbb{T}^d))}+\delta\int_0^T\int_\mathbb{R}\int_{\mathbb{T}^d}\abs{\xi}^{\delta-1}(p+q)\dx\dxi\dr \\ & \leq C\left(\norm{u_0}^{1+\delta}_{L^{1+\delta}(\mathbb{T}^d)}+\norm{u_0}^{m+\delta}_{L^1(\mathbb{T}^d)}+\left(\int_0^T\int_{\mathbb{R}}\int_{\mathbb{T}^d}q\dx\dxi\dr\right)^{\frac{m+\delta}{m+1}}\right). \end{aligned}$$
\end{prop}

\begin{proof}  Let $\delta\in(0,1]$ be arbitrary.  Suppose that $u_0\in L^2(\mathbb{T}^d)$, and suppose that $u$ is a pathwise kinetic solution of \eqref{intro_eq} with initial data $u_0$.  We will write $\chi$ for the kinetic function of $u$, $(p,q)$ respectively for the entropy and parabolic defect measures, and $\mathcal{N}$ for the exceptional set.

Let $T\in[0,\infty)\setminus\mathcal{N}$ be fixed but arbitrary.  Definition~\ref{def_solution}, in particular the global integrability of the parabolic and entropy defect measures, and Lemma~\ref{lem_interpolate} imply that the map $\xi\in\mathbb{R}\mapsto \xi^{[\delta]}$ is an admissible test function.  Therefore, for each $t\in[0,T]\setminus\mathcal{N}$,
\begin{equation}\begin{aligned}\label{aux_p_0} & \left.\int_{\mathbb{R}}\int_{\mathbb{T}^d} \chi_r\left(\Pi^{x,\xi}_{r,r}\right)^{[\delta]}\dx\dxi\right|_{r=0}^{r=t}+\delta\int_0^t\int_\mathbb{R}\int_{\mathbb{T}^d}\abs{\Pi^{x,\xi}_{r,r}}^{\delta-1}\partial_\xi\Pi^{x,\xi}_{r,r}\left(p_r+q_r\right)\dx\dxi\dr \\ & =\int_0^t\int_\mathbb{R}\int_{\mathbb{T}^d} m\abs{\xi}^{m-1}\chi_r\Delta\left(\Pi^{x,\xi}_{r,r}\right)^{[\delta]}\dx\dxi\dr.\end{aligned}\end{equation}

For the first term on the righthand side of \eqref{aux_p_0}, the integration by parts formula of Lemma~\ref{lem_ibp}, which is justified using an approximation argument and Lemma~\ref{auxlem} below, implies that, for each $t\in[0,T]\setminus\mathcal{N}$,
\begin{equation}\begin{aligned} \label{aux_p_1} & \int_0^t\int_\mathbb{R}\int_{\mathbb{T}^d} m\abs{\xi}^{m-1}\chi_r\Delta\left(\Pi^{x,\xi}_{r,r}\right)^{[\delta]}\dx\dxi\dr \\ & =-\frac{2m\delta}{m+1} \int_0^t\int_{\mathbb{T}^d} \abs{u}^{\frac{m-1}{2}}\nabla (u)^{\left[\frac{m+1}{2}\right]}\cdot\nabla\Pi^{x,u}_{r,r}\abs{\Pi^{x,u}_{r,r}}^{\delta-1}\dx\dr. \end{aligned}\end{equation}
Lemma~\ref{auxlem} implies that, for $C=C(T)>0$, for each $(x,t)\in\mathbb{T}^d\times[0,T]$,
$$\nabla\Pi^{x,u}_{r,r}\abs{\Pi^{x,u}_{r,r}}^{\delta-1}\leq Ct^\alpha\abs{u(x)}^\delta.$$
Therefore, using H\"older's inequality, Young's inequality, and the definition of the parabolic defect measure, the righthand side of \eqref{aux_p_1} satisfies, for $C_1=C_1(m,T)>0$, for each $t\in[0,T]\setminus\mathcal{N}$,
\begin{equation}\begin{aligned}\label{aux_p_2} & -\frac{2m\delta}{m+1}\int_0^t\int_{\mathbb{T}^d} \abs{u}^{\frac{m-1}{2}}\nabla (u)^{\left[\frac{m+1}{2}\right]}\nabla\Pi^{x,u}_{r,r}\abs{\Pi^{x,u}_{r,r}}^{\delta-1}\dx\dr \\ & \leq C_1t^\alpha\delta\int_0^t\int_{\mathbb{T}^d}\abs{u}^{\frac{m+\delta}{2}}\abs{u}^{\frac{\delta-1}{2}}\abs{\nabla u^{\left[\frac{m+1}{2}\right]}}\dx\dr \\ & \leq C_1t^\alpha\delta  \left(\int_0^t\int_{\mathbb{T}^d}\abs{u}^{m+\delta}\dx\dr+\int_0^t\int_\mathbb{R}\int_{\mathbb{T}^d}\abs{\xi}^{\delta-1}q\dx\dxi\dr\right). \end{aligned}\end{equation}

The final term on the righthand side of \eqref{aux_p_2} will be absorbed.  Proposition~\ref{rough_est} implies that there exists $\tilde{t}\in(0,T]$ such that
$$\inf_{(x,\xi,t)\in\mathbb{T}^d\times\mathbb{R}\times(0,\tilde{t}]}\partial_\xi\Pi^{x,\xi}_{t,t}>0.$$
It follows from Lemma~\ref{auxlem} that, for $C_2=C_2(T)>0$, for each $t\in[0,\tilde{t}]\setminus\mathcal{N}$,
\begin{equation}\label{aux_p_3}   C_2\delta\int_0^t\int_{\mathbb{R}}\int_{\mathbb{T}^d}\abs{\xi}^{\delta-1}\partial_\xi\Pi^{x,\xi}_{r,r}\left(p_r+q_r\right)\dx\dxi\dr  \leq  \delta\int_0^t\int_\mathbb{R}\int_{\mathbb{T}^d}\abs{\Pi^{x,\xi}_{r,r}}^{\delta-1}\partial_\xi\Pi^{x,\xi}_{r,r}\left(p_r+q_r\right)\dx\dxi\dr.\end{equation}
The estimates of Proposition~\ref{rough_est} imply that there exists $t_*\in(0,\tilde{t}]\setminus\mathcal{N}$ satisfying
\begin{equation}\label{aux_p_7}\inf_{(x,\xi,t)\in\mathbb{T}^d\times\mathbb{R}\times[0,t_*]}\left(C_2\partial_\xi\Pi^{x,\xi}_{r,r}-C_1t^\alpha\right)\geq\frac{C_2}{2}.\end{equation}
Therefore, returning to \eqref{aux_p_0}, for each $t\in(0,t_*]\setminus\mathcal{N}$, estimates \eqref{aux_p_2}, \eqref{aux_p_3}, and \eqref{aux_p_7} imply that, for $C=C(T)>0$,
\begin{equation}\label{aux_p_8} \left.\int_{\mathbb{R}}\int_{\mathbb{T}^d} \chi_r\left(\Pi^{x,\xi}_{r,r}\right)^{[\delta]}\right|_{r=0}^{r=t}+\delta\int_0^t\int_\mathbb{R}\int_{\mathbb{T}^d}\abs{\xi}^{\delta-1}\left(p_r+q_r\right)\dx\dxi\dr\leq C\int_0^t\int_{\mathbb{T}^d}\abs{u}^{m+\delta}\dx\dx.\end{equation}

The definition of the kinetic function and Lemma~\ref{auxlem} imply that there exists $C=C(T)>0$ such that, for each $t\in[0,T]$,
\begin{equation}\label{aux_p_9}\norm{u(\cdot,t)}^{1+\delta}_{L^{1+\delta}(\mathbb{T}^d)}\leq C\int_\mathbb{R}\int_{\mathbb{T}^d}\chi(x,\xi,r)\left(\Pi^{x,\xi}_{r,r}\right)^{[\delta]},\end{equation}
and, by Definition~\ref{def_solution}, the initial data is attained in the sense that
\begin{equation}\label{aux_p_10}\int_\mathbb{R}\int_{\mathbb{T}^d}\chi(x,\xi,0)\left(\Pi^{x,\xi}_{0,0}\right)^{[\delta]}\dx\dxi=\int_{\mathbb{R}}\int_{\mathbb{T}^d}\overline{\chi}(u_0(x))\xi^{[\delta]}\dx\dxi=\frac{1}{1+\delta}\norm{u_0}^{1+\delta}_{L^1(\mathbb{T}^d)}.\end{equation}
Finally, Corollary~\ref{path_L1} and Lemma~\ref{lem_interpolate} imply that, for $C=C(m,d,T)>0$, for each $t\in[0,T]\setminus\mathcal{N}$,
\begin{equation}\label{aux_p_11}\int_0^t\int_{\mathbb{T}^d}\abs{u}^{m+\delta}\dx\dr\leq C\left(\norm{u_0}^{m+\delta}_{L^1(\mathbb{T}^d)}+\left(\int_0^t\int_{\mathbb{R}}\int_{\mathbb{T}^d}q\dx\dxi\dr\right)^{\frac{m+\delta}{m+1}}\right).\end{equation}
Returning to \eqref{aux_p_8}, the estimates \eqref{aux_p_9}, \eqref{aux_p_10}, and \eqref{aux_p_11} imply that, for each $t\in[0,t_*]\setminus\mathcal{N}$, for $C=C(m,d,T)>0$,
\begin{equation}\begin{aligned}\label{aux_p_12} & \norm{u}^{1+\delta}_{L^\infty([0,t];L^{1+\delta}(\mathbb{T}^d))}+\delta\int_0^t\int_\mathbb{R}\int_{\mathbb{T}^d}\abs{\xi}^{\delta-1}(p+q)\dx\dxi\dr \\ & \leq C\left(\norm{u_0}^{1+\delta}_{L^{1+\delta}(\mathbb{T}^d)}+\norm{u_0}^{m+\delta}_{L^1(\mathbb{T}^d)}+\left(\int_0^t\int_{\mathbb{R}}\int_{\mathbb{T}^d}q\dx\dxi\dr\right)^{\frac{m+\delta}{m+1}}\right). \end{aligned}\end{equation}

The argument now follows by induction.  Precisely, assume that for some $k\geq 1$, the estimate of \eqref{aux_p_12} is satisfied on the interval $[0,kt_*\wedge T]$.  The identical reasoning applied to the interval $[kt_*\wedge T,(k+1)t_*\wedge T]$ and Corollary~\ref{path_L1} yield the analogue of \eqref{aux_p_12} on the interval $[kt_*\wedge T,(k+1)t_*\wedge T]$.  The inductive hypothesis and linearity then imply the estimate on the interval $[0,(k+1)t_*\wedge T]$, where the constant increases at every step.  This completes the induction argument, since the base case is \eqref{aux_p_12}, and therefore the proof.\end{proof}

The second proposition of this section improves the integrability of the entropy and parabolic defect measures in a neighborhood of zero.  Informally, this implies regularity of $u^{[\frac{m}{2}]}$ in $L^2([0,T];H^1(\mathbb{T}^d))$.

\begin{prop}\label{aux_log}  Let $u_0\in L^2_+(\mathbb{T}^d)$ be arbitrary.  Suppose that $u$ is a pathwise kinetic solution of \eqref{intro_eq} in the sense of Definition~\ref{def_solution} with initial data $u_0$.  For each $T>0$, there exists $C=C(m,d,T)>0$ such that
$$\int_0^T\int_\mathbb{R}\int_{\mathbb{T}^d}\abs{\xi}^{-1}(p+q)\dx\dxi\dr\leq C\left(1+\norm{u_0}^{(m-1)\vee 0}_{L^1(\mathbb{T}^d)}+\norm{u_0}^2_{L^2(\mathbb{T}^d)}+\int_0^T\int_{\mathbb{R}}\int_{\mathbb{T}^d}q\dx\dxi\dr\right).$$
\end{prop}

\begin{proof}  Let $u_0\in L^2_+(\mathbb{T}^d)$ be arbitrary, and let $u$ be a pathwise kinetic solution of \eqref{intro_eq} with initial data $u_0$.  We will write $\chi$ for the kinetic function of $u$, $(p,q)$ for the entropy and parabolic defect measures, and $\mathcal{N}$ for the exceptional set.

Let $T\in[0,\infty)\setminus\mathcal{N}$ be fixed but arbitrary.  Definition~\ref{def_solution}, Lemma~\ref{lem_interpolate}, the nonnegativity of the initial condition, and Corollary~\ref{path_L1} imply, following an approximation argument, that the map $\xi\in\mathbb{R}\mapsto\log(\xi)$ is an admissible test function.  Therefore, after applying the integration by parts formula, which is justified using an approximation argument and Lemma~\ref{auxlem} below, for each $t\in[0,T]\setminus\mathcal{N}$,
\begin{equation}\label{log_3}\begin{aligned} & \left.\int_\mathbb{R}\int_{\mathbb{T}^d} \chi_r\log(\Pi^{x,\xi}_{r,r})\dx\dxi\right|_{r=0}^{r=t} +\int_0^t\int_\mathbb{R}\int_{\mathbb{T}^d}\log'(\Pi^{x,\xi}_{r,r})\partial_\xi\Pi^{x,\xi}_{r,r}(p_r+q_r)\dx\dxi\dr \\ & = -\frac{2m}{m+1}\int_0^t\int_{\mathbb{T}^d}\abs{u}^{\frac{m-1}{2}}\nabla u^{\left[\frac{m+1}{2}\right]}\cdot \log'(\Pi^{x,u}_{r,r})\nabla\Pi^{x,u}_{r,r}\dx\dr.\end{aligned}\end{equation}

Lemma~\ref{auxlem} implies that there exists $C=C(T)>0$ such that
{\color{black}$$\sup_{(x,\xi,r)\in\mathbb{T}^d\times(0,\infty)\times[0,T]}\abs{\log'(\Pi^{x,\xi}_{r,r})\nabla\Pi^{x,u}_{r,r}}\leq C.$$}
Applying this estimate to the righthand side of \eqref{log_3}, it follows from H\"older's inequality, Young's inequality, and the definition of the parabolic defect measure that, for $C=C(m,T)>0$, for each $t\in[0,T]\setminus\mathcal{N}$,
$$\begin{aligned} & \abs{\int_0^t\int_{\mathbb{T}^d}\abs{u}^{\frac{m-1}{2}}\nabla u^{\left[\frac{m+1}{2}\right]}\cdot \nabla\Pi^{x,u}_{r,r}  \log' (\Pi^{x,u}_{r,r}) \dx\dr}  \\ & \leq C\left(\int_0^t\int_{\mathbb{T}^d}\abs{u}^{(m-1)\vee 0}\dx\dr+\int_0^t\int_\mathbb{R}\int_{\mathbb{T}^d}\abs{\xi}^{(m-1)\wedge 0}q\dx\dxi\dr\right).\end{aligned}$$
Therefore, Lemma~\ref{lem_interpolate} and Proposition~\ref{aux_p} imply that, for $C=C(m,d,T)>0$,
\begin{equation}\begin{aligned}\label{log_4}  & \abs{\int_0^t\int_{\mathbb{T}^d}\abs{u}^{\frac{m-1}{2}}\nabla u^{\left[\frac{m+1}{2}\right]}\cdot \nabla\Pi^{x,u}_{r,r}  \log' (\Pi^{x,u}_{r,r}) \dx\dr} \\ & \leq C\left(\norm{u_0}^{(m-1)\vee 0}_{L^1(\mathbb{T}^d)}+\left(\int_0^T\int_\mathbb{R}\int_{\mathbb{T}^d}q\dx\dxi\dr\right)^\frac{(m-1)\vee 0}{m+1}\right) \\ & \quad + C\left(\norm{u_0}^{(1+m)\wedge 2}_{L^{(1+m)\wedge 2}(\mathbb{T}^d)}+\norm{u_0}^{2m \wedge 2}_{L^1(\mathbb{T}^d)}+\left(\int_0^T\int_{\mathbb{R}}\int_{\mathbb{T}^d}q\dx\dxi\dr\right)^{\frac{2m}{m+1}\wedge 1}\right).\end{aligned}\end{equation}

For the first term of \eqref{log_3}, Proposition~\ref{aux_p} with $\delta=1$, Lemma~\ref{auxlem}, the integrability of the logarithm at zero, and the growth of the logarithm at infinity imply that, for $C=C(T)>0$, for each $t\in[0,T]\setminus\mathcal{N}$,
\begin{equation}\label{log_5}\abs{\left.\int_{\mathbb{R}}\int_{\mathbb{T}^d} \chi_r\log(\Pi^{x,\xi}_{r,r})\dx\dxi\right|_{r=0}^{r=t}}\leq C\left(1+\norm{u_0}^2_{L^2(\mathbb{T}^d)}\right).\end{equation}
Therefore, returning to \eqref{log_3}, estimates \eqref{log_4} and \eqref{log_5} imply that, for $C=C(m,d,T)>0$, for each $t\in[0,T]\setminus\mathcal{N}$,
\begin{equation}\begin{aligned}\label{log_6} & \int_0^t\int_\mathbb{R}\int_{\mathbb{T}^d}\log'(\Pi^{x,\xi}_{r,r})\partial_\xi\Pi^{x,\xi}_{r,r}(p_r+q_r)\dx\dxi\dr \\ & \leq C\left(1+\norm{u_0}^2_{L^2(\mathbb{T}^d)}+\norm{u_0}^{(m-1)\vee 0}_{L^1(\mathbb{T}^d)}+\left(\int_0^T\int_\mathbb{R}\int_{\mathbb{T}^d}q\dx\dxi\dr\right)^\frac{(m-1)\vee 0}{m+1}\right) \\ & \quad + C\left(\norm{u_0}^{(1+m)\wedge 2}_{L^{(1+m)\wedge 2}(\mathbb{T}^d)}+\norm{u_0}^{2m \wedge 2}_{L^1(\mathbb{T}^d)}+\left(\int_0^T\int_{\mathbb{R}}\int_{\mathbb{T}^d}q\dx\dxi\dr\right)^{\frac{2m}{m+1}\wedge 1}\right).\end{aligned}\end{equation}

The claim now follows similarly to Proposition~\ref{aux_p}: Proposition~\ref{rough_est} implies that there exists $t_*\in[0,T]\setminus\mathcal{N}$ such that
$$\inf_{(x,\xi,r)\in\mathbb{T}^d\times\mathbb{R}\times[0,t_*]}\partial_\xi\Pi^{x,\xi}_{r,r}\geq \frac{1}{2}.$$
Then, for $C=C(m,d,T)>0$, for each $t\in[0,t_*]$,
\begin{equation}\label{log_7} \int_0^t\int_\mathbb{R}\int_{\mathbb{T}^d}\abs{\xi}^{-1}(p+q)\dx\dxi\dr\leq C\left(1+\norm{u_0}^{(m-1)\vee 0}_{L^1(\mathbb{T}^d)}+\norm{u_0}^2_{L^2(\mathbb{T}^d)}+\int_0^T\int_{\mathbb{R}}\int_{\mathbb{T}^d}q\dx\dxi\dr\right),\end{equation}
where the righthand side of \eqref{log_6} simplifies to the righthand side of \eqref{log_7} after multiple applications of H\"older's inequality and Young's inequality.  Since the identical reasoning applies to any time interval of length less than or equal to $t_*>0$, Corollary~\ref{path_L1}, Proposition~\ref{aux_p} for $\delta=1$, and the linearity of the integral complete the proof.\end{proof}

\begin{remark} Proposition~\ref{aux_log} is not true for signed initial data.  Consider, for simplicity, the case $d=1$ and $m=1$.  Suppose that $u_0(x)=x$ in a neighborhood of the origin.  Then, since the heat flow preserves the linear behavior of the initial data locally in time, the failure of Proposition~\ref{aux_log} manifests as the non-integrability of the map $x\in\mathbb{R}\mapsto 1/\abs{x}$ in a neighborhood of the origin.  \end{remark}

\section{Stable Estimates and Existence}\label{sol_exists}

In this section, we establish the existence of pathwise kinetic solutions to the equation
$$\left\{\begin{array}{ll} \partial_t u=\Delta u^{[m]}+\nabla\cdot \left( A(x,u)\circ \dd z\right) & \textrm{in}\;\;\mathbb{T}^d\times(0,\infty), \\ u=u_0 & \textrm{on}\;\;\mathbb{T}^d\times\{0\}.\end{array}\right.$$
For this, it is necessary to derive stable estimates for the regularized equation, defined for each $\eta\in(0,1)$ and $\epsilon\in(0,1)$,
\begin{equation}\label{sm_kin}\left\{\begin{array}{ll} \partial_t u^{\eta,\epsilon}=\Delta \left(u^{\eta,\epsilon}\right)^{[m]}+\eta\Delta u+\nabla\cdot \left( A(x,u^{\eta,\epsilon})\dot{z}^\epsilon_t\right) & \textrm{in}\;\;\mathbb{T}^d\times(0,\infty), \\ u^{\eta,\epsilon}=u_0 & \textrm{on}\;\;\mathbb{T}^d\times\{0\},\end{array}\right.\end{equation}
where, as $\epsilon\rightarrow 0$, the smooth paths $\{z^\epsilon\}_{\epsilon\in(0,1)}$ converge to $z$ with respect to the $\alpha$-H\"older metric in the sense of \eqref{geometric_path}.  We will first establish estimates and the existence of pathwise kinetic solutions in the sense of Definition~\ref{def_solution} for initial data $u_0\in\C^\infty(\mathbb{T}^d)$.  The general statement will follow by density.

Returning for motivation to the kinetic formulation of the deterministic porous medium equation, the kinetic function $\chi$ of a solution $u$ satisfies
$$\left\{\begin{array}{ll} \partial_t\chi=m\abs{\xi}^{m-1}\Delta_x\chi+\partial_\xi(p+q) & \textrm{in}\;\;\mathbb{T}^d\times(0,\infty), \\ \chi=\overline{\chi}(u_0) & \textrm{on}\;\;\mathbb{T}^d\times\{0\}.\end{array}\right.$$
Following \cite{Perthame} and \cite{ChenPerthame}, estimates are obtained for the solution by testing the equation with the maps $\xi\in\mathbb{R}\mapsto\sgn(\xi)$ and $\xi\in\mathbb{R}\mapsto\xi$.  In the first case, owing to the positivity of the parabolic and entropy defect measures, observe the informal estimate
$$\norm{u}_{L^\infty\left([0,\infty);L^1(\mathbb{T}^d)\right)}=\norm{\chi}_{L^\infty\left([0,\infty);L^1(\mathbb{T}^d\times\mathbb{R})\right)}\leq \norm{\chi_0}_{L^1(\mathbb{T}^d\times\mathbb{R})}=\norm{u_0}_{L^1(\mathbb{T}^d)}.$$
In the second case, observe informally the estimate
$$\frac{1}{2}\norm{u}^2_{L^\infty\left([0,\infty);L^2(\mathbb{T}^d)\right)}+\int_0^\infty\int_\mathbb{R}\int_{\mathbb{T}^d}\left(p(x,\xi,s)+q(x,\xi,s)\right)\dx\dxi\ds\leq \frac{1}{2}\norm{u_0}^2_{L^2(\mathbb{T}^d)}.$$

In Proposition~\ref{stable_L1} we obtain the analogue of the $L^1$-estimate, and in Proposition~\ref{stable_estimates}, we obtain the analogue of the $L^2$-estimate and the estimate for the parabolic and entropy defect measures.  In the case of Proposition~\ref{stable_L1}, the argument is only a small modification of the relevant details of Theorem~\ref{theorem_uniqueness} and Corollary~\ref{path_L1}.  In the case of Proposition~\ref{stable_estimates}, the proof is essentially identical to the proof of Proposition~\ref{aux_p} for $\delta=1$.  We therefore omit the details.

\begin{prop}\label{stable_L1}For each $u_0\in L^2(\mathbb{T}^d)$, $\eta\in(0,1)$ and $\epsilon\in(0,1)$, the solution $u^{\eta,\epsilon}$ of \eqref{sm_kin} from Proposition~\ref{smooth_equation} satisfies
$$\norm{u^{\eta,\epsilon}}_{L^\infty\left([0,\infty);L^1(\mathbb{T}^d)\right)}\leq \norm{u_0}_{L^1(\mathbb{T}^d)}.$$
\end{prop}

\begin{prop}\label{stable_estimates} For each $u_0\in L^2(\mathbb{T}^d)$, $\eta\in(0,1)$ and $\epsilon\in(0,1)$, let $u^{\eta,\epsilon}$ denote the solution of \eqref{sm_kin} from Proposition \ref{smooth_equation}.  For each $T>0$, there exists $C=C(m,d,T)>0$ such that

$$\begin{aligned} & \norm{u^{\eta,\epsilon}}^2_{L^\infty([0,T];L^2(\mathbb{T}^d))} +\int_0^T\int_\mathbb{R}\int_{\mathbb{T}^d}(p^{\eta,\epsilon}(x,\xi,s)+q^{\eta,\epsilon}(x,\xi,s))\dx\dxi\ds \\ & \leq C\left(\norm{u_0}_{L^2(\mathbb{T}^d)}^2+\norm{u_0}^2_{L^1(\mathbb{T}^d)}+\norm{u_0}^{m+1}_{L^1(\mathbb{T}^d)}\right).\end{aligned}$$

\end{prop}

In general, we do not expect to obtain a stable estimate in time for the solutions $\{u^{\eta,\epsilon}\}_{\eta,\epsilon\in(0,1)}$.  However, we can obtain some regularity for the time derivative of the  transported kinetic functions, for $\eta\in(0,1)$ and $\epsilon\in(0,1)$,
\begin{equation}\label{regularized_transport}\tilde{\chi}^{\eta,\epsilon}(x,\xi,t):=\chi^{\eta,\epsilon}(X^{x,\xi,\epsilon}_{0,t},\Xi^{x,\xi,\epsilon}_{0,t},t)\;\;\textrm{for}\;\;(x,\xi,t)\in\mathbb{T}^d\times\mathbb{R}\times[0,\infty).\end{equation}
In effect, the transport cancels the oscillations introduced by the noise.  The following proposition proves that the collection $\{\partial_t\tilde{\chi}^{\eta,\epsilon}\}_{\eta,\epsilon\in(0,1)}$ is uniformly bounded in the negative Sobolev space $H^{-s}$, for $s>\frac{d}{2}+1$.

\begin{prop}\label{time_frac} For $\eta\in(0,1)$, $\epsilon\in(0,1)$ and $u_0\in L^2(\mathbb{T}^d)$, the transported kinetic function \eqref{regularized_transport} satisfies, for each $T\geq 0$, for $C=C(m,d,T)>0$,
$$\norm{\partial_t\tilde{\chi}^{\eta,\epsilon}_r}_{L^1\left([0,T];H^{-s}\left(\mathbb{T}^d\times\mathbb{R}\right)\right)} \leq C\left(1+\norm{u_0}^2_{L^1(\mathbb{T}^d)}+\norm{u_0}^{m+1}_{L^1(\mathbb{T}^d)}+\norm{u_0}^2_{L^2(\mathbb{T}^d)}\right),$$
for any Sobolev exponent $s>\frac{d}{2}+1$.
\end{prop}

\begin{proof}  Let $\epsilon\in(0,1)$, $\eta\in(0,1)$, $u_0\in L^2(\mathbb{T}^d)$, $T>0$, and $s>\frac{d}{2}+1$ be fixed but arbitrary.  For each $\delta\in(0,1)$, let $\rho^\delta_1$ and $\rho^\delta_d$ denote respectively the standard $1$-dimensional and $d$-dimensional convolution kernels of scale $\delta$.   Then, for each $\delta\in(0,1)$, define the regularization of the transported kinetic function, for $(x,\xi,t)\in\mathbb{T}^d\times\mathbb{R}\times[0,\infty)$,
$$\begin{aligned} \tilde{\chi}^{\eta,\epsilon,\delta}(x,\xi,t):= & \int_\mathbb{R}\int_{\mathbb{T}^d}\chi^{\eta,\epsilon}\left(X^{x',\xi',\epsilon}_t,\Xi^{x',\xi',\epsilon}_t,t\right)\rho^\delta_d(x'-x)\rho^\delta_1(\xi'-\xi)\dxp\dxip \\ = & \int_\mathbb{R}\int_{\mathbb{T}^d}\chi^{\eta,\epsilon}(x',\xi',t)\rho^\epsilon_d\left(Y^{x',\xi',\epsilon}_{t,t}-x\right)\rho^\epsilon_1\left(\Pi^{x',\xi',\epsilon}_{t,t}-\xi\right)\dxp\dxip,\end{aligned}$$
where the final equality is a consequence of conservative property of the characteristics (\ref{kin_measure}).

After applying the equation satisfied by $\tilde{\chi}^{\eta,\epsilon}$, and using identities (\ref{u_4}) and (\ref{u_5}), it follows after integrating by parts that, for each $r\in[0,T]$,
\begin{equation}\label{time_3}\begin{aligned} & \int_\mathbb{R}\int_{\mathbb{T}^d}\partial_t\tilde{\chi}^{\eta,\epsilon,\delta}_r\zeta\dx\dxi \\ &  =  -\int_\mathbb{R}\int_{\mathbb{T}^d}\left(\int_{\mathbb{R}^d}\left(\nabla \left(u^{\eta,\epsilon}\right)^{[m]}+\eta\nabla u^{\eta,\epsilon}\right)\cdot\left(\rho^\delta_r(x',x,u^{\eta,\epsilon}(x,r),\xi)\nabla_{x'}Y^{x',u^{\eta,\epsilon}(x,r)}_{r,r}\right)\dxp\right)\cdot \nabla_x \zeta\dx\dxi \\ & \quad -\int_\mathbb{R}\int_{\mathbb{T}^d}\left(\int_{\mathbb{R}^d} \left(\nabla \left(u^{\eta,\epsilon}\right)^{[m]}+\eta\nabla u^{\eta,\epsilon}\right)\cdot\left(\rho^\delta_r(x',x,u^{\eta,\epsilon}(x,r),\xi)\nabla_{x'}\Pi^{x',u^{\eta,\epsilon}(x,r)}_{r,r}\right)\dxp\right)\partial_\xi\zeta\dx\dxi \\ & \quad -\int_\mathbb{R}\int_{\mathbb{T}^d}\left(\int_\mathbb{R}\int_{\mathbb{T}^d}\left(p^{\eta,\epsilon}_r+q^{\eta,\epsilon}_r\right)\rho^\delta_r \partial_{\xi'} Y^{x',\xi'}_{r,r}\dxp\dxip\right)\cdot\nabla_x \zeta\dx\dxi \\ & \quad -\int_\mathbb{R}\int_{\mathbb{T}^d}\left(\int_\mathbb{R}\int_{\mathbb{T}^d}\left(p^{\eta,\epsilon}_r+q^{\eta,\epsilon}_r\right)\rho^\delta_r \partial_{\xi'} \Pi^{x',\xi'}_{r,r}\dxp\dxip\right)\partial_\xi \zeta\dx\dxi.\end{aligned}\end{equation}
The dependence on the convolution kernel is removed by integrating the variables $(x,\xi)\in\mathbb{T}^d\times\mathbb{R}$.  The characteristics are uniformly bounded, for $C=C(T)>0$, for each $r\in[0,T]$, using the estimates of Proposition~\ref{rough_est}.  Therefore, H\"older's inquality, Young's inequality, and the boundedness of the domain imply that
$$\begin{aligned}& \abs{\int_\mathbb{R}\int_{\mathbb{T}^d}\partial_t\tilde{\chi}^{\eta,\epsilon,\delta}_r\zeta\dx\dxi} \\ & \leq C\norm{\nabla_{(x,\xi)}\zeta}_{L^\infty(\mathbb{T}^d\times\mathbb{R};\mathbb{R}^{d+1})}\left(\eta+\int_{\mathbb{T}^d} \abs{u^{\eta,\epsilon}_r}^{(m-1)\vee 0}\dx\right) \\ & \quad +C\norm{\nabla_{(x,\xi)}\zeta}_{L^\infty(\mathbb{T}^d\times\mathbb{R};\mathbb{R}^{d+1})}\left(\int_\mathbb{R}\int_{\mathbb{T}^d}\left(p^{\eta,\epsilon}_r+(1+\abs{\xi}^{(m-1)\vee 1})q^{\eta,\epsilon}_r\right)\dx\dxi\right).\end{aligned}$$
Since $s>\frac{d}{2}+1$, the Sobolev embedding theorem and Proposition~\ref{stable_L1} imply that, for $C=C(m,d,T)>0$, for each $r\in[0,T]$,
\begin{equation}\begin{aligned}\label{time_6} &\abs{\int_\mathbb{R}\int_{\mathbb{T}^d}\partial_t\tilde{\chi}^{\eta,\epsilon,\delta}_r\zeta\dx\dxi} \\ & \leq C\norm{\zeta}_{H^s(\mathbb{T}^d\times\mathbb{R})}\left(\eta+\int_{\mathbb{T}^d} \abs{u^{\eta,\epsilon}_r}^{(m-1)\vee 0}\dx\right) \\ & \quad +C\norm{\zeta}_{H^s(\mathbb{T}^d\times\mathbb{R})}\left(\int_\mathbb{R}\int_{\mathbb{T}^d}\left(p^{\eta,\epsilon}_r+(1+\abs{\xi}^{(m-1)\vee 1})q^{\eta,\epsilon}_r\right)\dx\dxi\right).\end{aligned}\end{equation}

Since $\zeta\in\C^\infty_c(\mathbb{T}^d\times\mathbb{R})$ was arbitrary, it follows from \eqref{time_6} that, after integrating in time, for $C=C(m,d,T)>0$,
$$\begin{aligned}\norm{\partial_t\tilde{\chi}^{\eta,\epsilon,\delta}_r}_{L^1\left([0,T];H^{-s}\left(\mathbb{T}^d\times\mathbb{R}\right)\right)} \leq & C\left(\eta+\int_0^T\int_{\mathbb{T}^d} \abs{u^{\eta,\epsilon}}^{(m-1)\vee 0}\dx\dr\right) \\ & +C\int_0^T\int_{\mathbb{T}^d}\left(p^{\eta,\epsilon}+(1+\abs{\xi}^{(m-1)\vee 1})q^{\eta,\epsilon}\right)\dx\dxi\dr.\end{aligned}$$
Therefore, after passing to the limit $\delta\rightarrow 0$, a repetition of the arguments leading to the estimate for \eqref{non_2} implies that, for $C=C(m,d,T)>0$,
$$\norm{\partial_t\tilde{\chi}^{\eta,\epsilon}_r}_{L^1\left([0,T];H^{-s}\left(\mathbb{T}^d\times\mathbb{R}\right)\right)} \leq C\left(1+\norm{u_0}^{(m-1)\vee 0}_{L^1(\mathbb{T}^d)}+\norm{u_0}^2_{L^2(\mathbb{T}^d)}+\int_0^T\int_\mathbb{R}\int_{\mathbb{T}^d}(p+q)\dx\dxi\dr\right).$$
Proposition~\ref{stable_estimates}, H\"older's inequality, and Young's inequality therefore imply that, for $C=C(m,d,T)>0$,
$$\norm{\partial_t\tilde{\chi}^{\eta,\epsilon}_r}_{L^1\left([0,T];H^{-s}\left(\mathbb{T}^d\times\mathbb{R}\right)\right)} \leq C\left(1+\norm{u_0}^2_{L^1(\mathbb{T}^d)}+\norm{u_0}^{m+1}_{L^1(\mathbb{T}^d)}+\norm{u_0}^2_{L^2(\mathbb{T}^d)}\right),$$
which completes the proof.  \end{proof}

It remains to establish the regularity of the kinetic function with respect to the spatial and velocity variables.  The regularity in the velocity variable follows from Proposition \ref{BVto}, and the spatial regularity follows from Proposition \ref{gradtofrac}.  These estimates are combined using Proposition~\ref{meld} to obtain joint regularity in both variables.

\begin{prop}\label{fractional_sobolev}  Let $u_0\in L^2(\mathbb{T}^d)$, $\eta\in(0,1)$, and $\epsilon\in(0,1)$.  If $m\in(1,\infty)$,  for each $s\in(0,\frac{2}{m+1})$ and $T\geq 0$, there exists $C=C(m,d,T,s)>0$ such that
$$\norm{\chi^{\eta,\epsilon}}_{L^1_t\left([0,T];W^{s,1}_{x,\xi}(\mathbb{T}^d)\right)}\leq C\left(1+\norm{u_0}_{L^1(\mathbb{T}^d)}+\norm{u_0}^{m+1}_{L^1(\mathbb{T}^d)}+\norm{u_0}^2_{L^2(\mathbb{T}^d)}\right).$$
If $m\in(0,1]$, for each $s\in(0,1)$ and $T\geq 0$, there exists $C=C(m,d,T,s)>0$ such that
$$\norm{\chi^{\eta,\epsilon}}_{L^1_t\left([0,T];W^{s,1}_{x,\xi}(\mathbb{T}^d)\right)}\leq C\left(1+\norm{u_0}_{L^1(\mathbb{T}^d)}+\norm{u_0}^{2}_{L^1(\mathbb{T}^d)}+\norm{u_0}^2_{L^2(\mathbb{T}^d)}\right).$$
\end{prop}

\begin{proof}  Let $u_0\in L^2(\mathbb{T}^d)$, $\eta\in(0,1)$ and $\epsilon\in(0,1)$ be arbitrary.  Let $s\in(0,\frac{2}{m+1}\wedge 1)$ and $T\geq 0$ be arbitrary.  It follows from Corollary~\ref{kreg} that, for $C=C(d,s)>0$,
\begin{equation}\label{fs_0}\norm{\chi^{\eta,\epsilon}}_{L^1_t\left([0,T];W^{s,1}_{x,\xi}(\mathbb{T}^d\times\mathbb{R})\right)}\leq C\left(\norm{\chi^{\eta,\epsilon}}_{L^1_t\left([0,T];L^1_\xi(\mathbb{R};W^{s,1}_x(\mathbb{T}^d))\right)}+\norm{\chi^{\eta,\epsilon}}_{L^1_t\left([0,T];L^1_x(\mathbb{T}^d;W^{s,1}(\mathbb{R}))\right)}\right).\end{equation}
Corollary~\ref{BVtoKin} implies that, for $C=C(d,T,s)>0$,
\begin{equation}\label{fs_1}\norm{\chi^{\eta,\epsilon}}_{L^1_t([0,T];L^1_x\left(\mathbb{T}^d;W_\xi^{s,1}(\mathbb{R}))\right)}\leq C\left(1+\norm{u^{\eta,\epsilon}}_{L^1_t([0,T];L^1_x(\mathbb{T}^d))}\right).\end{equation}
Corollary~\ref{gtfk_cor} and Proposition~\ref{stable_L1} imply that, for $C=C(m,d,T,s)>0$, if $m\in(1,\infty)$,
\begin{equation}\label{fs_2}\begin{aligned} & \norm{\chi^{\eta,\epsilon}}_{L^1_t\left([0,T];L^1_\xi(\mathbb{R};W^{s,1}_x(\mathbb{T}^d))\right)}\leq  \\ & C\left(\norm{u_0}_{L^1(\mathbb{T}^d)}+\norm{u_0}^{m+1}_{L^1(\mathbb{T}^d)}+\int_0^T\norm{\nabla \left(u^{\eta,\epsilon}\right)^{\left[\frac{m+1}{2}\right]}(\cdot,r)}^\frac{2}{m+1}_{L^2(\mathbb{T}^d)}\dr\right),\end{aligned}\end{equation}
and, if $m\in(0,1]$,
\begin{equation}\label{fs_22}\begin{aligned} & \norm{\chi^{\eta,\epsilon}}_{L^1_t\left([0,T];L^1_\xi(\mathbb{R};W^{s,1}_x(\mathbb{T}^d))\right)}\leq  \\ & C\left(\norm{u_0}_{L^1(\mathbb{T}^d)}+\norm{u_0}^{2(1-m)}_{L^1(\mathbb{T}^d)}+\int_0^T\norm{\nabla u^{\left[\frac{m+1}{2}\right]}(\cdot,r)}^2_{L^2(\mathbb{T}^d)}\dr\right).\end{aligned}\end{equation}

Returning to \eqref{fs_0}, if $m\in(1,\infty)$, it follows from \eqref{fs_1} and \eqref{fs_2}, using the fact that $2/(m+1)<1$, H\"older's inequality, Young's inequality, and the definition of the parabolic defect measure that, for $C=C(m,d,T,s)>0$,
\begin{equation*}\begin{aligned} & \norm{\chi^{\eta,\epsilon}}_{L^1_t\left([0,T];W^{s,1}_{x,\xi}(\mathbb{T}^d\times\mathbb{R})\right)} \\ & \leq  C\left(1+\norm{u_0}_{L^1(\mathbb{T}^d)}+\norm{u_0}_{L^1(\mathbb{T}^d)}^{m+1}+\int_0^T\int_\mathbb{R}\int_{\mathbb{T}^d}q^{\eta,\epsilon}\dx\dxi\dt\right).\end{aligned}\end{equation*}
Similarly, from \eqref{fs_1} and \eqref{fs_22}, if $m\in(0,1]$, for $C=C(m,d,T,s)>0$,
\begin{equation*}\begin{aligned} & \norm{\chi^{\eta,\epsilon}}_{L^1_t\left([0,T];W^{s,1}_{x,\xi}(\mathbb{T}^d\times\mathbb{R})\right)} \\ & \leq  C\left(1+\norm{u_0}_{L^1(\mathbb{T}^d)}+\norm{u_0}_{L^1(\mathbb{T}^d)}^{2(m-1)}+\int_0^T\int_\mathbb{R}\int_{\mathbb{T}^d}q^{\eta,\epsilon}\dx\dxi\dt\right).\end{aligned}\end{equation*}
Therefore, if $m\in(1,\infty)$, Proposition~\ref{stable_estimates} and the fact that, for each $a\in[0,\infty)$, we have $a^2\leq \left(a\vee a^{m+1}\right)$, for $C=C(m,d,T,s)>0$,
$$\norm{\chi^{\eta,\epsilon}}_{L^1_t\left([0,T];W^{s,1}_{x,\xi}(\mathbb{T}^d)\right)}\leq C\left(1+\norm{u_0}_{L^1(\mathbb{T}^d)}+\norm{u_0}^{m+1}_{L^1(\mathbb{T}^d)}+\norm{u_0}^2_{L^2(\mathbb{T}^d)}\right).$$
If $m\in(0,1]$, Proposition~\ref{stable_estimates} and the fact that, for each $a\in[0,\infty)$, we have $a^{2(1-m)}\leq \left(a\vee a^2\right)$, imply that, for $C=C(m,d,T,s)>0$,
$$\norm{\chi^{\eta,\epsilon}}_{L^1_t\left([0,T];W^{s,1}_{x,\xi}(\mathbb{T}^d)\right)}\leq C\left(1+\norm{u_0}_{L^1(\mathbb{T}^d)}+\norm{u_0}^{2}_{L^1(\mathbb{T}^d)}+\norm{u_0}^2_{L^2(\mathbb{T}^d)}\right),$$
which completes the proof.  \end{proof}

The following corollary proves that the transported kinetic function $\tilde{\chi}^{\eta,\epsilon}$ inherits the regularity of $\chi^{\eta,\epsilon}$.  The proof is an immediate consequence of Proposition~\ref{fractional_sobolev} and Corollary~\ref{last_cor}.

\begin{cor}\label{prop_transfer} For each $\eta\in(0,1)$, $\epsilon\in(0,1)$, and $u_0\in L^2(\mathbb{T}^d)$, and for each $s\in(0,\frac{2}{m+1}\wedge 1)$ and $T\geq 0$, there exists  $C=C(m,d,T,s)>0$ such that
$$\norm{\tilde{\chi}^{\eta,\epsilon}}_{L^1\left([0,T];W^{s,1}(\mathbb{T}^d)\right)}\leq C\left(1+\norm{u_0}_{L^1(\mathbb{T}^d)}+\norm{u_0}^{(m+1)\vee 2}_{L^1(\mathbb{T}^d)}+\norm{u_0}^2_{L^2(\mathbb{T}^d)}\right).$$
\end{cor}

The following theorem establishes the existence of pathwise kinetic solutions for initial data $u_0\in L^2(\mathbb{T}^d)$.  The proof is consequence of Proposition~\ref{time_frac}, Corollary~\ref{prop_transfer}, and the Aubin-Lions-Simon lemma.

\begin{thm}\label{thm_solution}  For every $u_0\in L^2(\mathbb{T}^d)$, there exists a pathwise kinetic solution $u$ to the equation
\begin{equation}\label{solution_000}\left\{\begin{array}{ll} \partial_t u=\Delta u^{[m]}+\nabla\cdot \left(A(x,u)\circ \dd z_t\right) & \textrm{in}\;\;\mathbb{T}^d\times(0,\infty), \\ u=u_0 & \textrm{on}\;\;\mathbb{T}^d\times\{0\},\end{array}\right.\end{equation}
in the sense of Definition~\ref{def_solution}.  In particular, the solution satisfies the estimates of Corollary~\ref{path_L1} and Proposition~\ref{aux_p}.
\end{thm}

\begin{proof} Let $u_0\in L^2(\mathbb{T}^d)$ be arbitrary.  Let $\{u^{\eta,\epsilon}\}_{\eta,\epsilon\in(0,1)}$ denote the solutions of the regularized equation (\ref{sm_kin}) with initial data $u_0$, with transported kinetic functions $\{\tilde{\chi}^{\eta,\epsilon}\}_{\eta,\epsilon\in(0,1)}$, entropy defect measures $\{p^{\eta,\epsilon}\}_{\eta,\epsilon\in(0,1)}$, and parabolic defect measures $\{q^{\eta,\epsilon}\}_{\eta,\epsilon\in(0,1)}$.

Since, for each $s\in(0,\frac{2}{m+1}\wedge 1)$ and $R>0$, the embedding of $W^{s,1}(\mathbb{T}^d\times[-R,R])$ into $L^1(\mathbb{T}^d\times[-R,R])$ is compact, and since $L^1(\mathbb{T}^d\times\mathbb{R})$ embeds continuously into $H^{-s}(\mathbb{T}^d\times\mathbb{R})$ for $s>\frac{d}{2}+1$, it follows from Proposition \ref{time_frac}, Corollary~\ref{prop_transfer}, the Aubin-Lions-Simon lemma  Aubin \cite{Aubin}, Lions \cite{pLions}, and Simon \cite{Simon}, and a diagonal argument that, for each $T\geq 0$, the family
$$\{\tilde{\chi}^{\eta,\epsilon}\}_{\eta,\epsilon\in(0,1)}\;\;\textrm{is precompact in}\;L^1([0,T];L^1(\mathbb{T}^d\times\mathbb{R})).$$
The conservative property of the characteristics (\ref{kin_measure}) therefore implies that, for each $T\geq 0$,
$$\{\chi^{\eta,\epsilon}\}_{\eta,\epsilon\in(0,1)}\;\;\textrm{is precompact in}\;L^1([0,T];L^1(\mathbb{T}^d\times\mathbb{R})).$$
It is then immediate from the definition of the kinetic function that
\begin{equation}\label{solution_0}\{u^{\eta,\epsilon}\}_{\eta,\epsilon\in(0,1)}\;\;\textrm{is precompact in}\;L^1([0,T];L^1(\mathbb{T}^d)).\end{equation}
Furthermore, using Proposition \ref{stable_estimates}, the sequence of measures
\begin{equation}\label{solution_1}\{(p^{\eta,\epsilon},q^{\eta,\epsilon})\}_{\eta,\epsilon\in(0,1)}\;\;\textrm{is weakly precompact in}\;\;\BUC(\mathbb{T}^d\times\mathbb{R})^*,\end{equation}
and
\begin{equation}\label{solution_2}\left\{(u^{\eta,\epsilon})^{\left[\frac{m+1}{2}\right]}\right\}_{\eta,\epsilon\in(0,1)}\;\;\textrm{is weakly precompact in}\;L^2([0,T];H^1(\mathbb{T}^d)).\end{equation}

Therefore, after passing to a subsequence $\{(\eta_k,\epsilon_k)\rightarrow (0,0)\}_{k=1}^\infty$, there exists a function $u\in L^1([0,T];L^1(\mathbb{T}^d))$ such that, as $k\rightarrow\infty$,
\begin{equation}\label{solution_3}u^{\eta_k,\epsilon_k}\rightarrow u\;\;\textrm{strongly in}\;\;L^1([0,T];L^1(\mathbb{T}^d)).\end{equation}
Furthermore, as $k\rightarrow\infty$,
{\color{black}\begin{equation}\label{solution_4}(u^{\eta_k,\epsilon_k})^{\left[\frac{m+1}{2}\right]}\rightharpoonup u^{\left[\frac{m+1}{2}\right]}\;\;\textrm{weakly in}\;\;L^2([0,T];H^1(\mathbb{T}^d)).\end{equation}}
Since, by definition, for each $\eta\in(0,1)$ and $\epsilon\in(0,1)$, for $(x,\xi,t)\in\mathbb{T}^d\times\mathbb{R}\times[0,\infty)$,
$$p^{\eta,\epsilon}(x,\xi,t):=\delta_0(\xi-u^{\eta,\epsilon}(x,t))\eta\abs{\nabla u^{\eta,\epsilon}}^2,$$
and
$$q^{\eta,\epsilon}(x,\xi,t):=\delta_0(\xi-u^{\eta,\epsilon}(x,t))\frac{4m}{(m+1)^2}\abs{\nabla \left(u^{\eta,\epsilon}\right)^{\left[\frac{m+1}{2}\right]}(x,t)}^2,$$
the estimates of Proposition \ref{stable_estimates} imply that there exist positive measures $(p',q')$ such that, for each $T>0$, as $k\rightarrow\infty$,
\begin{equation}\label{solution_5} (p^{\eta_k,\epsilon_k},q^{\eta_k,\epsilon_k})\rightharpoonup(p',q')\;\;\textrm{weakly in}\;\;\BUC(\mathbb{T}^d\times\mathbb{R}\times[0,T])^*.\end{equation}
It follows from the strong convergence \eqref{solution_3} and the weak lower semicontinuity of the weighted Sobolev norm that, in the sense of measures,
\begin{equation}\label{solution_6}\delta_0(\xi-u(x,t))\frac{4m}{(m+1)^2}\abs{\nabla u^{\left[\frac{m+1}{2}\right]}}^2\leq q'(x,\xi,t)\;\;\textrm{for}\;\;(x,\xi,t)\in\mathbb{T}^d\times\mathbb{R}\times[0,\infty).\end{equation}
{\color{black}To see this, let $f\in\C^\infty_c(\mathbb{T}^d\times\mathbb{R}\times[0,T])$ be an arbitrary nonnegative function.  The strong convergence \eqref{solution_3} implies that, as $k\rightarrow\infty$, for every $p\in[1,\infty)$,
$$\sqrt{f(u^{\epsilon_k,\eta_k})}\rightarrow\sqrt{f(u)}\;\;\textrm{strongly in}\;\;L^p(\mathbb{T}^d\times[0,T]).$$
Hence, using the weak convergence \eqref{solution_4},
$$\sqrt{f(u^{\epsilon_k,\eta_k})}\nabla\left(u^{\epsilon_k,\eta_k}\right)^{\left[\frac{m+1}{2}\right]}\rightharpoonup\sqrt{f(u)}\nabla u^{\left[\frac{m+1}{2}\right]}\;\;\textrm{weakly in}\;\;L^p(\mathbb{T}^d\times[0,T]),$$
for each $p\in(1,2)$. Therefore, the weak convergence \eqref{solution_5}, the definition of the measures $\{q^{\epsilon_k,\eta_k}\}_{k=1}^\infty$, and the weak lower-semicontinuity of the $L^2$-norm prove that
$$\begin{aligned}\frac{4m}{(m+1)^2}\int_{\mathbb{T}^d}\int_0^Tf(u)\abs{\nabla u^{\left[\frac{m+1}{2}\right]}}^2\leq & \liminf_{k\rightarrow\infty}\frac{4m}{(m+1)^2}\int_{\mathbb{T}^d}\int_0^Tf(u^{\epsilon_k,\eta_k})\abs{\nabla \left(u^{\epsilon_k,\eta_k}\right)^{\left[\frac{m+1}{2}\right]}}^2 \\ = & \liminf_{k\rightarrow\infty}\int_0^T\int_\mathbb{R}\int_{\mathbb{T}^d}f \;q^{\epsilon_k,\eta_k} \\  = & \int_0^T\int_\mathbb{R}\int_{\mathbb{T}^d}f\;q',\end{aligned}$$
which, since $f$ was arbitrary, establishes \eqref{solution_6}.}

We define the parabolic defect measure
$$q(x,\xi,t):=\delta_0(\xi-u(x,t))\frac{4m}{(m+1)^2}\abs{\nabla u^{\left[\frac{m+1}{2}\right]}}^2\;\;\textrm{for}\;\;(x,\xi,t)\in\mathbb{T}^d\times\mathbb{R}\times[0,\infty),$$
and, since \eqref{solution_6} implies that that $q'-q$ is nonnegative, we define the entropy defect measure
$$p:=p'+q'-q\geq 0\;\;\textrm{on}\;\;\mathbb{T}^d\times\mathbb{R}\times[0,\infty).$$
Finally, as $\epsilon\rightarrow 0$, it follows from the regularity assumption \eqref{prelim_regular}, the choice of $\{z^\epsilon\}_{\epsilon\in(0,1)}$ satisfying \eqref{geometric_path}, and Proposition \ref{rough_est} that, for each $T\geq 0$,
\begin{equation}\label{solution_7}\lim_{\epsilon\rightarrow 0}\norm{\abs{Y^{x,\xi,\epsilon}_{t,t}-Y^{x,\xi}_{t,t}}+\abs{\Pi^{x,\xi,\epsilon}_{t,t}-\Pi^{x,\xi}_{t,t}}}_{L^\infty(\mathbb{T}^d\times\mathbb{R}\times[0,T])}=0.\end{equation}

For the kinetic function $\chi$ of $u$, the convergence (\ref{solution_3}) implies that, for a subset $\mathcal{N}\subset(0,\infty)$ {\color{black}of measure zero}, for each $t\in[0,\infty)\setminus\mathcal{N}$,
$$\lim_{k\rightarrow\infty} \norm{u^{\eta_k,\epsilon_k}(\cdot,t)-u(\cdot,t)}_{L^1(\mathbb{T}^d)}=0.$$
Therefore, the additional convergences (\ref{solution_4}), (\ref{solution_5}), and (\ref{solution_7}) imply that, for each $t_0,t_1\in [0,\infty)\setminus\mathcal{N}$, for every $\rho_0\in \C^\infty(\mathbb{T}^d)$, for the solution $\rho_{t_0,t}$ of \eqref{char_rough_eq} with initial data $\rho_0$,
$$\begin{aligned}\left.\int_\mathbb{R}\int_{\mathbb{T}^d}\chi_r\rho_{t_0,r}\dx\dxi\right|_{r=t_0}^{t_1} = & \int_{t_0}^{t_1}\int_\mathbb{R}\int_{\mathbb{T}^d}m\abs{\xi}^{m-1}\chi_r\Delta_x\rho_{t_0,r}\dx\dxi\dr \\ & - \int_{t_0}^{t_1}\int_\mathbb{R}\int_{\mathbb{T}^d}\left(p_r+q_r\right)\partial_\xi\rho_{t_0,r}\dx\dxi\dr, \end{aligned}$$
where, when $t_0=0$,
$$\int_\mathbb{R}\int_{\mathbb{T}^d}\chi(x,\xi,0)\rho_{0,0}\dx\dxi=\lim_{k\rightarrow\infty}\int_{\mathbb{R}}\int_{\mathbb{T}^d}\chi^{\eta_k,\epsilon_k}(x,\xi,0)\rho_{0,0}\dx\dxi=\int_{\mathbb{R}}\int_{\mathbb{T}^d}\overline{\chi}(u_0(x),\xi)\rho_0\dx\dxi.$$
This completes the proof that $u$ is a pathwise kinetic solution.  It is then immediate that the solution satisfies the estimates of Corollary~\ref{path_L1} and Proposition~\ref{aux_p}, which completes the proof of the theorem.  \end{proof}

{\color{black}We will now show that the solutions constructed in Theorem~\ref{thm_solution} depend continuously on the driving noise.  The proof will follow from a compactness argument relying on the estimates from the proof of Theorem~\ref{thm_solution}, the rough path estimates of Proposition~\ref{rough_est}, and the uniqueness of pathwise kinetic solutions from Theorem~\ref{theorem_uniqueness}.  In particular, these methods do not yield an explicit estimate quantifying the convergence of the solutions in terms of the convergence of the noise.  In the statement below, the metric $d_\alpha$ denotes the $\alpha$-H\"older metric on the space of geometric rough paths introduced in Section~\ref{sec_rough}.

\begin{thm}\label{final_initial_cts} Let $u_0\in L^2_+(\mathbb{T}^d)$ and $T>0$.  Let $\{z^n\}_{n=1}^\infty$ and $z$ be a sequence of $n$-dimensional, $\alpha$-H\"older continuous geometric rough paths on $[0,T]$ satisfying
\begin{equation}\label{rough_path_con}\lim_{n\rightarrow\infty}d_\alpha(z^n,z)=0.\end{equation}
Let $\{u^n\}_{n=1}^\infty$ and $u$ denote the pathwise kinetic solutions on $[0,T]$ with initial data $u_0$ and driving signals $\{z^n\}_{n=1}^\infty$ and $z$ respectively.   Then,
$$\lim_{n\rightarrow\infty}\norm{u^n-u}_{L^\infty([0,T];L^1(\mathbb{T}^d))}=0.$$
\end{thm}

\begin{proof}  Let $u_0\in L^2_+(\mathbb{T}^d)$ and $T>0$.  Let $\{z^n\}_{n=1}^\infty$ and $z$ be $\alpha$-H\"older continuous, geometric rough paths on $[0,T]$ satisfying \eqref{rough_path_con}.  The convergence implies that there exists $C>0$ such that, for each $n\geq 1$,
\begin{equation}\label{ncts_2}d_\alpha(z^n,e)\leq C,\end{equation}
where $e$ denotes the constant path beginning from the origin defined in Section~\ref{sec_rough}.

Let $\{u^n\}_{n=1}^\infty$ denote the solutions of \eqref{solution_000} constructed in Theorem~\ref{thm_solution} with initial data $u_0$ and driving signals $\{z^n\}_{n=1}^\infty$ respectively.  It follows from \eqref{ncts_2} and the rough path estimates of Proposition~\ref{rough_est} that the solutions $\{u^n\}_{n=1}^\infty$ satisfy the estimates of Proposition~\ref{stable_L1}, Proposition~\ref{stable_estimates}, Proposition~\ref{time_frac}, Proposition~\ref{fractional_sobolev} and Corollary~\ref{prop_transfer} on the interval $[0,T]$ for a constant that is independent of $n\geq 1$.

A repetition of the proof of Theorem~\ref{thm_solution} proves that, after passing to a subsequence $\{n_k\}_{k=1}^\infty$, there exists a pathwise kinetic solution $u$ of \eqref{solution_000} with initial data $u_0$ and driving noise $z$ such that, as $k\rightarrow\infty$,
$$\lim_{k\rightarrow\infty}\norm{u^{n_k}-u}_{L^\infty([0,T];L^1(\mathbb{T}^d))}=0.$$
However, since it follows from Theorem~\ref{theorem_uniqueness} that $u$ is the unique solution of \eqref{solution_000} with initial data $u_0$ and driving noise $z$, we conclude that, along the full sequence,
$$\lim_{n\rightarrow\infty}\norm{u^n-u}_{L^\infty([0,T];L^1(\mathbb{T}^d))}=0,$$
which completes the proof. \end{proof}}

\begin{appendix}

\section{A Regularized Equation and its Kinetic Formulation}\label{kinetic}

Since equation (\ref{intro_eq}) is not a priori well-defined, in this section we will consider a uniformly elliptic regularization of (\ref{intro_eq}).  For each integer $M\geq 1$, define the globally Lipschitz nonlinearity
\begin{equation}\label{e_bdd} \phi^M(\xi):=\left\{\begin{array}{ll} \xi^{[m]} & \textrm{if}\;\;\abs{\xi}\leq M, \\ \xi M^{m-1} & \textrm{if}\;\;\abs{\xi}\geq M.\end{array}\right.\end{equation}
Then, for each $\delta\in(0,1)$, for a standard one-dimensional convolution kernel $\rho^\delta_1$, for each $M\geq 1$ and $\delta\in(0,1)$, define the convolution
\begin{equation}\label{e_convolve} \phi^{M,\delta}(\eta):=(\phi^M*\rho^\delta_1)(\eta)=\int_\mathbb{R}\phi^M(\xi)\rho^\delta_1(\xi-\eta)\dxi\;\;\textrm{for each}\;\;\eta\in\mathbb{R}.\end{equation}
The nonlinearity $\phi^{M,\delta}$ will be used to approximate the porous medium nonlinearity $\xi\in\mathbb{R}\mapsto \xi^{[m]}$.  In fact, since the derivative of (\ref{e_bdd}) is positive away from zero, the nonlinearity (\ref{e_convolve}) defines a uniformly elliptic equation.  However, in order to preserve $H^1$-regularity in the limit $(M,\delta)\rightarrow (\infty,0)$, we will additionally consider an $\eta$-perturbation by the Laplacian, for $\eta\in(0,1)$.

It remains to regularize the noise. The assumption \eqref{prelim_Holder} that $z$ is a geometric rough path ensures that there exists a sequence of smooth paths
\begin{equation}\label{e_noise_1} \left\{z^\epsilon:[0,\infty)\rightarrow\mathbb{R}^n\right\}_{\epsilon\in(0,1)},\end{equation}
such that, as $\epsilon\rightarrow 0$, the paths $z^\epsilon$ converge to $z$ with respect to the $\alpha$-H\"older norm on the space of geometric rough paths $\C^{0,\alpha}([0,T];G^{\left\lfloor \frac{1}{\alpha}\right\rfloor}(\mathbb{R}^n))$ in the sense of \eqref{geometric_path}.

The first proposition of this section is essentially classical, and establishes the existence of solutions to a uniformly elliptic perturbation of equation (\ref{intro_eq}) driven by smooth noise.  In the proof, we consider the family of smooth equations defined by the family of nonlinearities (\ref{e_convolve}), for $M\geq 1$ and $\delta\in(0,1)$, and we obtain stable estimates in order to pass simultaneously to the limit $M\rightarrow\infty$ and $\delta\rightarrow 0$.

The estimates are based on testing the equation with the solution and the composition of the solution with $\phi^{M,\delta}$.   Therefore, an anti-derivative for (\ref{e_convolve}) will appear in the argument, which can be constructed via an explicit calculation.  Indeed, for each $M\geq 1$, define
\begin{equation}\label{e_anti}\psi^M(\xi):=\left\{\begin{array}{ll} \frac{1}{m+1}\abs{\xi}^{m+1} & \textrm{if}\;\;\abs{\xi}\leq M, \\ \frac{\xi^2}{2}M^{m-1}+\frac{M^{m+1}}{m+1}-\frac{M^{m+1}}{2} & \textrm{if}\;\;\abs{\xi}\geq M. \end{array}\right.\end{equation}
Observe that,  for each $M\geq 1$ and $\delta\in(0,1)$, for the one-dimensional convolution kernel $\rho^\delta_1$ used in (\ref{e_convolve}), the convolution
\begin{equation}\label{e_anti_1}\psi^{M,\delta}:=(\psi^M*\rho^\delta_1)\end{equation}
is an anti-derivative for (\ref{e_convolve}).

\begin{prop}\label{smooth_equation}  For each $\eta\in(0,1)$, $\epsilon\in(0,1)$, and $u_0\in L^2(\mathbb{T}^d)$, there exists a classical solution of the equation
\begin{equation}\label{smooth_eq}\left\{\begin{array}{ll} \partial_t u=\Delta u^{[m]}+\eta\Delta u+\nabla\cdot \left(A(x,u)\dot{z}^\epsilon_t\right) & \textrm{in}\;\;\mathbb{T}^d\times(0,\infty), \\ u=u_0 & \textrm{on}\;\;\mathbb{T}^d\times\{0\},\end{array}\right.\end{equation}
satisfying, for $C=C(\epsilon,T)>0$,
$$\norm{u}_{L^\infty(\mathbb{T}^d\times[0,T])}\leq C\norm{u_0}_{L^\infty(\mathbb{T}^d)}.$$
For $C=C(\epsilon,T)>0$,
$$\norm{u}^2_{L^\infty\left([0,T];L^2(\mathbb{T}^d)\right)}+\norm{\nabla u^{\left[\frac{m+1}{2}\right]}}^2_{L^2\left([0,T];L^2(\mathbb{T}^d;\mathbb{R}^d)\right)}+ \norm{\eta^\frac{1}{2}\nabla u}^2_{L^2\left([0,T];L^2(\mathbb{T}^d;\mathbb{R}^d)\right)}\leq C\norm{u_0}_{L^2(\mathbb{T}^d)},$$
and 
$$\norm{u}^{m+1}_{L^\infty\left([0,T];L^{m+1}(\mathbb{T}^d)\right)}+\norm{\nabla u^{[m]}}^2_{L^2\left([0,T];L^2(\mathbb{T}^d)\right)}\leq C\left(\norm{u_0}^{m+1}_{L^{m+1}(\mathbb{T}^d)}+\norm{u_0}^2_{L^2(\mathbb{T}^d)}\right).$$
Finally, for $C=C(\epsilon,T)>0$,
$$\norm{\partial_t u}^2_{L^2\left([0,T];H^{-1}(\mathbb{T}^d)\right)}\leq C\left(\norm{u_0}^{m+1}_{L^{m+1}(\mathbb{T}^d)}+\norm{u_0}^2_{L^2(\mathbb{T}^d)}\right).$$
\end{prop}

\begin{proof}  Let $u_0\in L^2(\mathbb{T}^d)$, $\eta\in(0,1)$, $\epsilon\in(0,1)$, and $T>0$ be arbitrary.  For arbitrary $M\geq 1$ and $\delta\in(0,1)$, the existence of a smooth solution
$$u^{M,\delta}\in\left( \C^{2,1}\left(\mathbb{T}^d\times(0,T)\right)\cap L^2\left([0,T];H^1(\mathbb{T}^d)\right)\right)$$
to the smoothed equation
\begin{equation}\label{smooth_eq_1}\left\{\begin{array}{ll} \partial_t u^{M,\delta}=\Delta\phi^{M,\delta}(u^{M,\delta})+\eta\Delta u^{M,\delta}+\nabla\cdot\left(A(x,u^{M,\delta})\dot{z}^\epsilon_t\right) & \textrm{in}\;\;\mathbb{T}^{d}\times(0,\infty), \\ u^{M,\delta}=u_0 & \textrm{on}\;\;\mathbb{T}^d\times\{0\},\end{array}\right.\end{equation}
follows from Ladyzenskaja, Solonnikov, and Uraltceva \cite[Chapter~V]{LSUBook}, the definition of the smooth nonlinearity (\ref{e_convolve}), the smooth noise $z^\epsilon$, the $\eta$-perturbation by the Laplacian, and the regularity assumption \eqref{prelim_regular}.

In view of \eqref{prelim_regular}, it is immediate from the {\color{black}maximum principle} that, for $C=C(\epsilon,T)>0$,
\begin{equation}\label{sm_00} \norm{u^{M,\delta}}_{L^\infty(\mathbb{T}^d\times[0,T])}\leq C\norm{u_0}_{L^\infty(\mathbb{T}^d)}.\end{equation}
After testing (\ref{smooth_eq_1}) with $u$, it follows from Gr\"{o}nwall's inequality, H\"older's inequality, Young's inequality, and (\ref{prelim_vanish}) that, for $C=C(\epsilon,T)>0$,
\begin{equation}\label{sm_0} \norm{u^{M,\delta}}_{L^\infty\left([0,T];L^2(\mathbb{T}^d)\right)}+\norm{\eta^\frac{1}{2}\nabla u^{M,\delta}}_{L^2\left([0,T];L^2(\mathbb{T}^d;\mathbb{R}^d)\right)}  \leq C\norm{u_0}_{L^2(\mathbb{T}^d)}.\end{equation}
Furthermore, in view of estimates (\ref{sm_00}) and (\ref{sm_0}), it follows from H\"older's inequality, Young's inequality, (\ref{prelim_regular}), and (\ref{prelim_vanish}) that, after testing equation (\ref{smooth_eq_1}) with $\phi^{M,\delta}(u^{M,\delta})$, for the anti-derivative $\psi^{M,\delta}$ from (\ref{e_anti_1}), for $C=C(\epsilon,T)>0$,
\begin{equation}\label{sm_1}\begin{aligned}& \norm{\psi^{M,\delta}(u^{M,\delta})}_{L^\infty\left([0,T];L^1(\mathbb{T}^d)\right)}+\norm{\nabla\phi^{M,\delta}(u^{M,\delta})}^2_{L^2\left([0,T];L^2(\mathbb{T}^d;\mathbb{R}^d)\right)} \\  & \leq \norm{\psi^{M,\delta}(u_0)}_{L^1(\mathbb{T}^d)}+C\norm{u^{M,\delta}}^2_{L^\infty([0,T];L^2(\mathbb{T}^d))} \\ &  \leq  C\left(\norm{\psi^{M,\delta}(u_0)}_{L^1(\mathbb{T}^d)}+\norm{u_0}^2_{L^2(\mathbb{T}^d)}\right).\end{aligned}\end{equation}
Therefore, in combination, estimates \eqref{sm_0} and \eqref{sm_1} imply that, for $C=C(\epsilon,T)>0$,
\begin{equation}\label{sm_2} \norm{\partial_t u^{M,\delta}}^2_{L^2([0,T];H^{-1}(\mathbb{T}^d))}\leq C\left(\norm{\psi^{M,\delta}(u_0)}_{L^1(\mathbb{T}^d)}+\norm{u_0}^2_{L^2(\mathbb{T}^d)}\right).\end{equation}

The combination of estimates (\ref{sm_0}), (\ref{sm_1}), and (\ref{sm_2}) together with the Aubins-Lions-Simon lemma, \cite{Aubin}, \cite{pLions}, and \cite{Simon}, imply that the collection
$$\left\{u^{M,\delta}\right\}_{M\geq 1, \delta\in(0,1)},$$
 is relatively pre-compact in $L^2([0,T];\mathbb{T}^d)$.  Therefore, after passing to a subsequence
 $$\left\{(M_k,\delta_k)\rightarrow(\infty,0)\right\}_{k=1}^\infty,$$
 there exists
 $$u\in \left(L^2\left(([0,T];H^1(\mathbb{T}^d)\right))\cap L^\infty\left(([0,T];L^2(\mathbb{T}^d)\right)\right)\;\;\textrm{with}\;\;\partial_t u\in L^2([0,T];H^{-1}(\mathbb{T}^d)),$$
 such that, as $k\rightarrow\infty$,
 \begin{equation}\label{sm_3}\begin{array}{ll}  u^{M_k,\delta_k}\rightarrow u  & \textrm{strongly in}\;\;L^2\left([0,T];L^2(\mathbb{T}^d)\right), \\ u^{M_k,\delta_k}\rightharpoonup u & \textrm{weakly in}\;\; L^2\left([0,T];H^1(\mathbb{T}^d)\right), \\ \partial_t u^{M_k,\delta_k}\rightharpoonup \partial_t u & \textrm{weakly in}\;\;L^2\left([0,T];H^{-1}(\mathbb{T}^d)\right).\end{array}\end{equation}
 The convergence (\ref{sm_3}) and \cite[Chapter~V]{LSUBook} imply that $u$ is a classical solution of (\ref{smooth_eq}).

 It is immediate from (\ref{sm_00}) and the strong convergence of (\ref{sm_3}) that, for $C=C(\epsilon,T)>0$,
\begin{equation}\label{sm_4} \norm{u}_{L^\infty(\mathbb{T}^d\times[0,T])}\leq C\norm{u_0}_{L^\infty(\mathbb{T}^d)}.\end{equation}
Definitions \eqref{e_bdd}, \eqref{e_convolve}, \eqref{e_anti}, and \eqref{e_anti_1}, estimates (\ref{sm_0}) and (\ref{sm_1}), the convergence (\ref{sm_3}), and the weak lower-semicontinuity of the norm imply that, for $C=C(\epsilon,T)>0$,
{\color{black}\begin{equation}\label{sm_5}\norm{u}^2_{L^\infty\left([0,T];L^2(\mathbb{T}^d)\right)}+\norm{\eta^\frac{1}{2}\nabla u}^2_{L^2\left([0,T];L^2(\mathbb{T}^d;\mathbb{R}^d)\right)}\leq C\norm{u_0}^2_{L^2(\mathbb{T}^d)}.\end{equation}}
Similarly, it follows from estimate (\ref{sm_1}) and the convergence (\ref{sm_3}) that, for $C=C(\epsilon,T)>0$,
\begin{equation}\label{sm_6}\norm{u}^{m+1}_{L^\infty\left([0,T];L^{m+1}(\mathbb{T}^d)\right)}+\norm{\nabla u^{[m]}}^2_{L^2\left([0,T];L^2(\mathbb{T}^d)\right)}\leq C\left(\norm{u_0}^{m+1}_{L^{m+1}(\mathbb{T}^d)}+\norm{u_0}^2_{L^2(\mathbb{T}^d)}\right).\end{equation}
Equation (\ref{smooth_eq}) and estimates (\ref{sm_5}) and (\ref{sm_6}) then imply that, for $C=C(\epsilon,T)>0$,
\begin{equation}\label{sm_7}\norm{\partial_t u}^2_{L^2\left([0,T];H^{-1}(\mathbb{T}^d)\right)}\leq C\left(\norm{u_0}^{m+1}_{L^{m+1}(\mathbb{T}^d)}+\norm{u_0}^2_{L^2(\mathbb{T}^d)}\right).\end{equation}
{\color{black}Lastly, after testing equation \eqref{smooth_eq} with $u$, which is justified by estimates \eqref{sm_5}, \eqref{sm_6}, and \eqref{sm_7}, it follows from H\"older's inequality, Young's inequality, \eqref{sm_5}, and \eqref{sm_6} that, for $C=C(\epsilon,T,m)>0$,
\begin{equation}\label{sm_8}\begin{aligned}\norm{\nabla u^{\left[\frac{m+1}{2}\right]}}^2_{L^2([0,T];L^2(\mathbb{T}^d;\mathbb{R}^d))} & =\frac{(m+1)^2}{4m}\int_0^T\int_{\mathbb{T}^d}\nabla u^{m}\cdot\nabla u\dx\dt \\ &  \leq C\left(\norm{u_0}^{m+1}_{L^{m+1}(\mathbb{T}^d)}+\norm{u_0}^2_{L^2(\mathbb{T}^d)}\right). \end{aligned}\end{equation}}
The convergence \eqref{sm_3} and estimates \eqref{sm_4}, \eqref{sm_5}, \eqref{sm_6}, \eqref{sm_7}, and \eqref{sm_8} complete the proof.  \end{proof}

In Section~\ref{sol_exists}, estimates were obtained for the solutions of (\ref{smooth_eq_1}) which are stable with respect to the $\eta$-perturbation by the Laplacian.  To obtain these estimates, it was necessary to pass to the kinetic formulation of (\ref{smooth_eq_1}), and to subsequently analyze the underlying stochastic characteristics.  It remains only to derive the kinetic equation associated to \eqref{smooth_eq}.

The following approach follows the general strategy of \cite{ChenPerthame}.  However, in our case, we must account for the $x$-dependence of the equation and the unbounded porous medium nonlinearity.  Fix $\eta,\epsilon\in(0,1)$.  Let $u^{\eta,\epsilon}$ denote a solution of
\begin{equation}\label{entropy_eq}\left\{\begin{array}{ll}\partial_t u=\Delta u^{[m]}+\eta\Delta u+\nabla\cdot(A(x,u)\dot{z}^\epsilon_t) & \textrm{in}\;\;\mathbb{T}^d\times(0,\infty), \\ u=u_0 & \textrm{on}\;\;\mathbb{T}^d\times\left\{0\right\}.\end{array}\right.\end{equation}
In order to expand the divergence appearing in (\ref{entropy_eq}), we define the matrix-valued function
\begin{equation}\label{kin_b}b(x,\xi)=(b_{ij}(x,\xi)):=\partial_\xi A(x,\xi)\in\mathcal{M}^{d\times n}\;\;\textrm{for each}\;\;(x,\xi)\in\mathbb{T}^d\times\mathbb{R},\end{equation}
and the vector-valued
\begin{equation}\label{kin_c}c(x,\xi)=(c_i(x,\xi)):=\left(\sum_{i=1}^d\partial_{x_i}a_{ij}(x,\xi)\right)\in\mathbb{R}^n\;\;\textrm{for each}\;(x,\xi)\in\mathbb{T}^d\times\mathbb{R}.\end{equation}
In combination, (\ref{entropy_eq}), (\ref{kin_b}), and (\ref{kin_c}) yield the equation
\begin{equation}\label{entropy_expand}\left\{\begin{array}{ll}\partial_t u=\Delta u^{[m]}+\eta\Delta u +b(x,u)\dot{z}^\epsilon_t\cdot \nabla u+c(x,u)\cdot \dot{z}^\epsilon_t & \textrm{in}\;\;\mathbb{T}^d\times(0,\infty), \\ u=u_0 & \textrm{on}\;\;\mathbb{T}^d\times\left\{0\right\}.\end{array}\right.\end{equation}

The entropy formulation of (\ref{entropy_expand}) is based upon studying the equations satisfied by compositions $S(u^{\eta,\epsilon})$, for smooth functions $S:\mathbb{R}\rightarrow\mathbb{R}$ which are convex and satisfy $S(0)=S'(0)=0$.  Indeed, after multiplying (\ref{entropy_expand}) by the composition $S'(u^{\eta,\epsilon})$, the chain rule implies that $S(u^{\eta,\epsilon})$ is a solution of the equation
\begin{equation}\label{kin_comp}\begin{aligned} \partial_tS(u^{\eta,\epsilon}) =  & \nabla\cdot\left(m\abs{u^{\eta,\epsilon}}^{m-1}\nabla S(u^{\eta,\epsilon})\right)+\eta\Delta S(u^{\eta,\epsilon}) +b(x,u^{\eta,\epsilon})\dot{z}^\epsilon_t\cdot \nabla S(u^{\eta,\epsilon}) \\ & +\left(c(x,u^{\eta,\epsilon})\cdot \dot{z}^\epsilon_t\right)S'(u^{\eta,\epsilon})  - S''(u^{\eta,\epsilon})m\abs{u^{\eta,\epsilon}}^{m-1}\abs{\nabla u^{\eta,\epsilon}}^2-S''(u^{\eta,\epsilon})\eta\abs{\nabla u^{\eta,\epsilon}}^2,\end{aligned}\end{equation}
on $\mathbb{T}^d\times(0,\infty)$, with initial data $S(u_0)$.   The kinetic formulation of (\ref{kin_comp}), through the introduction of an additional velocity variable $\xi\in\mathbb{R}$, replaces the ensemble of equations (\ref{kin_comp}), as defined by the collection of entropies $\{S\}$, by a single equation in $(d+1)$-variables.  This is effectively achieved by factoring out $S'(u)$.

Precisely, define the kinetic function $\overline{\chi}:\mathbb{R}^2\rightarrow\mathbb{R}$ by the rule
\begin{equation}\label{kinetic_function_foundation}\overline{\chi}(s,\xi):=\left\{\begin{array}{ll} 1 & \textrm{if}\;\;0<\xi<s, \\ -1 & \textrm{if}\;\;s<\xi<0, \\ 0 & \textrm{else,}\end{array}\right.\end{equation}
and consider the composition
\begin{equation}\label{kinetic_function} \chi^{\eta,\epsilon}(x,\xi,t):=\overline{\chi}(u^{\eta,\epsilon}(x,t),\xi)\;\;\textrm{for}\;\;(x,\xi,t)\in\mathbb{T}^d\times\mathbb{R}\times[0,\infty).\end{equation}
The identity, for each smooth $S:\mathbb{R}\rightarrow\mathbb{R}$ satisfying $S(0)=0$,
$$S(u^{\eta,\epsilon})=\int_{\mathbb{R}}S'(\xi)\chi^{\eta,\epsilon}(x,\xi,t)\;d\xi\;\;\textrm{for}\;\;x\in\mathbb{T}^d\;\;\textrm{and}\;\;t\in[0,\infty),$$
then suggests that, since $S$ can be an arbitrary smooth, convex function satisfying $S(0)=S'(0)=0$, the kinetic function $\chi^{\eta,\epsilon}$ is a solution of the equation
\begin{equation}\label{kin_eq}\begin{aligned} \partial_t\chi^{\eta,\epsilon}= & m\abs{\xi}^{m-1}\Delta_x\chi^{\eta,\epsilon}+\eta\Delta_x\chi^{\eta,\epsilon} +b(x,\xi)\dot{z}^\epsilon_t\cdot \nabla_x\chi^{\eta,\epsilon}- \left(c(x,\xi)\cdot\dot{z}^\epsilon_t\right)\partial_\xi\chi^{\eta,\epsilon} \\ & +\partial_\xi p^{\eta,\epsilon}(x,\xi,t)+\partial_\xi q^{\eta,\epsilon}(x,\xi,t), \end{aligned}\end{equation}
on $\mathbb{T}^d\times\mathbb{R}\times(0,\infty)$, with initial data $(x,\xi)\in\mathbb{T}^d\times\mathbb{R}\mapsto \overline{\chi}(u_0(x),\xi)$, for the entropy defect measure
\begin{equation}\label{ent_measure}p^{\eta,\epsilon}(x,\xi,t):=\delta_0\left(\xi-u^{\eta,\epsilon}(x,t)\right)\eta\abs{\nabla u^{\eta,\epsilon}}^2\;\;\textrm{for each}\;\;(x,\xi,t)\in\mathbb{T}^d\times\mathbb{R}\times[0,\infty),\end{equation}
and for the parabolic defect measure
\begin{equation}\label{parab_measure}q^{\eta,\epsilon}(x,\xi,t):=\delta_0\left(\xi-u^{\eta,\epsilon}(x,t)\right)\frac{4m}{(m+1)^2}\abs{\nabla\left(u^{\eta,\epsilon}\right)^{\left[\frac{m+1}{2}\right]}}^2\;\;\textrm{for each}\;\;(x,\xi,t)\in\mathbb{T}^d\times\mathbb{R}\times[0,\infty),\end{equation}
where $\delta_0$ is one-dimensional Dirac mass centered at the origin.  The following proposition proves that this is indeed the case.

\begin{prop}\label{smooth_kinetic_solution}  For each $\eta\in(0,1)$, $\epsilon\in(0,1)$, and $u_0\in L^2(\mathbb{T}^d)$, let $u^{\eta,\epsilon}$ denote a solution of \eqref{smooth_eq} from Proposition \ref{smooth_equation}.  Then, the kinetic function $\chi^{\eta,\epsilon}$ defined in \eqref{kinetic_function} is a distributional solution of \eqref{kin_eq} in the sense that, for every $t_1,t_2\in[0,\infty)$, for every $\psi\in\C^\infty_c(\mathbb{T}^d\times\mathbb{R}\times[t_1,t_2]))$,
\begin{equation}\begin{aligned}&\left.\int_\mathbb{R}\int_{\mathbb{T}^d}\chi^{\eta,\epsilon}(x,\xi,t)\psi(x,\xi,t)\dx\dxi\right|_{t=t_1}^{t_2} =  \int_{t_1}^{t_2}\int_\mathbb{R}\int_{\mathbb{T}^d}\chi^{\eta,\epsilon} \partial_t\psi\dx\dxi\dt  \\ &+\int_{t_1}^{t_2}\int_\mathbb{R}\int_{\mathbb{T}^d}m\abs{\xi}^{m-1}\chi^{\eta,\epsilon} \Delta_x\psi+\eta\chi^{\eta,\epsilon}\Delta_x\psi \dx\dxi\dt \\ &- \int_{t_1}^{t_2}\int_\mathbb{R}\int_{\mathbb{T}^d}\chi^{\eta,\epsilon}\nabla_x\cdot\left(\left(b(x,\xi)\dot{z}^\epsilon_t\right)\psi\right)-\chi^{\eta,\epsilon}\partial_\xi\left(\left(c(x,\xi)\cdot \dot{z}^\epsilon_t\right)\psi\right)\dx\dxi\dt \\ &-\int_{t_1}^{t_2}\int_\mathbb{R}\int_{\mathbb{T}^d}\left(p^{\eta,\epsilon}+q^{\eta,\epsilon}\right)\partial_\xi\psi \dx\dxi\dt. \end{aligned}\end{equation}
\end{prop}

\begin{proof}  Let $\eta\in(0,1)$, $\epsilon\in(0,1)$, $u_0\in L^2(\mathbb{T}^d)$, and $t_1,t_2\in[0,\infty)$ be arbitrary.  Let $u^{\eta,\epsilon}$ denote a solution of \eqref{smooth_eq} satisfying the estimates of Proposition \ref{smooth_equation}, and let $\chi^{\eta,\epsilon}$ denote its kinetic function defined in (\ref{kinetic_function}).  The estimates of Proposition \ref{smooth_eq} imply that, for every $\psi\in\C^\infty_c(\mathbb{T}^d\times\mathbb{R}\times[t_1,t_2])$, the composition $(x,t)\in\mathbb{R}^d\times[t_1,t_2]\mapsto \psi(x,u^{\eta,\epsilon}(x,t),t)$ is an admissable test function for (\ref{smooth_eq}).

It is necessary to use the following identity, which holds for every for every $\psi\in\C^\infty_c(\mathbb{T}^d\times\mathbb{R}\times[t_1,t_2])$, for each $(x,t)\in\mathbb{T}^d\times[0,\infty)$,
\begin{equation}\label{sks_000}\partial_tu^{\eta,\epsilon}(x,t)\partial_\xi \psi(x,u^{\eta,\epsilon}(x,t),t)=\partial_t\left(\psi(x,u^{\eta,\epsilon}(x,t),t)\right)-(\partial_t\psi)(x,u^{\eta,\epsilon}(x,t),t).\end{equation}
It follows from (\ref{sks_000}) that, for any $\psi\in\C^\infty_c(\mathbb{T}^d\times\mathbb{R}\times[t_1,t_2])$, after defining
\begin{equation}\label{sks_00}\tilde{\psi}(x,\xi,t):=\int_0^\xi\psi(x,\xi',t)\dxip\;\;\textrm{for}\;\;(x,\xi,t)\in\mathbb{T}^d\times\mathbb{R}\times[t_1,t_2],\end{equation}
and testing equation \eqref{smooth_eq} with the composition $(x,t)\in\mathbb{R}^d\times[0,\infty)\mapsto \psi(x,u^{\eta,\epsilon}(x,t),t)$,
\begin{equation}\begin{aligned}\label{sks_0} & \left.\int_{\mathbb{T}^d}\tilde{\psi}(x,u^{\eta,\epsilon}(x,t),t)\dx\right|_{t=t_1}^{t_2} = \int_{t_1}^{t_2}\int_{\mathbb{T}^d}\left(\partial_t\tilde{\psi}\right)(x,u^{\eta,\epsilon}(x,t),t)\dx\dt  \\ & \quad -\int_{t_1}^{t_2} \int_{\mathbb{T}^d}\nabla\left(u^{\eta,\epsilon}\right)^{[m]}\cdot \left(\left(\nabla_x\psi\right)(x,u^{\eta,\epsilon}(x,t),t)+\partial_\xi\psi(x,u^{\eta,\epsilon}(x,t),t)\nabla u^{\eta,\epsilon}(x,t)\right) \dx\dt \\ & \quad -\int_{t_1}^{t_2}\int_{\mathbb{T}^d} \eta\nabla u^{\eta,\epsilon}\cdot \left(\left(\nabla_x\psi\right)(x,u^{\eta,\epsilon}(x,t),t)+\partial_\xi\psi(x,u^{\eta,\epsilon}(x,t),t)\nabla u^{\eta,\epsilon}(x,t)\right) \dx\dt \\ & \quad +\int_{t_1}^{t_2}\int_{\mathbb{T}^d}\left(b(x,u^{\eta,\epsilon})\dot{z}^\epsilon_t\cdot \nabla u^{\eta,\epsilon}\right)\psi(x,u^{\eta,\epsilon}(x,t),t)\dx\dt \\ & \quad +\int_{t_1}^{t_2}\int_{\mathbb{T}^d}\left(c(x,u^{\eta,\epsilon})\cdot \dot{z}^\epsilon_t\right)\psi(x,u^{\eta,\epsilon}(x,t),t)\dx\dt.\end{aligned} \end{equation}
The estimates of Proposition~\ref{smooth_equation}, in particular the fact that, for each $T>0$,
$$u^{\eta,\epsilon}\in L^2\left([0,T];H^1(\mathbb{R}^d)\right),$$
and definition \eqref{kinetic_function} imply that the kinetic function $\chi^{\eta,\epsilon}$ satisfies the distributional equalities, for $(x,\xi,t)\in\mathbb{T}^d\times\mathbb{R}\times[0,\infty)$,
\begin{equation}\label{sks_1} \nabla_x\chi^{\eta,\epsilon}(x,\xi,t)=\delta_0\left(\xi-u^{\eta,\epsilon}(x,t)\right)\nabla u^{\eta,\epsilon}(x,t)\;\;\textrm{and}\;\;\partial_\xi\chi^{\eta,\epsilon}(x,\xi,t)=\delta_0(\xi)-\delta_0\left(\xi-u^{\eta,\epsilon}(x,t)\right).\end{equation}
The essential point is that $u^{\eta,\epsilon}$ has a distributional derivative, and it is for this reason that the $\eta$-perturbation by the Laplacian is retained.

Therefore, returning to \eqref{sks_0}, it follows by definition of the kinetic function and the definition of $\tilde{\psi}$ from (\ref{sks_00}) that, for each $t\in[t_1,t_2]$,
\begin{equation}\label{sks_2}\int_{\mathbb{T}^d}\tilde{\psi}(x,u(x,t),t))\dx= \int_\mathbb{R}\int_{\mathbb{T}^d}\partial_\xi\tilde{\psi}(x,\xi,t)\overline{\chi}(u^{\eta,\epsilon}(x,t),\xi)\dx\dxi =\int_{\mathbb{T}^d}\psi(x,\xi,t)\chi^{\eta,\epsilon}(x,\xi,t)\dx\dxi ,\end{equation}
and
\begin{equation}\label{sks_3}\begin{aligned} \int_{t_1}^{t_2}\int_{\mathbb{T}^d}\left(\partial_t\tilde{\psi}\right)(x,u^{\eta,\epsilon}(x,t),t)\dx\dt= & \int_{t_1}^{t_2}\int_\mathbb{R}\int_{\mathbb{T}^d}\partial_t\partial_\xi\tilde{\psi}(x,\xi,t)\chi^{\eta,\epsilon}(x,\xi,t)\dx\dxi\dt \\ = & \int_{t_1}^{t_2}\int_\mathbb{R}\int_{\mathbb{T}^d}\partial_t\psi(x,\xi,t)\chi^{\eta,\epsilon}(x,\xi,t)\dx\dxi\dt. \end{aligned}\end{equation}
The identity $\nabla \left(u^{\eta,\delta}\right)^{[m]}=m\abs{u^{\eta,\epsilon}}^{m-1}\nabla u^{\eta,\epsilon},$ the definition of the parabolic defect measure \eqref{parab_measure}, and the distributional inequality \eqref{sks_1} imply that
\begin{equation}\label{sks_4}\begin{aligned} &\int_{t_1}^{t_2} \int_{\mathbb{T}^d}\nabla\left(u^{\eta,\epsilon}\right)^{[m]}\cdot \left(\left(\nabla_x\psi\right)(x,u^{\eta,\epsilon}(x,t),t)+\partial_\xi\psi(x,u^{\eta,\epsilon}(x,t),t)\nabla u^{\eta,\epsilon}(x,t)\right) \dx\dt  \\ & =\int_{t_1}^{t_2}\int_\mathbb{R}\int_{\mathbb{T}^d}m\abs{\xi}^{m-1}\nabla_x\chi^{\eta,\epsilon}(x,\xi,t)\nabla_x\psi(x,\xi,t)+q^{\eta,\epsilon}(x,\xi,t)\partial_\xi\psi(x,\xi,t)\dx\dxi\dt, \end{aligned}\end{equation}
and the definition of the entropy defect measure \eqref{ent_measure} implies that
\begin{equation}\begin{aligned}\label{sks_5} & \int_{t_1}^{t_2}\int_{\mathbb{T}^d} \eta\nabla u^{\eta,\epsilon}\cdot \left(\left(\nabla_x\psi\right)(x,u^{\eta,\epsilon}(x,t),t)+\partial_\xi\psi(x,u^{\eta,\epsilon}(x,t),t)\nabla u^{\eta,\epsilon}(x,t)\right) \dx\dt \\ & =\int_{t_1}^{t_2}\int_\mathbb{R}\int_{\mathbb{T}^d}\eta\nabla_x\chi^{\eta,\epsilon}(x,\xi,t)\cdot\nabla_x\psi(x,\xi,t)+p^{\eta,\epsilon}(x,\xi,t)\partial_\xi\psi(x,\xi,t)\dx\dxi\dt.\end{aligned}\end{equation}
It is immediate from the distributional equality \eqref{sks_1} that
\begin{equation}\begin{aligned}\label{sks_6} & \int_{t_1}^{t_2}\int_{\mathbb{T}^d}\left(b(x,u^{\eta,\epsilon})\dot{z}^\epsilon_t\cdot \nabla u^{\eta,\epsilon}\right)\psi(x,u^{\eta,\epsilon}(x,t),t)\dx\dt \\ & =\int_{t_1}^{t_2}\int_\mathbb{R}\int_{\mathbb{T}^d}\left(b(x,\xi)\dot{z}^\epsilon_t\cdot \nabla_x\chi^{\eta,\epsilon}(x,\xi,t)\right)\psi(x,\xi,t)\dx\dxi\dt.\end{aligned}\end{equation}
Finally, assumption \eqref{prelim_vanish} and the distributional equality \eqref{sks_1} imply that
\begin{equation}\begin{aligned}\label{sks_7} & \int_{t_1}^{t_2}\int_{\mathbb{T}^d}\left(c(x,u^{\eta,\epsilon})\cdot \dot{z}^\epsilon_t\right)\psi(x,u^{\eta,\epsilon}(x,t),t)\dx\dt  \\  &= -\int_{t_1}^{t_2}\int_\mathbb{R}\int_{\mathbb{T}^d}\left(c(x,\xi)\cdot\dot{z}^\epsilon_t\right)\partial_\xi\chi^{\eta,\epsilon}(x,\xi,t)\psi(x,\xi,t)\dx\dxi\dt.\end{aligned}\end{equation}
After integrating by parts, equation \eqref{sks_0} and equalities \eqref{sks_2}, \eqref{sks_3}, \eqref{sks_4}, \eqref{sks_5}, \eqref{sks_6}, and \eqref{sks_7} imply that, for every $\psi\in\C^\infty_c(\mathbb{T}^d\times\mathbb{R}\times[t_1,t_2])$,
\begin{equation}\begin{aligned}\label{sks_8}& \left.\int_\mathbb{R}\int_{\mathbb{T}^d}\chi^{\eta,\epsilon}(x,\xi,t)\psi(x,\xi,t)\dx\dxi\right|_{t=t_1}^{t_2}=\int_{t_1}^{t_2}\int_\mathbb{R}\int_{\mathbb{T}^d}\chi^{\eta,\epsilon} \partial_t\psi\dx\dxi\dt  \\ & \quad +\int_{t_1}^{t_2}\int_\mathbb{R}\int_{\mathbb{T}^d}m\abs{\xi}^{m-1}\chi^{\eta,\epsilon} \Delta_x\psi+\eta\chi^{\eta,\epsilon}\Delta_x\psi \dx\dxi\dt \\ & \quad - \int_{t_1}^{t_2}\int_\mathbb{R}\int_{\mathbb{T}^d}\chi^{\eta,\epsilon}\nabla_x\cdot\left(\left(b(x,\xi)\dot{z}^\epsilon_t\right)\psi\right)-\chi^{\eta,\epsilon}\partial_\xi\left(\left(c(x,\xi)\cdot \dot{z}^\epsilon_t\right)\psi\right)\dx\dxi\dt \\ & \quad -\int_{t_1}^{t_2}\int_\mathbb{R}\int_{\mathbb{T}^d}\left(p^{\eta,\epsilon}+q^{\eta,\epsilon}\right)\partial_\xi\psi \dx\dxi\dt. \end{aligned}\end{equation}
This completes the proof.  \end{proof}

\section{Rough Path Estimates}\label{sec_rough}

The theory of rough paths was first introduced by Lyons \cite{Lyons}, and overviews of the theory can be found in Friz and Hairer \cite{FrizHairer} or in Friz and Victoir \cite{FrizVictoir}.  We therefore only sketch some of the main details here.  For the remainder of this section fix $d\geq 1$ and $T\geq 0$.  Let $x\in\C^{1-\textrm{var}}\left([0,T];\mathbb{R}^d\right)$ be a path with bounded $1$-variation.  For each $M\geq 1$ the $M$-step signature of $x$ is defined as
$$S_M(x)_{0,T}:=\left(1,\int_0^T\dd x_s,\int_{0<s_0<s_1<T}\dd x_{s_1}\otimes \dd x_{s_2},\ldots,\int_{0<s_1<\ldots<s_M<T} \dd x_{s_1}\otimes\ldots\otimes \dd x_{s_M}\right).$$
It is immediate from the definition that $S_M(x)_{0,T}$ takes values in the the truncated $M$-step tensor algebra
$$T^M(\mathbb{R}^d):=\mathbb{R}\oplus \mathbb{R}^d\oplus \left(\mathbb{R}^d\right)^{\otimes 2}\oplus\ldots\oplus\left(\mathbb{R}^d\right)^{\otimes M}.$$
Following a reparametrization of the path, it follows that $S^M(x)_{0,T}$ actually lies in the smaller space $G^M(\mathbb{R}^d)\subset T^M(\mathbb{R}^d)$ defined by
$$G^M(\mathbb{R}^d):=\left\{\;S^M(x)_{0,1}\;|\;x\in\C^{1-\textrm{var}}\left([0,1];\mathbb{R}^d\right)\;\right\}.$$
The space $G^M(\mathbb{R}^d)$ comes equipped with the so-called Carnot-Caratheodory norm, for $\sigma\in G^M(\mathbb{R}^d)$,
{\color{black}$$\norm{\sigma}_{CC}=\inf\left\{\;\int_0^1\abs{\dot{\gamma}}\ds\;|\;\gamma\in\C^{1-\textrm{var}}\left([0,1];\mathbb{R}^d\right)\;\;\textrm{and}\;\;S^M(\gamma)_{0,1}=\sigma\;\right\}.$$}
This norm defines a homogenous on the space $G^M(\mathbb{R}^d)$.  We remark that an inhomogenous but equivalent norm can also be chosen by defining the norm of an element $\sigma\in G^M(\mathbb{R}^d)$ to be the supremum of the respective $L^\infty$-norms of its components.

{\color{black}The Carnot-Caratheodory norm induces, following \cite[Definition~7.41]{FrizVictoir}, the Carnot-Caratheodory metric $d_{CC}$ on $G^M(\mathbb{R}^d)$.  For $\beta\in(0,1)$, the homogenous $\beta$-H\"older metric, for $\beta\in(0,1)$ and paths $z,w$ taking values in $G^M(\mathbb{R}^d)$, is defined as
$$d_\beta(z,w):=\sup_{0\leq s\leq t\leq 1}\frac{d_{CC}(z_{t,s},w_{t,s})}{\abs{t-s}^\beta}.$$}
For $\beta\in(0,1)$, a geometric $\beta$-H\"older continuous rough path is a path $z$ taking values in $T^{\left\lfloor\frac{1}{\beta}\right\rfloor}(\mathbb{R}^d)$ which can be approximated by the signatures of smooth paths with respect to the $\beta$-H\"older metric $d_\beta$.  Precisely, a path $z:[0,T]\rightarrow T^{\left\lfloor\frac{1}{\beta}\right\rfloor}(\mathbb{R}^d)$ is a geometric rough path if there exists a sequence of smooth paths $\{z^n:[0,T]\rightarrow\mathbb{R}^d\}$ such that, as $n\rightarrow\infty$,
\begin{equation}\label{geometric_path}d_\beta\left(z, S_{\left\lfloor\frac{1}{\beta}\right\rfloor}(z^n)\right)\rightarrow 0.\end{equation}
It can be shown that $\beta$-H\"older continuous geometric rough paths take values in the space $G^{\left\lfloor\frac{1}{\beta}\right\rfloor}(\mathbb{R}^d)$.  We will denote by $\C^{0,\beta}([0,T];G^{\left\lfloor\frac{1}{\beta}\right\rfloor}(\mathbb{R}^d))$ the space of $\beta$-H\"older continuous geometric rough paths starting at zero.

In the final part of this section, we will recall some stability estimates for the solutions of rough differential equations.  For each $x\in\mathbb{R}^d$ and $z\in \C^{0,\beta}([0,T];G^{\left\lfloor\frac{1}{\beta}\right\rfloor}(\mathbb{R}^d)))$, for some $\beta\in(0,1)$, let $X^{x,z}$ be the solution of the equation
\begin{equation}\label{rough_eq} \left\{\begin{array}{ll} dX^{x,z}_t=V(X^{x,z}_t)\circ dz_t & \textrm{on}\;\;(0,\infty), \\ X^{x,z}_0=x. & \end{array}\right.\end{equation}
The ensemble \eqref{rough_eq} defines a flow map $\psi^z:\mathbb{R}^d\times[0,T]\rightarrow\mathbb{R}^d$ by the rule
$$\psi^z_t(x)=X^{x,z}_t\;\;\textrm{for}\;\;(x,t)\in\mathbb{R}^d\times[0,T].$$
The following proposition encodes the regularity of the flow map with respect to the initial condition and the driving signal.  The regularity is inherited from the nonlinearity $V$, which must be sufficiently regular to overcome the roughness of the noise.  A proof of the following proposition can be found in Crisan, Diehl, Friz, and Oberhauser \cite[Lemma~13]{CrisanDiehlFrizOberhauser}.  In the statement below, we will write $e=1\oplus0\oplus\dots\oplus0$ to denote the signature of the zero path.

\begin{prop}\label{rough_est} Fix $T\geq 0$, $\beta\in(0,1)$, $\gamma>\frac{1}{\beta}\geq 1$, and $k\in\mathbb{N}$.  Assume $V\in\Lip^{\gamma+k}(\mathbb{R}^d;\mathbb{R}^d)$, and for a $R\geq 0$, assume that $z^1, z^2\in\C^{0,\beta}\left([0,T]; G^{\left\lfloor\frac{1}{\beta}\right\rfloor}(\mathbb{R}^d)\right)$ with, for each $j\in\{1,2\}$,
\begin{equation}\label{rough_est_1}d(z^j,e)_\beta\leq R.\end{equation}
There exist $C=C(R,\norm{V}_{\Lip^{\gamma+k}})>0$ and $K=K(R,\norm{V}_{\Lip^{\gamma+k}})>0$ independent of $z^1,z^2$ satisfying \eqref{rough_est_1} such that, for all $n\in\{0,\ldots,k\}$,
\begin{equation}\label{rough_est_2}\sup_{x\in\mathbb{R}^d}\norm{D^n(\psi^{z_1}_t-\psi^{z_2}_t)(x)}_\beta\leq Cd_\beta(z^1,z^2),\end{equation}
and
\begin{equation}\label{rough_est_3}\sup_{x\in\mathbb{R}^d}\norm{D^n((\psi^{z_1}_t)^{-1}-(\psi^{z_2}_t)^{-1})(x)}_\beta\leq Cd_\beta(z^1,z^2).\end{equation}
Furthermore, for each $n\in\{0,\ldots,k\}$,
\begin{equation}\label{rough_est_4}\sup_{x\in\mathbb{R}^d}\norm{D^n\psi^{z_1}_t(x)}_\beta \leq K \;\;\textrm{and}\;\;\sup_{x\in\mathbb{R}^d}\norm{D^n(\psi^{z_1}_t)^{-1}(x)}_\beta\leq K.\end{equation}
\end{prop}

We conclude this section with a lemma which asserts that the characteristics in velocity are locally in time comparable to their initial condition.

\begin{lem}\label{auxlem}  For each $T>0$ there exists $C=C(T)\geq 1$ such that, for each $(x,\xi)\in\mathbb{T}^d\times\mathbb{R}$ and $t\in[0,T]$,
$$C^{-1}\abs{\xi}\leq \abs{\Pi^{x,\xi}_{t,t}}\leq C\abs{\xi}.$$
Furthermore, there exists $C=C(T)>0$ such that, for each $(x,\xi)\in\mathbb{T}^d\times\mathbb{R}$ and $t\in[0,T]$, for $\alpha\in(0,\frac{1}{2})$ from \eqref{prelim_Holder},
$$\abs{\nabla\Pi^{x,\xi}_{t,t}}\leq Ct^\alpha\left(\abs{\xi}\wedge1\right).$$
\end{lem}

\begin{proof}  The proof is a consequence of assumption \eqref{prelim_regular} and the estimates of Proposition~\ref{rough_est}.  There exists $t_*\in(0,\infty)$ such that, for each $(x,\xi,t)\in\mathbb{T}^d\times\mathbb{R}\times[0,\infty)$, for each $s\in[0,t_*\wedge t]$,
\begin{equation}\label{auxlem_1}\frac{1}{2}\leq \partial_\xi\Pi^{x,\xi}_{t,s}\leq \frac{3}{2}.\end{equation}
The proof will follow by induction.  For the base case, observe that, since for each $x\in\mathbb{T}^d$ and $t\geq 0$ we have $\Pi^{x,0}_{t,t}=0$, it follows by integration and \eqref{auxlem_1} that there exists $C=C(t_*)>0$ such that, for each $(x,\xi)\in\mathbb{T}^d\times\mathbb{R}$ and $t\in[0,t_*]$,
\begin{equation}\label{auxlem_2}C^{-1}\abs{\xi}\leq \abs{\Pi^{x,\xi}_{t,t}}\leq C\abs{\xi}.\end{equation}
For the inductive statement, suppose that for some $k\in\mathbb{N}$, there exists $C=C(kt_*)>0$ such that, for each $(x,\xi,t)\in\mathbb{T}^d\times\mathbb{R}\times[0,kt_*]$,
\begin{equation}\label{auxlem_3}C^{-1}\abs{\xi}\leq \abs{\Pi^{x,\xi}_{t,t}}\leq C\abs{\xi}.\end{equation}
The semigroup property implies that, for each $(x,\xi,t)\in\mathbb{T}^d\times\mathbb{R}\times[kt_*,(k+1)t_*]$,
$$\Pi^{x,\xi}_{t,t}=\Pi^{Y^{x,\xi}_{t-t_*,t-t_*},\Pi^{x,\xi}_{t-t_*,t-t_*}}_{t,t_*}.$$
It follows from \eqref{auxlem_1}, the fact that $\Pi^{Y^{x,\xi}_{t-t_*,t-t_*},0}_{t,t_*}=0$, and integration that, for $C\geq1$, for each $(x,\xi,t)\in\mathbb{T}^d\times\mathbb{R}\times[kt_*,(k+1)t_*]$,
$$C^{-1}\abs{\Pi^{x,\xi}_{t-t_*,t-t_*}}\leq \abs{\Pi^{x,\xi}_{t,t}}\leq C\abs{\Pi^{x,\xi}_{t-t_*,t-t_*}}.$$
Finally, since $t-t_*\in[0,kt_*]$ for each $t\in[kt_*,(k+1)t_*]$, the inductive statement \eqref{auxlem_3} implies that, for $C=C((k+1)t_*)>0$, for each $(x,\xi,t)\in\mathbb{T}^d\times\mathbb{R}\times[kt_*,(k+1)t_*]$,
\begin{equation}\label{auxlem_4}C^{-1}\abs{\xi}\leq \abs{\Pi^{x,\xi}_{t,t}}\leq C\abs{\xi}.\end{equation}
The base case \eqref{auxlem_2} and \eqref{auxlem_4} complete the proof.

The second claim is simpler and follows similarly from assumption \eqref{prelim_regular} and the estimates of Proposition~\ref{rough_est}.  For each $T>0$ there exists $C=C(T)>0$ such that, for each $(x,\xi)\in\mathbb{T}^d\times\mathbb{R}$ and $t\in[0,T]$, for $\alpha\in(0,\frac{1}{2})$ defining the regularity of the noise in \eqref{prelim_Holder},
$$\abs{\partial_\xi\nabla\Pi^{x,\xi}_{t,t}}\leq Ct^\alpha.$$
Therefore, since for each $(x,\xi)\in\mathbb{T}^d\times\mathbb{R}$ and $t\geq 0$, we have $\nabla\Pi^{x,\xi}_{0,0}=0$ and $\nabla\Pi^{x,0}_{t,t}=0$, the claim follows from the estimates of Proposition~\ref{rough_est} and integration.  This completes the proof.  \end{proof}

\section{Fractional Sobolev Regularity of the Kinetic Function}\label{sec_frac}

The purpose of this section is to prove the fractional Sobolev regularity of the kinetic function $\chi$ of a pathwise kinetic solution $u$, in the sense of Definition~\ref{def_solution}.  We will first consider the kinetic function's regularity in the velocity variable where, for each $x\in\mathbb{T}^d$, the map $\xi\in\mathbb{R}\mapsto\chi(x,\xi)$ is the indicator function of either the open interval $(0,u(x))$, if $u(x)\geq 0$, or the open interval $(u(x),0)$.

The first proposition proves that the space of $\BV$ functions locally embeds into the fractional Sobolev space $W^{s,1}$, for every $s\in(0,1)$.  We will apply this to the kinetic function $\chi$ in the corollary to follow, after making the elementary observation that the one-dimensional indicator function of a finite interval is of  bounded variation.

\begin{prop}\label{BVto} Let $d\geq 1$, and suppose that $U\subset\mathbb{R}^d$ is a convex open subset. Then, for every $\psi\in \BV(U)$, and for each $s\in(0,1)$, there exists $C=C(d,s)>0$ such that
$$\norm{\psi}_{W^{s,1}(U)}\leq C\norm{\psi}_{\BV(U)}.$$
\end{prop}

\begin{proof}\label{BV}  Let $U\subset\mathbb{R}^d$ be a convex open subset.  Fix $\psi\in\BV(U)$ and $s\in(0,1)$.  Then, choose a sequence $\left\{\psi_n\right\}_{n=1}^\infty\subset \left(W^{1,1}\cap\C^\infty\right)(U)$ such that, as $n\rightarrow\infty$,
\begin{equation}\label{BV_0}\lim_{n\rightarrow\infty}\norm{\psi-\psi_n}_{L^1(U)}=0\;\;\textrm{and}\;\;\lim_{n\rightarrow\infty}\abs{\norm{\nabla\psi_n}_{L^1(U)}-\abs{\nabla\psi}(U)}=0,\end{equation}
where $\abs{\nabla \psi}(U)$ denotes the measure of $U$ with respect to the total variation of the measure $\nabla\psi$.  This sequence can be constructed, for instance, via convolution.

It is only necessary to estimate the fractional Sobolev semi-norm.  For this, for each $n\geq 0$,
$$\begin{aligned}\int_{U\times U}\frac{\abs{\psi_n(x)-\psi_n(y)}}{\abs{x-y}^{d+s}}\dx\dy = & \int_{\{\abs{x-y}>1\}\cap \left(U\times U\right)}\frac{\abs{\psi_n(x)-\psi_n(y)}}{\abs{x-y}^{d+s}}\dx\dy \\ & + \int_{\{\abs{x-y}\leq 1\}\cap \left(U\times U\right)}\frac{\abs{\psi_n(x)-\psi_n(y)}}{\abs{x-y}^{d+s}}\dx\dy,\end{aligned}$$
and, therefore,
\begin{equation}\label{BV_1}\int_{U\times U}\frac{\abs{\psi_n(x)-\psi_n(y)}}{\abs{x-y}^{d+s}}\dx\dy \leq  2\norm{\psi_n}_{L^1(U)}+\int_{\{\abs{x-y}\leq 1\}\cap\left(U\times U\right)}\frac{\abs{\psi_n(x)-\psi_n(y)}}{\abs{x-y}^{d+s}}\dx\dy. \end{equation}
For the final term on the righthand side of (\ref{BV_1}), the regularity of the $\{\psi_n\}_{n=1}^\infty$ and the convexity of $U$ imply that, for $C=C(d,s)>0$,
\begin{equation}\label{BV_2}\begin{aligned}\int_{\{\abs{x-y}\leq 1\}\cap\left(U\times U\right)}\frac{\abs{\psi_n(x)-\psi_n(y)}}{\abs{x-y}^{d+s}}\leq & \int_{\{\abs{x-y}\leq 1\}\cap\left(U\times U\right)}\int_0^1\abs{x-y}^{1-d-s}\abs{\nabla\psi_n}(x+r(y-x))\dr\\  \leq & \int_{B_1}\abs{x}^{1-d-s}\int_{U}\abs{\nabla\psi_n} \\ \leq & C\norm{\nabla\psi_n}_{L^1(U)}.\end{aligned}\end{equation}

The statement now follows by passing to the limit $n\rightarrow\infty$.  Precisely, the dominated convergence theorem, \eqref{BV_0}, \eqref{BV_1}, and \eqref{BV_2} imply that, for each $\delta\in(0,1)$, for $C=C(d,s)>0$,
\begin{equation}\label{BV_3}\begin{aligned} \int_{U\times U}\frac{\psi(x)-\psi(y)}{\abs{x-y}^{d+s}+\delta}\dx\dy= & \lim_{n\rightarrow\infty}\int_{U\times U}\frac{\psi_n(x)-\psi_n(y)}{\abs{x-y}^{d+s}+\delta}\dx\dy \\ & \leq \lim_{n\rightarrow\infty}\int_{U\times U}\frac{\psi_n(x)-\psi_n(y)}{\abs{x-y}^{d+s}}\dx\dy \\ & \leq \lim_{n\rightarrow\infty}C\left(\norm{\psi_n}_{L^1(U)}+\norm{\nabla\psi_n}_{L^1(U)}\right) \\ & \quad =C\left(\norm{\psi}_{L^1(U)}+\abs{\nabla\psi}(U)\right)= C\norm{\psi}_{\BV(U)}.\end{aligned}\end{equation}
Hence, after passing to the limit $\delta\rightarrow 0$ in \eqref{BV_3}, by Fatou's lemma, for $C=C(d,s)>0$,
\begin{equation}\label{BV_4}  \int_{U\times U}\frac{\psi(x)-\psi(y)}{\abs{x-y}^{d+s}}\dx\dy\leq C\norm{\psi}_{\BV(U)}.\end{equation}
Since by definition $\norm{\psi}_{L^1(U)}\leq\norm{\psi}_{\BV(U)}$, it follows from \eqref{BV_4} that, for $C=C(d,s)>0$,
$$\norm{\psi}_{W^{s,1}(U)} \leq C\norm{\psi}_{\BV(U)}.$$
This completes the argument. \end{proof}

We will use Proposition \ref{BVtoKin} to understand, for each $x\in\mathbb{T}^d$, the regularity of the map $\xi\in\mathbb{R}\mapsto\chi(x,\xi)$.  Note that this regularity does not rely upon any properties of a pathwise kinetic solution except its integrability.

\begin{cor}\label{BVtoKin}  Let $u:\mathbb{T}^d\rightarrow\mathbb{R}$ be measurable, and let $\chi$ denote the kinetic function of $u$.  Then, for each $s\in(0,1)$, for $C=C(d,s)>0$,
$$\norm{\chi}_{L^1_x\left(\mathbb{T}^d;W^{s,1}_\xi(\mathbb{R})\right)}\leq C\left(1+\norm{u}_{L^1(\mathbb{T}^d)}\right).$$
\end{cor}

\begin{proof}  Let $u:\mathbb{T}^d\rightarrow\mathbb{R}$ be an arbitrary measureable function, and let $\chi$ denote the kinetic function of $u$.  Let $s\in(0,1)$ be arbitrary.  From the definition of the kinetic function \eqref{kinetic_function_foundation}, it is immediate that, for each $x\in\mathbb{T}^d$,
$$\norm{\chi(x,\cdot)}_{\BV_\xi(\mathbb{R})}\leq 2+\abs{u(x)}.$$
The claim now follows from Proposition~\ref{BVto}.  \end{proof}

We obtain the spatial regularity of a kinetic function $\chi$ associated to a pathwise kinetic solution $u$ with initial data $u_0\in L^2_+(\mathbb{T}^d)$.  The higher integrability of the initial data implies with Proposition~\ref{stable_estimates} that the corresponding parabolic defect measure $q$ is globally integrable in velocity, locally in time.   Precisely, for each $T>0$, for $C=C(T)>0$,
$$\int_0^T\int_{\mathbb{T}^d}\abs{\nabla u^{\left[\frac{m+1}{2}\right]}}^2(x,t)\dx\dt=\frac{(m+1)^2}{4m}\int_0^T\int_\mathbb{R}\int_{\mathbb{T}^d}q(x,\xi,t)\dx\dxi\dt\leq C<\infty.$$
The following two propositions prove that any function $u\in L^1(\mathbb{T}^d)$ satisfying the estimate
$$\int_{\mathbb{T}^d}\abs{\nabla u^{\left[\frac{m+1}{2}\right]}}^2\dx<\infty,$$
is in the fractional Sobolev space $W^{s,m+1}(\mathbb{T}^d)$, for any $s\in(0,\frac{2}{m+1})$, when $m\in(0,\infty)$, and is in the Sobolev space $W^{1,1}(\mathbb{T}^d)$ when $m\in(0,1]$.  In fact, in the case $m\in(0,1]$, an application of H\"older's inequality and Lemma~\ref{lem_interpolate} imply that the solution is actually in $W^{1,\frac{2}{2-m}}(\mathbb{T}^d)$, but since this fact will not be used the details are omitted.  The first of these propositions is a small modification of the results of Ebmeyer \cite{Ebmeyer}.

\begin{prop}\label{gradtofrac}  Suppose that $m\in(1,\infty)$.  Let $u\in L^1(\mathbb{T}^d)$, and suppose that
\begin{equation}\label{gtf}\int_{\mathbb{T}^d}\abs{\nabla u^{\left[\frac{m+1}{2}\right]}}^2\dx<\infty.\end{equation}
Then, for each $s\in\left(0,\frac{2}{m+1}\right)$, there exists $C=C(m,d,s)>0$ such that
$$\norm{u}^{m+1}_{W^{s,m+1}(\mathbb{T}^d)}\leq C\left(\norm{u}^{m+1}_{L^1(\mathbb{T}^d)}+\norm{\nabla u^{\left[\frac{m+1}{2}\right]}}^2_{L^2(\mathbb{T}^d;\mathbb{R}^d)}\right).$$
\end{prop}

\begin{proof}  Let $u\in L^1(\mathbb{T}^d)$ satisfying \eqref{gtf}, $m\in(1,\infty)$, and $s\in(0,\frac{2}{m+1})$ be arbitrary.  It is first necessary to estimate the $L^{m+1}$-norm of $u$.  Lemma~\ref{lem_interpolate} implies that, for $C=C(m,d)>0$,
\begin{equation}\label{gtf_12}\norm{u}^{m+1}_{L^{m+1}(\mathbb{T}^d)}\leq C\left(\norm{u}_{L^1(\mathbb{T}^d)}^{m+1}+\norm{\nabla u^{\left[\frac{m+1}{2}\right]}}^2_{L^2(\mathbb{T}^d;\mathbb{R}^d)}\right).\end{equation}

It remains necessary to estimate the fractional Sobolev norm.  The estimate will rely on the elementary inequality, for $C=C(m)>0$,
{\color{black}\begin{equation}\label{gtf_0}\abs{r-s}^{m+1}\leq C\abs{r^{\left[\frac{m+1}{2}\right]}-s^{\left[\frac{m+1}{2}\right]}}^2,\end{equation}}
which relies upon the assumption $m\in(1,\infty)$ and can be proven, for instance, by a Taylor expansion.  Form the decomposition
\begin{equation}\label{gtf_1}\begin{aligned} \int_{\mathbb{R}^{2d}}\frac{\abs{u(x)-u(x')}^{m+1}}{\abs{x-x'}^{d+s(m+1)}}\dx\dxp = &\int_{\{\abs{x-x'}\leq1\}}\frac{\abs{u(x)-u(x')}^{m+1}}{\abs{x-x'}^{d+s(m+1)}}\dx\dxp \\ & +  \int_{\{\abs{x-x'}>1\}}\frac{\abs{u(x)-u(x')}^{m+1}}{\abs{x-x'}^{d+s(m+1)}}\dx\dxp. \end{aligned}\end{equation}
The second term of (\ref{gtf_1}) satisfies, for $C=C(m)>0$,
\begin{equation}\label{gtf_2} \int_{\{\abs{x-x'}>1\}}\frac{\abs{u(x)-u(x')}^{m+1}}{\abs{x-x'}^{d+s(m+1)}}\dx\dxp\leq C\norm{u}^{m+1}_{L^{m+1}(\mathbb{R}^d)}.\end{equation}
For the first term of \eqref{gtf_1}, in view of inequality \eqref{gtf_0}, for $C=C(m)>0$,
\begin{equation}\label{gtf_3}\begin{aligned}\int_{\{\abs{x-x'}\leq1\}}\frac{\abs{u(x)-u(x')}^{m+1}}{\abs{x-x'}^{d+s(m+1)}}\dx\dxp \leq & C \int_{\{\abs{x-x'}\leq1\}}\frac{\abs{u^{\left[\frac{m+1}{2}\right]}(x)-u^{\left[\frac{m+1}{2}\right]}(x')}^2}{\abs{x-x'}^{d+s(m+1)}}\dx\dxp \\  \leq & C\int_{B_1}\abs{x}^{-(d+s(m+1)-2)}\dx\int_{\mathbb{T}^d}\abs{\nabla u^{\left[\frac{m+1}{2}\right]}(x)}^2\dx.\end{aligned}\end{equation}
The choice $s\in(0,\frac{2}{m+1})$ guarantees that, for $C=C(d,s)>0$,
$$\int_{B_1}\abs{x}^{-(d+s(m+1)-2)}\dx\leq C<\infty.$$
Therefore, after combining \eqref{gtf_1}, \eqref{gtf_2}, and \eqref{gtf_3}, for $C=C(m,d,s)>0$,
\begin{equation}\label{gtf_4} \int_{\mathbb{R}^{2d}}\frac{\abs{u(x)-u(x')}^{m+1}}{\abs{x-x'}^{d+s(m+1)}}\dx\dxp\leq C\left(\norm{u}^{m+1}_{L^{m+1}(\mathbb{T}^d)}+\norm{\nabla u^{\left[\frac{m+1}{2}\right]}}^2_{L^2(\mathbb{T}^d;\mathbb{R}^d)}\right).\end{equation}
The claim now follows from \eqref{gtf_12} and \eqref{gtf_4}.  \end{proof}

The second proposition establishes the the Sobolev regularity for diffusion exponents $m\in(0,1]$.  The regularity is established in $W^{1,1}(\mathbb{T}^d)$, although a small modification of this argument and Lemma~\ref{lem_interpolate} readily prove that the solutions are in the stronger space $W^{1,\frac{2}{2-m}}(\mathbb{T}^d)$.  The proof is essentially a consequence of H\"older's inequality.

\begin{prop}\label{fast_est}  Suppose that $m\in(0,1]$.  Let $u\in L^1(\mathbb{T}^d)$, and suppose that
\begin{equation}\label{fast_est_1}\int_{\mathbb{T}^d}\abs{\nabla u^{\left[\frac{m+1}{2}\right]}}^2\dx<\infty.\end{equation}
Then, for $C=C(m)>0$,
{\color{black}$$\norm{u}_{W^{1,1}(\mathbb{T}^d)}\leq C\left(\norm{u}_{L^1(\mathbb{T}^d)}+\norm{u}^{2(1-m)}_{L^1(\mathbb{T}^d)}+\norm{\nabla u^{\left[\frac{m+1}{2}\right]}}^2_{L^2(\mathbb{T}^d)}\right).$$}
\end{prop}

\begin{proof}  Let $u\in \C^\infty(\mathbb{T}^d)$ satisfying \eqref{fast_est_1} and $m\in(0,\infty)$ be arbitrary.  It is only necessary to estimate the $L^1$-norm of the gradient.  First, observe the equality
$$\nabla u=\abs{u}^{\frac{1-m}{2}}\abs{u}^{\frac{m-1}{2}}\nabla u.$$
H\"older's inequality and $m\in(0,1]$ imply that, for $C=C(m)>0$,
$$\norm{\nabla u}_{L^1(\mathbb{T}^d)}\leq C\norm{\abs{u}^{\frac{1-m}{2}}}_{L^2(\mathbb{T}^d)}\norm{\nabla u^{\left[\frac{m+1}{2}\right]}}_{L^2(\mathbb{T}^d)}\leq C\norm{u}^{1-m}_{L^1(\mathbb{T}^d)}\norm{\nabla u^{\left[\frac{m+1}{2}\right]}}_{L^2(\mathbb{T}^d)}.$$
Therefore, it follows from Young's inequality that, for $C=C(m)>0$,
{\color{black}$$\norm{u}_{W^{1,1}(\mathbb{T}^d)}\leq C\left(\norm{u}_{L^1(\mathbb{T}^d)}+\norm{u}^{2(1-m)}_{L^1(\mathbb{T}^d)}+\norm{\nabla u^{\left[\frac{m+1}{2}\right]}}^2_{L^2(\mathbb{T}^d)}\right),$$}
from which the argument follows using the density of smooth functions in $L^1(\mathbb{T}^d)$.  \end{proof}

The following corollary proves that the kinetic function of a function $u\in L^1(\mathbb{T}^d)$ {\color{black}satisfying (\ref{gtf})} is locally in $W^{s,1}(\mathbb{R}^d)$, for $s\in(0,\frac{2}{m+1}\wedge 1)$, after integration in the velocity variable.  The proof essentially amounts to showing the standard fact that, for each $\delta\in(0,1-s)$, whenever $p\leq q\in[1,\infty)$, the fractional space $W^{s,p}$ embeds locally into $W^{s+\delta,q}$.

\begin{cor}\label{gtfk_cor} Let $u\in L^1(\mathbb{T}^d)$, and suppose that
\begin{equation}\label{gtfk}\int_{\mathbb{T}^d}\abs{\nabla u^{\left[\frac{m+1}{2}\right]}}^2\dx<\infty.\end{equation}
Then, if $m\in(1,\infty)$, for each $s\in\left(0,\frac{2}{m+1}\right)$, the corresponding kinetic function $\chi$ satisfies, for  $C=C(m,d,s)>0$,
$$\norm{\chi}_{L^1_\xi\left(\mathbb{R};W^{s,1}_x(\mathbb{T}^d)\right)}\leq C\left(\norm{u}_{L^1(\mathbb{T}^d)}+\norm{\nabla u^{\left[\frac{m+1}{2}\right]}}^{\frac{2}{m+1}}_{L^2(\mathbb{T}^d;\mathbb{R}^d)}\right).$$
If $m\in(0,1]$, for each $s\in(0,1)$, the corresponding kinetic function satisfies, for $C=C(m,s)>0$,
$$\norm{\chi}_{L^1_\xi\left(\mathbb{R};W^{s,1}_x(\mathbb{T}^d)\right)}\leq C\left(\norm{u}_{L^1(\mathbb{T}^d)}+\norm{u}^{2(1-m)}_{L^1(\mathbb{T}^d)}+\norm{\nabla u^{\left[\frac{m+1}{2}\right]}}^2_{L^2(\mathbb{T}^d)}\right).$$
\end{cor}

\begin{proof} Let $u\in L^1(\mathbb{T}^d)$ satisfying \eqref{gtfk} be arbitrary, and let $\chi$ denote the corresponding kinetic function.  First, we consider arbitrary $m\in(1,\infty)$ and $s\in(0,\frac{2}{m+1})$.  It follows by definition of the kinetic function \eqref{kinetic_function_foundation} that
\begin{equation}\label{gtfk_0} \norm{\chi}_{L^1_\xi\left(\mathbb{R};L^1(\mathbb{T}^d)\right)}=\norm{u}_{L^1(\mathbb{T}^d)}.\end{equation}
For the fractional Sobolev semi-norm, the definition of the kinetic function implies that
\begin{equation}\label{gtfk_1}\int_{\mathbb{R}}\int_{\mathbb{T}^{2d}}\frac{\abs{\chi(x,\xi)-\chi(x',\xi)}}{\abs{x-x'}^{d+s}}\dx\dxp\dxi=\int_{\mathbb{T}^{2d}}\frac{\abs{u(x)-u(x')}}{\abs{x-x'}^{d+s}}\dx\dxp.\end{equation}
Then, fix $\delta=\delta(m,s)\in(0,\frac{2}{m+1}-s)$.  It follows from \eqref{gtfk_1} that
$$\int_{\mathbb{R}}\int_{\mathbb{T}^{2d}}\frac{\abs{\chi(x,\xi)-\chi(x',\xi)}}{\abs{x-x'}^{d+s}}= \int_{\mathbb{T}^{2d}}\left(\frac{\abs{u(x)-u(x')}^{m+1}}{\abs{x-x'}^{d+(s+\delta)(m+1)}}\right)^{\frac{1}{m+1}}\abs{x-x'}^{-\left(\frac{dm}{m+1}-\delta\right)}.$$
Therefore, following an application of H\"older's inequality,
\begin{equation}\label{gtfk_3}\int_{\mathbb{R}}\int_{\mathbb{T}^{2d}}\frac{\abs{\chi(x,\xi)-\chi(x',\xi)}}{\abs{x-x'}^{d+s}}\leq \norm{u}_{W^{s+\delta,m+1}(\mathbb{T}^d)}\left(\int_{\mathbb{T}^{2d}}\abs{x-x'}^{-d+\frac{\delta(m+1)}{m}}\right)^{\frac{m}{m+1}}.\end{equation}
Since, for $C=C(m,d,s)>0$,
$$\int_{\mathbb{T}^{2d}}\abs{x-x'}^{-d+\frac{\delta(m+1)}{m}}\dx\dxp\leq C<\infty,$$
it follows from (\ref{gtfk_0}) and (\ref{gtfk_3}) that, for $C=C(m,d,s)>0$,
\begin{equation}\label{gtfk_4}\norm{\chi}_{L^1_\xi\left(\mathbb{R};W^{s,1}(\mathbb{T}^d)\right)}\leq C\left(\norm{u}_{L^1(\mathbb{T}^d)}+\norm{u}_{W^{s+\delta,m+1}(\mathbb{T}^d)}\right).\end{equation}
Finally, since $s+\delta\in(0,\frac{2}{m+1})$, Proposition~\ref{gradtofrac} and \eqref{gtfk_4} imply that, for $C=C(m,d,s)>0$,
\begin{equation}\label{gtfk_5}\norm{\chi}_{L^1_\xi\left(\mathbb{R};W^{s,1}(\mathbb{T}^d)\right)}\leq C\left(\norm{u}_{L^1(\mathbb{T}^d)}+\norm{\nabla u^{\left[\frac{m+1}{2}\right]}}^{\frac{2}{m+1}}_{L^2(\mathbb{T}^d;\mathbb{R}^d)}\right).\end{equation}

It remains to consider the case of arbitrary $m\in(0,1]$ and $s\in(0,1)$.  In this case, it follows from \eqref{gtfk_1} that, for $C=C(s)>0$,
\begin{equation}\label{gtfk_6}\int_{\mathbb{R}}\int_{\mathbb{T}^{2d}}\frac{\abs{\chi(x,\xi)-\chi(x',\xi)}}{\abs{x-x'}^{d+s}}\dx\dxp\dxi\leq \int_{\mathbb{T}^{2d}}\frac{\abs{u(x)-u(x')}}{\abs{x-x'}^{d+s}}\dx\dxp\leq \norm{u}_{W^{1,s}(\mathbb{T}^d)}\leq C\norm{u}_{W^{1,1}(\mathbb{R}^d)}.\end{equation}
Therefore, from the definition and Proposition~\ref{fast_est}, for $C=C(m,s)>0$,
{\color{black}$$\norm{u}_{W^{s,1}(\mathbb{T}^d)}\leq C\left(\norm{u}_{L^1(\mathbb{T}^d)}+\norm{u}^{2(1-m)}_{L^1(\mathbb{T}^d)}+\norm{\nabla u^{\left[\frac{m+1}{2}\right]}}^2_{L^2(\mathbb{T}^d)}\right).$$}
Together with \eqref{gtfk_5}, this completes the argument.  \end{proof}

We will now combine Corollary~\ref{BVtoKin} and Corollary~\ref{gtfk_cor} in order to obtain the regularity of the kinetic function jointly in the spatial and velocity variables.  We will apply the following proposition to the case $U_1=\mathbb{T}^d$, $U_2=\mathbb{R}$, and $s_1=s_2\in(0,\frac{2}{m+1}\wedge 1)$.

\begin{prop}\label{meld} Let  $n_1,n_2\geq 1$ and $p\in[1,\infty)$.  Let $U_1\subset\mathbb{R}^{n_1}$ and $U_2\subset\mathbb{R}^{n_2}$ be open subsets.  Suppose that $u:U_1\times U_2\rightarrow\mathbb{R}$ satisfies, for $s_1,s_2\in(0,1)$,
\begin{equation}\label{meld_0}\norm{u}_{L^p_x\left(U_1;W^{s_2,p}_y(U_2)\right)}+\norm{u}_{L^p_y\left(U_2;W^{s_1,p}_x(U_1)\right)}<\infty.\end{equation}
Then, for $s=\min\{s_1,s_2\}$, for $C=C(n_1,n_2,p)>0$,
$$\norm{u}_{W^{s,p}(U_1\times U_2)}\leq C\left(\norm{u}_{L^p_x\left(U_1;W^{s_2,p}_y(U_2)\right)}+\norm{u}_{L^p_y\left(U_2;W^{s_1,p}_x(U_1)\right)}\right).$$
\end{prop}

\begin{proof}  Fix positive integers $n_1,n_2\geq 1$, open subsets $U_1\subset\mathbb{R}^{n_1}$ and $U_2\subset\mathbb{R}^{n_2}$, fractional Sobolev exponents $s_1,s_2\in(0,1)$ and a function $u:U_1\times U_2\rightarrow\mathbb{R}$ satisfying (\ref{meld_0}).  It is immediate from the definition that
\begin{equation}\label{meld_1}\norm{u}_{L^p(U_1\times U_2)}\leq \min\left\{\norm{u}_{L^p_x\left(U_1;W^{s_2,p}_y(U_2)\right)},\norm{u}_{L^p_y\left(U_2;W^{s_1,p}_x(U_1)\right)}\right\}.\end{equation}
It remains only to estimate the fractional Sobolev semi-norm.

In the argument to follow, we will denote points $x,x'\in\mathbb{R}^{n_1}$ and $y,y'\in\mathbb{R}^{n_2}$.  Then, for $s=\min\{s_1,s_2\}\in(0,1)$, for $C=C(p)>0$,
\begin{equation}\label{meld_2}\begin{aligned} &\int_{\left(U_1\times U_2\right)^2}\frac{\abs{u(x,y)-u(x',y')}^p}{\left(\abs{x-x'}+\abs{y-y'}\right)^{n_1+n_2+sp}}\dx\dxp\dy\dyp \\  &\leq C\int_{\{\abs{x-x'}+\abs{y-y'}\leq 1\}\cap\left(U_1\times U_2\right)^2}\frac{\abs{u(x,y)-u(x',y')}^p}{\left(\abs{x-x'}+\abs{y-y'}\right)^{n_1+n_2+sp}}\dx\dxp\dy\dyp \\ &\quad +C\int_{\{\abs{x-x'}+\abs{y-y'}>1\}\cap\left(U_1\times U_2\right)^2}\frac{\abs{u(x,y)-u(x',y')}^p}{\left(\abs{x-x'}+\abs{y-y'}\right)^{n_1+n_2+sp}}\dx\dxp\dy\dyp. \end{aligned} \end{equation}
For the first term of (\ref{meld_2}), in view of (\ref{meld_1}),
\begin{equation}\label{meld_3}\begin{aligned} &\int_{\{\abs{x-x'}+\abs{y-y'}>1\}\cap\left(U_1\times U_2\right)^2}\frac{\abs{u(x,y)-u(x',y')}^p}{\left(\abs{x-x'}+\abs{y-y'}\right)^{n_1+n_2+sp}}\dx\dxp\dy\dyp\\ &\leq 2\min\left\{\norm{u}^p_{L^p_x\left(U_1;W^{s_2,p}_y(U_2)\right)},\norm{u}^p_{L^p_y\left(U_2;W^{s_1,p}_x(U_1)\right)}\right\}.\end{aligned}\end{equation}
The second term of (\ref{meld_2}) is decomposed using the triangle inequality to obtain, for $C=C(p)>0$,
\begin{equation}\begin{aligned}\label{meld_4} &\int_{\{\abs{x-x'}+\abs{y-y'}\leq1\}\cap\left(U_1\times U_2\right)^2}\frac{\abs{u(x,y)+u(x',y')}^p}{\left(\abs{x-x'}+\abs{y-y'}\right)^{n_1+n_2+sp}}\dx\dxp\dy\dyp \\ &\leq C\int_{\{\abs{x-x'}+\abs{y-y'}\leq1\}\cap\left(U_1\times U_2\right)^2}\frac{\abs{u(x,y)-u(x',y)}^p}{\left(\abs{x-x'}+\abs{y-y'}\right)^{n_1+n_2+sp}}\dx\dxp\dy\dyp \\  &\quad +C\int_{\{\abs{x-x'}+\abs{y-y'}\leq1\}\cap\left(U_1\times U_2\right)^2}\frac{\abs{u(x',y)-u(x',y')}^p}{\left(\abs{x-x'}+\abs{y-y'}\right)^{n_1+n_2+sp}}\dx\dxp\dy\dyp. \end{aligned}\end{equation}

For the first term on the righthand side of (\ref{meld_4}), for $C=C(n_2)>0$, since $s_1\geq s$,
\begin{equation}\begin{aligned}\label{meld_5} &\int_{\{\abs{x-x'}+\abs{y-y'}\leq1\}\cap\left(U_1\times U_2\right)^2}\frac{\abs{u(x,y)-u(x',y)}^p}{\left(\abs{x-x'}+\abs{y-y'}\right)^{n_1+n_2+sp}}\dx\dxp\dy\dyp \\ &\leq C\int_{U_2}\int_0^1 \int_{\left(U_1\right)^2}\frac{\abs{u(x,y)-u(x',y)}^p}{\left(\abs{x-x'}+r\right)^{n_1+n_2+sp}}r^{n_2-1}\dx\dxp\dr\dy \\ &\leq C\int_{U_2} \int_0^1\int_{\left(U_1\right)^2}\frac{\abs{u(x,y)-u(x',y)}^p}{\left(\abs{x-x'}+r\right)^{n_1+1+sp}}\dx\dxp\dr\dy \\ & \leq C\int_{U_2}\int_{\left(U_1\right)^2}\frac{\abs{u(x,y)-u(x',y)}^p}{\abs{x-x'}^{n_1+sp}}\dx\dxp\dy \\ &\leq C\norm{u}^p_{L^p_y\left(U_2;W^{s,p}_x(U_1)\right)}\leq C\norm{u}^p_{L^p_y\left(U_2;W^{s_1,p}_x(U_1)\right)} .\end{aligned} \end{equation}
For the second term on the righthand side of (\ref{meld_4}), since $s_2\geq s$, the analogous computation proves that, for $C=C(n_1)>0$,
\begin{equation}\label{meld_6} \int_{\{\abs{x-x'}+\abs{y-y'}\leq1\}\cap\left(U_1\times U_2\right)^2}\frac{\abs{u(x',y)-u(x',y')}^p}{\left(\abs{x-x'}+\abs{y-y'}\right)^{n_1+n_2+sp}}\dx\dxp\dy\dyp\leq C\norm{u}^p_{L^p_x\left(U_1;W^{s_2,p}_y(U_2)\right)}.\end{equation}

In combination, estimates \eqref{meld_1}, \eqref{meld_3}, \eqref{meld_5}, and \eqref{meld_6} combined with \eqref{meld_2} and \eqref{meld_4} prove that, for $C=C(n_1,n_2,p)>0$,
$$\norm{u}_{W^{s,p}(U_1\times U_2)}\leq C\left(\norm{u}_{L^p_x\left(U_1;W^{s_2,p}_y(U_2)\right)}+\norm{u}_{L^p_y\left(U_2;W^{s_1,p}_x(U_1)\right)}\right).$$
This completes the argument.  \end{proof}

We now apply Proposition~\ref{meld} to the kinetic function corresponding to a function $u\in L^1(\mathbb{T}^d)$ satisfying \eqref{gtf_0}.  The estimates are obtained from Corollary~\ref{BVtoKin} and Corollary~\ref{gtfk_cor}.

\begin{cor}\label{kreg}Let $u\in L^1(\mathbb{T}^d)$, and suppose that
\begin{equation}\label{kreg_00}\int_{\mathbb{R}^d}\abs{\nabla u^{\left[\frac{m+1}{2}\right]}}^2\dx<\infty.\end{equation}
Then, if $m\in(1,\infty)$, for each $s\in\left(0,\frac{2}{m+1}\right)$, the corresponding kinetic function $\chi$ satisfies, for $C=C(m,d,s)>0$,
$$\norm{\chi}_{W^{s,1}_{x,\xi}(\mathbb{T}^d\times\mathbb{R})}\leq C\left(1+\norm{u}_{L^1(\mathbb{T}^d)}+\norm{\nabla u^{\left[\frac{m+1}{2}\right]}}^{\frac{2}{m+1}}_{L^2(\mathbb{R}^d;\mathbb{R}^d)}\right).$$
If $m\in(0,1]$, for each $s\in(0,1)$, the corresponding kinetic function $\chi$ satisfies, for $C=C(m,s)>0$,
$$\norm{\chi}_{L^1_\xi\left(\mathbb{R};W^{s,1}_x(\mathbb{T}^d)\right)}\leq C\left(1+\norm{u}_{L^1(\mathbb{T}^d)}+\norm{u}^{2(1-m)}_{L^1(\mathbb{T}^d)}+\norm{\nabla u^{\left[\frac{m+1}{2}\right]}}^2_{L^2(\mathbb{T}^d)}\right).$$
\end{cor}

\begin{proof}  Let $u\in L^1(\mathbb{T}^d)$ satisfying \eqref{kreg_00} be arbitrary, and let $\chi$ denote the corresponding kinetic function.  Fix $m\in(0,\infty)$ and $s\in(0,\frac{2}{m+1}\wedge 1)$.  In the statement of Proposition~\ref{meld}, choose $n_1=d$, $n_2=1$, $U_1=\mathbb{T}^d$, $U_2=\mathbb{R}$ and $s_1=s_2=s$, which implies that, for $C=C(d)>0$,
\begin{equation}\label{kreg_1}\norm{\chi}_{W^{s,1}_{x,\xi}(\mathbb{T}^d\times\mathbb{R})}\leq C\left(\norm{\chi}_{L^1_x\left(\mathbb{T}^d;W^{s,1}_\xi(\mathbb{R})\right)}+\norm{\chi}_{L^1_\xi\left(\mathbb{R};W^{s,1}_x(\mathbb{T}^d)\right)}\right).\end{equation}
The claim is now an immediate consequence of Corollary~\ref{BVtoKin} and Corollary~\ref{gtfk_cor}.\end{proof}

The final proposition of this section proves that the transport under the characteristics system preserves the fractional Sobolev norm locally in time.  For each $(x,\xi)\in\mathbb{T}^d\times\mathbb{R}$, $t_0\geq 0$, and $\epsilon\in[0,1)$, where $\epsilon=0$ corresponds to the system \eqref{kin_rough_for}, recall the forward characteristic system
\begin{equation}\label{last}\left\{\begin{array}{ll} \dd X^{x,\xi,\epsilon}_{t_0,t}=-b\left(X^{x,\xi,\epsilon}_{t_0,t},\Xi^{x,\xi,\epsilon}_{t_0,t}\right)\circ \dd z^\epsilon_t & \textrm{in}\;\;(t_0,\infty), \\ \dd \Xi^{x,\xi,\epsilon}_{t_0,t}=c(X^{x,\xi,\epsilon}_{t_0,t},\Xi^{x,\xi,\epsilon}_{t_0,t})\circ \dd z^\epsilon_t & \textrm{in}\;\;(t_0,\infty), \\ (X^{x,\xi,\epsilon}_{t_0,t_0},\Xi^{x,\xi,\epsilon}_{t_0,t_0})=(x,\xi). & \end{array}\right.\end{equation}
The following statement is used to transfer the regularity of a kinetic function $\chi$ to the transported kinetic function
$$\tilde{\chi}(x,\xi,t):=\chi(X^{x,\xi,\epsilon}_{t_0,t},\Pi^{x,\xi,\epsilon}_{t_0,t},t)\;\;\textrm{for}\;\;(x,\xi,t)\in\mathbb{T}^d\times\mathbb{R}\times[t_0,\infty),$$
for arbitrary $\epsilon\in(0,1)$ and $t_0\geq 0$.  We first prove the statement for an arbitrary measure preserving diffeomorphism of $\mathbb{T}^d\times\mathbb{R}$.

\begin{prop}\label{last_prop}  Let $s\in(0,1)$ and $p\in[1,\infty)$.  Suppose that $T:\mathbb{T}^d\times\mathbb{R}\rightarrow\mathbb{T}^d\times\mathbb{R}$ is a measure-preserving $\C^1$-diffeomorphism with bounded gradient.  For every measurable function $\psi:\mathbb{T}^d\times\mathbb{R}\rightarrow\mathbb{R}$, define
$$\tilde{\psi}(x,\xi)=\psi(T(x,\xi))\;\;\textrm{for}\;\;(x,\xi)\in\mathbb{T}^d\times\mathbb{R}.$$
Then, for every measurable $\psi:\mathbb{T}^d\times\mathbb{R}\rightarrow\mathbb{R}$, for every open subset $U\subset\mathbb{T}^d\times\mathbb{R}$, there exists a $C=C(T)>0$ such that
$$\norm{\tilde{\psi}}_{W^{s,p}(T^{-1}(U))}\leq C\norm{\psi}_{W^{s,p}(U)}.$$
\end{prop}

\begin{proof} Fix $s\in(0,1)$ and $p\in[1,\infty)$.  Suppose that $T:\mathbb{T}^d\times\mathbb{R}\rightarrow\mathbb{T}^d\times\mathbb{R}$ is a measure-preserving $\C^1$-diffeomorphism with bounded gradient.  Let $\psi:\mathbb{T}^d\times\mathbb{R}\rightarrow\mathbb{R}$ be an arbitrary measurable function, and let $U\subset\mathbb{T}^d\times\mathbb{R}$ be an arbitrary open set.  Since $T$ preserves the measure, it is immediate that
\begin{equation}\label{last_0}\norm{\tilde{\psi}}_{L^p(T^{-1}(U))}=\norm{\psi}_{L^p(U)}.\end{equation}
The fractional Sobolev seminorm is estimated in a similar fashion.  It follows again from the fact that $T$ preserves the measure that
\begin{equation}\begin{aligned}\label{last_1} & \int_{T^{-1}(U)\times T^{-1}(U)}\frac{\abs{\tilde{\psi}(x)-\tilde{\psi}(x')}^p}{\abs{x-x'}^{(d+1)+sp}}\dx\dxip =\int_{U\times U}\frac{\abs{\psi(x)-\psi(x')}^p}{\abs{T^{-1}(x)-T^{-1}(x')}^{(d+1)+sp}}\dx\dxp.\end{aligned}\end{equation}
Similar to estimate \eqref{u_9}, since, for each $x,x'\in\mathbb{T}^d\times\mathbb{R}$,
$$\abs{x-x'}=\abs{T(T^{-1}(x))-T(T^{-1}(x'))}\leq \norm{\nabla T}_{L^\infty\left(\mathbb{T}^d\times\mathbb{R};\mathcal{M}^{(d+1)\times(d+1)}\right)}\abs{T^{-1}(x)-T^{-1}(x')},$$
there exists $C=C(T)>0$ for which, for each $x,x'\in\mathbb{T}^d\times\mathbb{R}$,
\begin{equation}\label{last_2} \frac{\abs{x-x'}}{\abs{T^{-1}(x)-T^{-1}(x')}}\leq C.\end{equation}
In combination, equality \eqref{last_1} and inequality \eqref{last_2} imply that, for $C=C(T)>0$,
\begin{equation*}\begin{aligned} & \int_{T^{-1}(U)\times T^{-1}(U)}\frac{\abs{\tilde{\psi}(x)-\tilde{\psi}(x')}^p}{\abs{x-x'}^{(d+1)+sp}}\dx\dxip \leq C \int_{U\times U}\frac{\abs{\psi(x)-\psi(x')}^p}{\abs{x-x'}^{(d+1)+sp}}\dx\dxp.\end{aligned}\end{equation*}
The result follows from (\ref{last_0}) and (\ref{last_2}).  \end{proof}

In the final corollary of this section, we apply Proposition~\ref{last_prop} to the transport map defined by the characteristics \eqref{last}.  The proof is an immediate consequence of the fact that the characteristics preserve the Lebesgue measure \eqref{kin_measure}, the regularity assumption \eqref{prelim_regular}, and the estimates of Proposition~\ref{rough_est}.

\begin{cor}\label{last_cor}  Let $s\in(0,1)$ and $p\in[1,\infty)$.  For every $\epsilon\in[0,1)$, $t_0\geq 0$, and $t\geq t_0$, define the $\C^1$-diffeomorphism $T^{\epsilon}_{t_0,t}:\mathbb{T}^d\times\mathbb{R}\rightarrow\mathbb{T}^d\times\mathbb{R}$ to be the transport map defined by the characteristics \eqref{last}.  That is,
$$T^{\epsilon}_{t_0,t}(x,\xi)=\left(X^{x,\xi,\epsilon}_{t_0,t}, \Xi^{x,\xi,\epsilon}_{t_0,t}\right)\;\;\textrm{for}\;\;(x,\xi)\in\mathbb{T}^d\times\mathbb{R}.$$
For each open subset $U\subset\mathbb{T}^d\times\mathbb{R}$ and for each $\psi\in W^{s,p}(U)$ define, for $\epsilon\in[0,1)$, $t_0\geq 0$, and $t\geq t_0$,
$$\tilde{\psi}^\epsilon_{t_0,t}(x,\xi)=\psi(T^{\epsilon}_{t_0,t}(x,\xi))\;\;\textrm{for}\;\;(x,\xi)\in\mathbb{T}^d\times\mathbb{R}.$$
For each $\epsilon\in[0,1)$, $t_0\geq 0$ and $t\geq t_0$, there exists $C=C(\abs{t-t_0})>0$ such that
$$\norm{\tilde{\psi}^\epsilon_{t_0,t}}_{W^{s,p}\left(\left(T^{\epsilon}_{t_0,t}\right)^{-1}(U)\right)}\leq C\norm{\psi}_{W^{s,p}(U)}.$$
\end{cor}

\end{appendix}

\section*{Acknowledgements}

We would like to thank the referees for their careful reports. Their comments were of substantial benefit to the paper. 

The first author was supported by the National Science Foundation Mathematical Sciences Postdoctoral Research Fellowship under Grant Number 1502731.

The second author acknowledges financial support by the the Max Planck Society through the Max Planck Research Group``Stochastic partial differential equations'' and by the DFG through the CRC ``Taming uncertainty and profiting from randomness and low regularity in analysis, stochastics and their applications.''

\bibliography{Exit}
\bibliographystyle{plain}

\end{document}